\documentclass{amsart}

\usepackage{amssymb} \usepackage{amsfonts} \usepackage{amsmath}
\usepackage{amsthm} \usepackage{epsfig} 
\usepackage{color}
\usepackage{amscd}
\usepackage[all]{xy}
\usepackage{graphicx}
\usepackage{pinlabel}

\usepackage{pgf,tikz}
\usetikzlibrary{arrows}

\newtheorem{lemma}{Lemma}[section]
\newtheorem{teo}[lemma]{Theorem}
\newtheorem{prop}[lemma]{Proposition}
\newtheorem{cor}[lemma]{Corollary}

\theoremstyle{definition}
\newtheorem{defn}[lemma]{Definition}

\newtheorem{example}[lemma]{Example}

\theoremstyle{remark}
\newtheorem{rem}[lemma]{Remark} 

\newcommand{\matr} [4] {\big({\tiny\begin{array}{@{}c@{\ }c@{}} #1 & #2 \\ #3 & #4 \\ \end{array}} \big)}

\newcommand{\interior}[1]{{\rm int}(#1)}

\newcommand{\Iso}{{\rm Isom}}

\newcommand{\p}[1]{\ensuremath{\boldsymbol{#1^+} }}
\newcommand{\m}[1]{\ensuremath{\boldsymbol{#1^-} } }
\renewcommand{\l}[1]{\ensuremath{\boldsymbol{#1}} }

\newcommand{\G}{\ensuremath{\boldsymbol{G} }}

\newcommand{\teta}{\includegraphics[width = .4 cm]{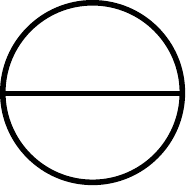}}
\newcommand{\tetra}{\includegraphics[width = .4 cm]{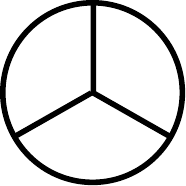}}

\newcommand{\matR} {\ensuremath {\mathbb{R}}}
\newcommand{\matQ} {\ensuremath {\mathbb{Q}}}

\newcommand{\matH} {\ensuremath {\mathbb{H}}}

\newcommand{\calF} {\ensuremath {\mathcal{F}}}

\newcommand{\Vol} {\ensuremath {{\rm Vol}}}
\newcommand{\Area} {\ensuremath {{\rm Area}}}

\newcommand{\nota} [1] {\caption{\footnotesize{#1}}}

\author{Bruno Martelli}
\address{Dipartimento di Matematica, Largo Pontecorvo 5, 56127 Pisa, Italy}
\email{martelli at dm dot unipi dot it}

\author{Stefano Riolo}
\address{Dipartimento di Matematica, Largo Pontecorvo 5, 56127 Pisa, Italy}
\email{riolo at mail dot dm dot unipi dot it}

\title{Hyperbolic Dehn filling in dimension four}

\begin{document}

\begin{abstract}
We introduce and study some deformations of complete finite-volume hyperbolic four-manifolds that may be interpreted as four-dimensional analogues of Thurston's hyperbolic Dehn filling.

We construct in particular an analytic path of complete, finite-volume cone four-manifolds $M_t$ that interpolates between two hyperbolic four-manifolds $M_0$ and $M_1$ with the same volume $\frac {8}3\pi^2$. The deformation looks like the familiar hyperbolic Dehn filling paths that occur in dimension three, where the cone angle of a core simple closed geodesic varies monotonically from $0$ to $2\pi$. Here, the singularity of $M_t$ is an immersed geodesic surface whose cone angles also vary monotonically from $0$ to $2\pi$. When a cone angle tends to $0$ a small core surface (a torus or Klein bottle) is drilled producing a new cusp.

We show that various instances of hyperbolic Dehn fillings may arise, including one case where a degeneration occurs when the cone angles tend to $2\pi$, like in the famous figure-eight knot complement example. 

The construction makes an essential use of a family of four-dimensional deforming hyperbolic polytopes recently discovered by Kerckhoff and Storm.
\end{abstract}

\maketitle
\section{Introduction}
By Mostow-Prasad rigidity \cite{Mos, Pra}, complete finite-volume hyperbolic manifolds can be deformed only in dimension two. Some deformations may arise also in higher dimension if one accepts to work in the more general setting of \emph{hyperbolic cone-manifolds}: the celebrated Thurston Hyperbolic Dehn filling theorem states that every cusped hyperbolic three-manifold may be deformed to a hyperbolic cone-manifold, whose singular locus consists of small simple closed geodesics with small cone angles. 
As the deformation goes on, both the geodesic length and the cone angle increase: if the cone angle reaches $2\pi$ we get a genuine hyperbolic manifold without singularities. 

The aim of this paper is to show that this phenomenon occurs sometimes also in dimension four. We prove this by constructing some examples explicitly.

\subsection*{Hyperbolic cone-manifolds}
Hyperbolic cone-manifolds were defined in every dimension by Thurston \cite{Th}, see also \cite{BLP, CHK, McM}. Hyperbolic cone surfaces and three-manifolds are widely studied, see for instance \cite{Br, HK, Ko, MM, W1, W2}.
The singular locus in an orientable hyperbolic cone three-manifold consists of closed geodesics or more complicated graphs. Not much seems to be known in dimension four or higher.

We construct here some hyperbolic cone four-manifolds $M$ whose singular locus $\Sigma$ is the image of a (possibly disconnected) geodesically immersed hyperbolic cone-surface $i\colon \tilde\Sigma \looparrowright M$ that self-intersects orthogonally at its cone points. This seems a natural kind of hyperbolic cone four-manifold to study: see Section \ref{cone:intro:subsection} for a precise definition. The image of every connected component of $\tilde\Sigma$ has some cone angle 
in $M$, and at each double point $p\in \Sigma$ two components of $\tilde\Sigma$ with (possibly different) cone angles $\alpha$ and $\beta$ meet orthogonally.  Note that every component of $\tilde \Sigma$ is a hyperbolic cone-surface and as such it can also be topologically a sphere or a torus.

\subsection*{Main result}
The main result of this paper is Theorem \ref{main:teo} below. It shows a number of new phenomena. First, it shows that complete finite-volume hyperbolic cone four-manifolds with singular locus a geodesically immersed surface exist. 
Then, it shows that these cone manifolds can sometimes be deformed, via a deformation that varies the cone angles of the strata, like in dimensions two and three. Finally, it displays an example where the deformation can be carried in both directions until a torus or Klein bottle is drilled, interpolating between two cusped hyperbolic four-manifolds.
Such a deformation may be interpreted as a four-dimensional hyperbolic Dehn filling (at both endpoints of the deformation path).

\begin{teo} \label{main:teo}
There is a compact smooth non-orientable four-manifold $M$ with $\partial M$ diffeomorphic to a three-torus, which contains a smooth two-torus and a smooth Klein bottle $T, K \subset \interior M$, 
both with trivial normal bundle, that intersect transversely in two points (see Figure \ref{curve:fig}), such that the following holds. 

There is an analytic path $\{M_t\}_{t\in (0,1)}$ of complete finite-volume hyperbolic cone-manifold structures on $\interior M$ with singular locus the immersed geodesic cone-surface $\Sigma = T \cup K$. The two cone-surfaces $T$ and $K$ have cone angles $0< \alpha < 2\pi$ and $0< \beta < 2\pi$ respectively. We have
$$\Area (T) = 4\pi -2\beta, \qquad \Area (K) =4\pi -2\alpha.$$
When $t$ varies from $0$ to $1$ the angle $\alpha$ goes from $0$ to $2\pi$ and $\beta$ goes from $2\pi$ to $0$. The path converges as $t\to 0$ and $t\to 1$ to two complete, finite-volume hyperbolic four-manifolds $M_0 = \interior M\setminus T$ and $M_1 = \interior M\setminus K$.
\end{teo}

\begin{figure}
\labellist
\small\hair 2pt
\pinlabel $T$ at 127 36
\pinlabel $K$ at 230 36
\pinlabel $M$ at 230 0
\endlabellist
\centering
\includegraphics[width=8 cm]{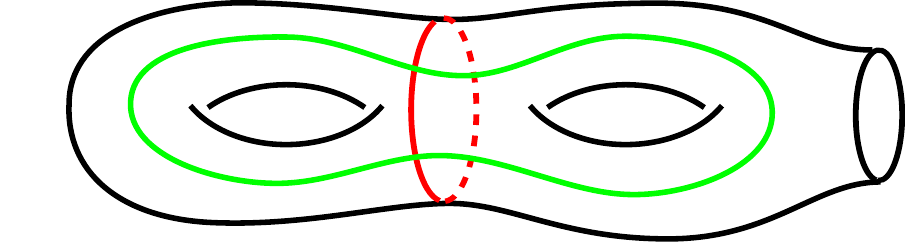}
\nota{A schematic picture of $M$ and the immersed surface $\Sigma=T\cup K$.}\label{curve:fig}
\end{figure}

The deformation interpolates analytically between two cusped hyperbolic four-manifolds $M_0$ and $M_1$. As opposite to $M_t$ with $t\in (0,1)$, the manifolds $M_0$ and $M_1$ are genuine hyperbolic manifolds, with no cone singularities. The boundary three-torus $\partial M$ gives rise to a cusp in $M_t$ for all $t\in [0,1]$ diffeomorphic to $S^1\times S^1\times S^1 \times [0,+\infty)$, whose Euclidean shape varies with $t$.
The manifolds $M_0$ and $M_1$ have also one additional cusp each, obtained by drilling $T$ or $K$ respectively, whose Euclidean section is diffeomorphic to $T\times S^1$ or $K\times S^1$. 

We recall that an important theorem of Garland and Raghunathan \cite{GR} implies that the holonomy of a complete finite-volume hyperbolic $n$-manifold cannot be perturbed when $n\geq 4$. Of course we are not
violating this theorem here, because the holonomy that is moving is that of the \emph{non-complete} hyperbolic manifold $M\setminus (T\cup K)$. When we say that the deformation varies analytically, we mean that this holonomy does.

The overall picture has some evident similarities with some familiar two- and three-dimensional deformations. The interpolation looks like an analytic path in the moduli or Teichm\"uller space of a surface connecting two points at infinity, where two intersecting simple closed curves as in Figure \ref{curve:fig} are shrunk respectively in opposite directions of the path. 

If we look at the deformation by starting at one extreme $t_0=0$ or $t_0=1$ and moving $t$ towards the other extreme $t_1=1-t_0$, we get a hyperbolic Dehn filling path as in dimension three: the topology of the manifold is modified as soon as we move away from $t_0$ by a topological Dehn filling (we close a cusp by adding a two-torus or a Klein bottle), and the metric changes by adding a small core geodesic cone-surface $S_{t_0}\in\{T,K\}$ with small cone angle. The deformation can be pursued until, at time $1-t_0$, the core geodesic cone-surface $S_{t_0}$ reaches a cone angle of $2\pi$. At the same time the other cone-surface $S_{t_1}$ disappears and the two cone-points of $S_{t_0}$ become two cusps.

The manifolds $M_0$ and $M_1$ have the same small Euler characteristic $\chi = 2$, and hence the same volume 
$$\Vol(M_0) = \Vol(M_1) = \frac{8\pi^2}3.$$ 
The volume of $M_t$ is easily expressed in terms of the cone angles $\alpha$ and $\beta$ as
$$\Vol(M_t) = \frac{8\pi^2}{3}\left(2-\frac{\alpha + \beta}{2\pi} + \frac{\alpha\beta}{4\pi^2}\right).$$
The volume of $M_t$ is shown in Figure \ref{vol_manifold:fig}.
As opposite to dimension three, in our case the volume \emph{increases} under hyperbolic Dehn filling (at both endpoints of the
deformation path).

\begin{figure}
\vspace{.5 cm}
\labellist
\small\hair 2pt
\endlabellist
\centering
\includegraphics[width=8 cm]{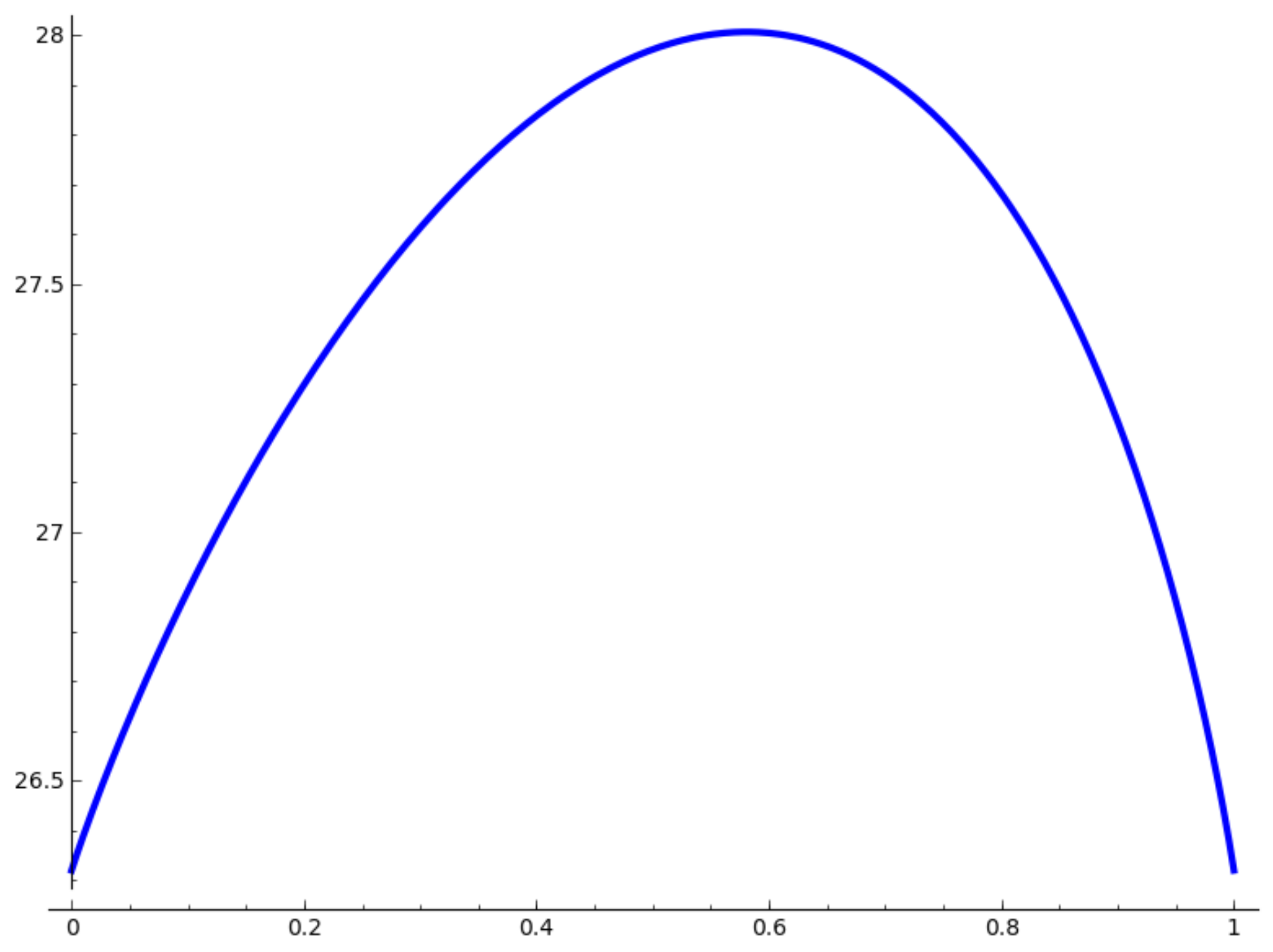}
\vspace{.3 cm}
\nota{The function $\Vol(M_t)$.}\label{vol_manifold:fig}
\end{figure}

The manifolds $M_0$ and $M_1$ are clearly not diffeomorphic; we show that they are not even commensurable: the manifold $M_0$ is commensurable with the integral lattice in O$(4,1)$, and $M_1$ appears to be at the time of writing the smallest known hyperbolic four-manifold that is not commensurable with that lattice. Both manifolds are arithmetic. 
Added: more recently, some more examples have been constructed in \cite{RiSl} using \cite{S}.)
We can in fact interpret $M_1$ as a new hyperbolic manifold constructed by deforming $M_0$. It would be interesting to understand in more generality whether one can vary the cone angles along immersed geodesic cone-surfaces in hyperbolic cone four-manifolds, as a tool to construct new hyperbolic manifolds. Some infinitesimal rigidity and existence results were obtained by Montcouquiol \cite{Mon, Mon2} for (non-singular) closed surfaces in the wider context of Einstein deformations. 

We note that the manifolds $M_t$ that we construct here are non-orientable. One may build a similar family of orientable deforming cone-manifolds by taking the orientable double cover $\tilde M_t$. The cone-surfaces $T$ and $K$ lift to three cone-tori in $\tilde M_t$, two of cone angle $\alpha$ lying above $T$ and one of cone angle $\beta$ above $K$. The manifolds $\tilde M_0$ and $\tilde M_1$ have three and two cusps respectively, all of three-torus type.

\subsection*{Sketch of the proof}
Theorem \ref{main:teo} is proved by constructing the family of hyperbolic cone-manifolds $M_t$ explicitly. 

The construction goes as follows. The fundamental ingredient is a deforming family $\calF_t\subset \matH^4$ of infinite-volume polytopes built by Kerckhoff and Storm in \cite{KS}. We truncate here $\calF_t$ via two additional hyperplanes to get a deforming family of \emph{finite}-volume polytopes $P_t\subset \matH^4$. These polytopes are quite remarkable, because they have for all times $t$ only few non-right dihedral angles. In particular, for the times $t$ that are relevant for the proof of Theorem \ref{main:teo}, the (two-dimensional) faces with non-right dihedral angles intersect pairwise only at some vertices.

The family $P_t$ interpolates between two Coxeter polytopes of the same volume: the familiar ideal right-angled 24-cell and another interesting polytope with dihedral angles $\frac \pi 2$ and $\frac \pi 3$. We then employ some mirroring and assembling techniques similar to the ones used in \cite{KM} to promote each polytope $P_t$ to a hyperbolic cone-manifold $M_t$. Since $P_t$ has few non-right dihedral angles, the manifold $M_t$ has few controlled singularities.

\subsection*{More hyperbolic Dehn fillings}
In the Dehn fillings that we have considered in Theorem \ref{main:teo}, the cusp shape is a flat three-manifold that fibers over a torus or a Klein bottle, and the filling collapses
the $S^1$ fibers. 
In the deforming cone-manifolds context, more different kinds of Dehn fillings may arise that are also interesting. For instance, one may close a cusp of type $S^1 \times S^1 \times S^1$ by collapsing
a $S^1 \times S^1$ factor: in this case we add a closed curve instead of a two-torus, and the resulting space is not a topological manifold. This kind of topological Dehn filling was considered by Fujiwara and Manning in \cite{FM, FM2}. 

Another variation occurs when the Euclidean cusp section is not a three-torus. For instance, a Euclidean cusp section of a hyperbolic cone four-manifold may be one of the following types:
$$S^2 \times S^1, \qquad S^3$$
where we see $S^n$ as the \emph{Euclidean} cone-manifold obtained by doubling the regular Euclidean $n$-simplex along its boundary. In this case one may Dehn fill this cusp by collapsing
one of the spheres $S^1, S^2,$ or $S^3$. This corresponds to adding a core $S^2$, $S^1$, or a couple of points.

We will show in this paper that all the examples of Dehn fillings mentioned in the above paragraphs arise geometrically as hyperbolic Dehn fillings of some hyperbolic cone-manifolds. It is also possible to perform a hyperbolic Dehn \emph{surgery}, the concatenation of a hyperbolic drilling and a hyperbolic filling along an analytic path, that substitutes a small geodesic $S^k$ with a small geodesic $S^{3-k}$. Topologically, this is just the usual surgery along $k$-spheres with trivial normal bundles, that is the substitution of a $S^k\times D^{4-k}$ with a $D^{k+1} \times S^{3-k}$.
See Theorem \ref{main2:teo} below.

\subsection*{Degeneration}
An important phenomenon that arises in dimension three, first described by Thurston \cite{bibbia}, is that of a hyperbolic Dehn filling that degenerates when the cone angle tends to $2\pi$ into a Seifert manifold with hyperbolic base.

We show here a similar phenomenon: a four-dimensional hyperbolic Dehn filling $W_t$ that degenerates as the cone angle tends to $2\pi$ into a product $C\times S^1$ where $C$ is a cusped hyperbolic 3-manifold. (The manifold $C$ found here is tessellated into four copies of the ideal right-angled cuboctahedron, and we call it the \emph{cuboctahedral manifold}.) In the following theorem, we think of the time $t$ running backwards from $t=1$ to $t=0$, in accordance with \cite{KS}.

\begin{teo} \label{main2:teo}
There is an analytic path $\{W_t\}_{t \in (0,1]}$ of complete finite-volume hyperbolic cone four-manifolds with cone angles $<2\pi$, with some times $1>t_1>t_2>\bar t > 0$, such that $W_1$ is a manifold, and $W_{t_1}$, $W_{\bar t}$ are orbifolds. At the critical times $1,t_1, t_2,0$ the topology of $W_t$ changes as follows:
\begin{itemize}
\item at $t=1$ by hyperbolic Dehn filling 12 three-torus cusps by adding 12 tori; 
\item at $t=t_1$ by hyperbolic Dehn surgerying 8 small $S^2$ with 8 small $S^1$;
\item at $t=t_2$ by hyperbolic Dehn surgerying 4 small $S^3$ with four $S^0$;
\item at $t=0$, the cone angles tend to $2 \pi$ and $W_t$ degenerates into $C\times S^1$.
\end{itemize}
When $t\in (t_1,1)$  the singular set of $W_t$ is an immersed geodesic surface made of 12 cone-tori and 8 cone-spheres. When $t\in(0,t_1)$ the singular set is a 2-complex with generic singularities.
\end{teo}

The manifolds or orbifolds $W_t$ at the times $t=1, t_1, \bar t$ have Euler characteristic 8, 8, and 5. The volume of $W_t$ is shown in Figure \ref{vol_Wt:fig}. 
In the degeneration, the holonomy of $W_t$ tends algebraically to the holonomy of $C$.

The behaviour of $W_t$ when $t\in [t_1,1]$ is much similar to the one of $M_t$ from Theorem \ref{main:teo} when $t\in [0,1]$, as
will be evident from the construction. The cone-manifolds $W_t$ are also constructed using the Kerckhoff--Storm deforming polytopes mentioned above.

\begin{figure}
\vspace{.5 cm}
\labellist
\small\hair 2pt
\pinlabel $t_1$ at 480 0 
\pinlabel $t_2$ at 430 0 
\pinlabel $\bar t$ at 350 0 
\pinlabel $\frac {32}3\pi^2$ at -8 420 
\pinlabel $\frac {20}3\pi^2$ at -8 270 
\endlabellist
\centering
\includegraphics[width=8 cm]{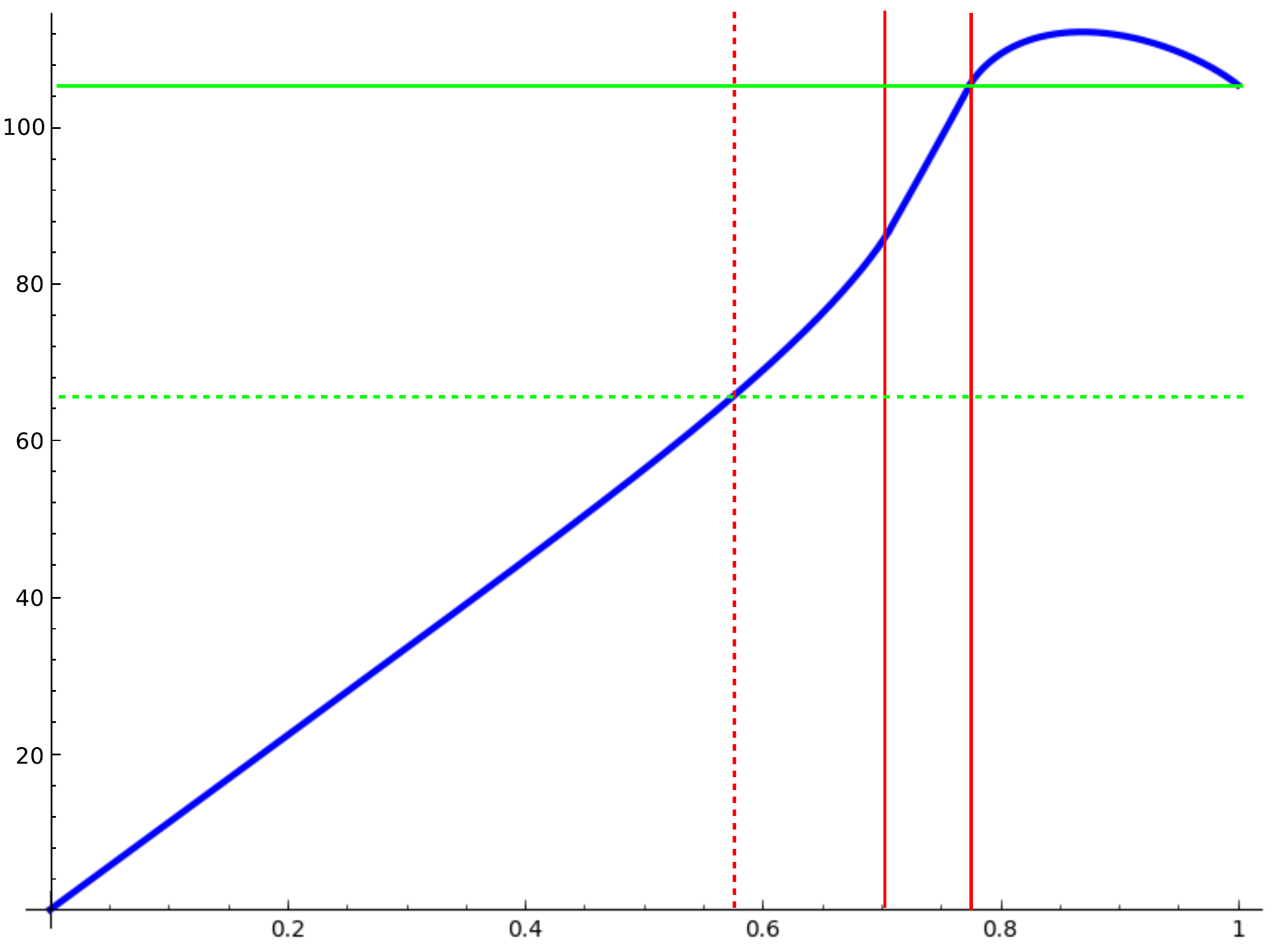}
\vspace{.3 cm}
\nota{The function $\Vol(W_t)$.}\label{vol_Wt:fig}
\end{figure}

\subsection*{Acknowledgements} We thank Joan Porti and the anonymous referee for pointing out a mistake in an earlier version of Theorem \ref{main:teo}.

\subsection*{Structure of the paper} 
The paper is organized as follows. In Section \ref{preliminaries:section} we recall some well-known facts about (acute-angled) polytopes, Coxeter diagrams, and cone manifolds. The main references are the seminal papers of Vinberg \cite{Vin} and McMullen \cite{McM}.

In Section \ref{polytopes:section} we define and study the family of finite-volume polytopes $P_t \subset \matH^4$. The quite long section is almost entirely self-contained: many arguments were taken from the paper of Kerckhoff and Storm \cite{KS}, which is fundamental for our constructions, and are included for the sake of completeness.

Finally, the deforming cone-manifolds $W_t, N_t, M_t$ are constructed in Section \ref{manifolds:section} by assembling carefully some copies of $P_t$. Theorems \ref{main:teo} and \ref{main2:teo} are proved there.

\section{Preliminaries} \label{preliminaries:section}

We introduce in this section some preliminaries on polytopes and cone-manifolds, focusing mostly on dimension four.

\subsection{Polytopes}\label{polytopes_prelim:section}
We represent the hyperbolic four-space $\matH^4$ as the upper sheet of the hyperboloid $\langle v,v \rangle =-1$ in $\matR^5$ with respect to the Lorentzian product
$$\langle v,w \rangle = -v_0w_0 + v_1w_1+v_2w_2+v_3w_3+v_4w_4.$$

\subsection*{Half-spaces}
Every space-like vector $v$ determines a half-space in $\matH^4$, that consists of all $w\in\matH^4$ with $\langle v,w\rangle \leq 0$. We are interested in the case where two space-like vectors $v$ and $v'$ determine two half-spaces whose intersection is non-empty and is a proper subset of both half-spaces. There are three possible configurations to consider, easily determined by the number
\begin{equation}  \label{formula:eqn}
\alpha = \frac{-\langle v,v'\rangle}{\sqrt{\langle v,v\rangle \langle v',v'\rangle}}
\end{equation}
as follows:
\begin{itemize}
\item if $-1< \alpha < 1$, the boundary hyperplanes of the two half-spaces intersect with a dihedral angle $\theta$ such that $\cos \theta = \alpha$;
\item if $\alpha = 1$, the boundary hyperplanes are asymptotically parallel;
\item if $\alpha > 1$, the boundary hyperplanes are ultra-parallel, and their distance $d$ is such that $\cosh d = \alpha$.
\end{itemize}

\subsection*{Finite polytopes}
We define as usual a (\emph{finite convex}) \emph{polytope} to be the intersection $P$ of finitely many half-spaces in $\matH^4$, with the additional hypothesis that $\interior P \neq \emptyset$. The boundary $\partial P$ is naturally stratified into \emph{vertices}, \emph{edges}, \emph{faces}, and \emph{walls} (also called \emph{facets}). 

If the closure $\overline P$ of $P$ in the compactification $\overline{\matH^4}$ intersects $\partial \matH^4$ in finitely many (possibly zero) points,
the volume of $P$ is finite; otherwise it is infinite. These points in $\partial \matH^4$ are called \emph{ideal vertices}.

\subsection*{Volume}
To compute the volume of a finite-volume even dimensional polytope $P$ there is a formula due to Poincar\'e (see \cite[page 120]{Vi}).
Denoting by $L_S$ the spherical link of the stratum $S$ and by $\alpha_F$ the dihedral angle at the (two-dimensional) face $F$, in dimension four the formula is
$$\mathrm{Vol}(P) = \frac{4\pi^2}{3}\!
\left(\!1-\frac 12 N
 +
 \frac{1}{2\pi}\sum_{F\, \mathrm{face}}\alpha_F 
 - 
\frac{1}{4\pi}\sum_{E\, \mathrm{edge}}\!\mathrm{Area}(L_E)
 +  
\frac{1}{2\pi^2}\sum_{V\, \mathrm{vertex}}\!\!\mathrm{Vol}(L_V)\!\right)$$
where $N$ is the number of walls.

In any dimension, there is also the well-known Schl\"afli formula (also on \cite[page 122]{Vi}) that expresses the variation of the volume of a deforming polytope $P$ (whose combinatorics stays constant) in terms of the area of the faces and of the variation of the dihedral angles. In dimension four, it is
$$d\mathrm{Vol}(P)=-\frac 13\sum_{F\ \mathrm{face}}\mathrm{Area}(F)d\alpha_F.$$
To apply that formula, recall that the area of a hyperbolic $k$-gon $F$ with inner angles $\alpha_1,\ldots,\alpha_k$ is
$$\mathrm{Area}(F)=(k-2)\pi-\sum_{i=1}^k\alpha_i.$$

\subsection*{Topology}
Let $X$ be a compact metric space. Recall that the Hausdorff distance defines a topology on the closed subsets of $X$ which depends only on the topology of $X$.

Every polytope and more generally every closed subset $C\subset \matH^n$ has a compactification $\overline C \subset \overline{\matH^n}$. We endow the family of all closed subsets $C\subset \matH^n$ with the Hausdorff distance topology of their compactifications in $\overline{\matH^n}$ (here $\overline{\matH^n}$ is equipped with any compatible metric). Note that the volume function on this family is not continuous.

This topology will be used tacitly throughout all the paper. The situation that is relevant here is when a family of polytopes is defined as the intersection of some moving half-spaces determined by some space-like vectors $v_1, \ldots, v_m$. If the vectors $v_1,\ldots, v_m$ move continuously, the polytope deforms continuously. 

\subsection{Acute-angled polytopes} \label{acute:subsection}
The theory of acute-angled hyperbolic polytopes is beautifully introduced in a paper of Vinberg \cite{Vin} and we briefly recall some of the facts described in that paper. We stick to dimension four for simplicity, although everything applies to any dimension.

\subsection*{Gram matrix}
Let $P\subset \matH^4$ be a polytope, defined as the intersection of the half-spaces dual to some unit space-like vectors $v_1,\ldots, v_m$. We calculate $\alpha_{ij}$ from $v_i, v_j$ using (\ref{formula:eqn}) for any $i, j$. The $m\times m$ matrix $-\alpha_{ij}$ is the \emph{Gram matrix} of $P$, see \cite{Vin}. 

We say that $P$ is \emph{acute-angled} if $\alpha_{ij} \geq 0$ for all $i\neq j$. Acute-angled polytopes have many nice properties. In this section, we will always suppose that $P$ is acute-angled.

\begin{rem}
By a theorem of Andreev \cite{A} a generic polytope $P$ is acute-angled if and only if all its dihedral angles are $\leq \frac \pi 2$, and this explains the terminology. 
\end{rem}

\subsection*{Generalised Coxeter diagrams}
The Gram matrix of an acute-angled polytope $P$ is nicely encoded via the \emph{generalised Coxeter diagram} $D$ of $P$, which is constructed as follows: every vertex of $D$ represents a vector $v_i$, and every edge between two distinct vertices $v_i$ and $v_j$ has a label that depends on $\alpha_{ij}\geq 0$ as follows:
\begin{itemize}
\item if $\alpha_{ij} > 1$ the edge is dashed (and sometimes labeled with the number $d>0$ such that $\cosh d = \alpha_{ij}$, but we will not do that);
\item if $\alpha_{ij} = 1$ the edge is thickened;
\item if $0\leq \alpha_{ij} < 1$ the edge is labeled with the angle $\frac \pi 2 \geq \theta> 0$ such that $\cos \theta = \alpha_{ij}$.
\end{itemize}
To simplify the picture, the edges labeled with an angle $\frac \pi 2$ are not drawn, and in those with $\frac \pi 3$ the label is omitted.

\subsection*{Strata}
The following facts are proved in \cite[Section 3]{Vin}. Every acute-angled polytope $P$ is \emph{simple}, that is each stratum $S$ of $P$ of codimension $k$ is contained in exactly $k$ walls. All the strata of $P$ may be easily determined from $D$ as follows: 
\begin{itemize}
\item the vertices $v_i$ represent the walls of $P$;
\item the pairs of vertices connected by an edge labeled with some angle $\theta$ represent the faces of $P$; the angle $\theta$ is the dihedral angle of that face;
\item more generally, the strata $S$ of codimension $k$ correspond to the $k$-uples of vertices of $D$ whose subdiagram represents a $(k-1)$-dimensional spherical simplex $L_S$; the spherical simplex $L_S$ is geometrically the link of $S$.
\end{itemize}

In particular, the set of vectors $v_1,\ldots, v_m$ defining $P$ is minimal (no proper subset defines $P$), and $k$ walls in $P$ intersect if and only if the hyperplanes containing them do. 
These nice facts are not true in general for non acute-angled polytopes.

\subsection*{Diagrams of the strata}
Every stratum $S$ of an acute-angled polytope $P$ is also acute-angled, and one can deduce a Coxeter diagram $D_S$ for $S$ from that $D$ of $P$. We explain how this works in the easier case when $S$ is a wall, the procedure can then be applied iteratively.

The diagram $D_S$ is formed by all the vertices of $D$ that represent walls that are incident to $S$; that is, $D_S$ is constructed from $D$ by removing the vertex $v_i$ corresponding to $S$ and all the vertices $v_j$ that are connected to $v_i$ by either a dashed or a thickened edge.

The resulting diagram $D_S$ is \emph{not} yet a generalised Coxeter diagram for $S$, because the value of $\alpha$ from formula (\ref{formula:eqn}) needs to be recomputed for every edge. To do so we must substitute each space-like vector $v_j$ with its projection $P(v_j)$ in the time-like hyperplane $v_i^\perp$ containing $S$, using the formula
$$P(v_j)=v_j-\frac{\langle v_j, v_i\rangle}{\langle v_i, v_i\rangle} v_i.$$
The new $\alpha \geq 0$ is computed using the projections $P(v_j)$ and is equal or bigger than the original one (in particular $S$ is still acute-angled).

\subsection*{Ideal vertices}
The ideal vertices $v$ of $P$ are also detected in a similar fashion: they correspond to the subdiagrams of $D$ that represent some compact $3$-dimensional Euclidean acute-angled polyhedron $Q$, which is in fact the link of $v$. The polyhedron $Q$ must be a product of simplexes, so the subdiagram is a disjoint union of diagrams representing Euclidean simplexes. (In all dimensions, every acute-angled spherical polytope without antipodal points is a simplex, and every acute-angled compact Euclidean polytope is a product of Euclidean simplexes.)

There is a combinatorial criterion that one can use to check from $D$ whether $P$ is compact and/or has finite volume, see \cite[Proposition 4.2]{Vin}. We suppose that $P$ contains at least one (finite or ideal) vertex.

\begin{teo} \label{Vinberg:teo}
The polytope $P$ is compact (has finite volume) if and only if each of its edges joins exactly two finite (finite or ideal) vertices.
\end{teo}

This condition is designed to exclude the presence of hyper-ideal vertices, see \cite{Vin}. In this paper we will only deal with finite-volume polytopes.

\subsection*{Coxeter polytopes}
If all the dihedral angles of $P$ are of type $\frac \pi n$ for some $n\geq 2$, then $P$ is a \emph{Coxeter polytope}. In this case the group $\Gamma < \Iso(\matH^4)$  generated by the reflections along its walls is discrete and has $P$ as a fundamental domain, so that $P = \matH^4/_\Gamma$ may be interpreted as an orbifold.

Recall that the orbifold Euler characteristic of a Coxeter polytope $P$ is given by the formula
$$\chi(P)=\sum_s\frac{(-1)^{\mathrm{dim} (s)}}{|\mathrm{Stab}(s)|},$$
where the sum is over all the strata $s$ of the polytope (ideal vertices are excluded) and $\mathrm{Stab}(s)$ is the stabilizer of a stratum inside the Coxeter reflection group of $P$.

\subsection{Cone-manifolds} \label{cone:intro:subsection}
Constant curvature cone-manifolds (and more generally $(X,G)$-cone-manifolds)
were defined by Thurston \cite{Th} inductively on the dimension as follows: a cone 1-manifold is an ordinary Riemannian 1-manifold, and a hyperbolic (or Euclidean, spherical) cone $n$-manifold is locally a hyperbolic (or Euclidean, spherical) cone over a compact connected spherical cone $(n-1)$-manifold.

Every point $p\in M$ in a hyperbolic (or Euclidean, spherical) cone $n$-manifold $M$ is locally a cone over 
a compact spherical cone $(n-1)$-manifold $S_p(M)$, called the \emph{unit tangent space} to $M$ at $p$. If $S_p(M)$ is isometric to $S^{n-1}$ the point is \emph{regular}, and it is \emph{singular} otherwise. The singular points form the \emph{singular set} $\Sigma \subset M$. McMullen defined a natural stratification on $\Sigma$ that we now recall, see \cite{McM} for more details (and proofs).

Let $EA$ denote the Euclidean cone over a spherical cone-manifold $A$. The \emph{join} $A*B$ of two spherical cone manifolds $A$ and $B$ is defined as
$$A*B = S_{(0,0)} (EA \times EB).$$
In particular we have $S^m * S^n \cong S^{m+n+1}$. We set $S^{-1} = \emptyset$.
It is proved in \cite[Theorem 5.1]{McM} that every compact spherical cone-manifold $N$ decomposes uniquely as a join 
$$N\cong S^{k-1} * B$$ 
for some $k\geq 0$ and some \emph{prime} $B$, that is a $B$ that does not decompose further as $B=S^0*C$. Let now $M$ be a hyperbolic (or Euclidean, spherical) cone $n$-manifold. We define
$$M[k] = \big\{p \in M\ \big|\ S_p(M) \cong S^{k-1}*B \ {\rm with}\ B \ {\rm prime}\big\}.$$
A \emph{$k$-stratum} of $M$ is a connected component of $M[k]$. It is a totally geodesic 
$k$-dimensional hyperbolic (or Euclidean, spherical) manifold. Points lying in the same $k$-stratum have isometric unit tangent spaces.

The regular points form the open dense set $M[n]$, and $M[n-1]$ is empty. The singular set $\Sigma=\cup_{k<n} M[k]$ has codimension at least two. If $M$ is complete (as it will always be the case in this paper) then $M$ is the metric completion of $M[n] = M \setminus \Sigma$.

We denote by $C_\theta$ the Riemannian circle of length $\theta$. The unit tangent space of a point $p\in M[n-2]$ is a join $S^{n-3} * C_\theta$ for some number $\theta \neq 2\pi$ that depends only on the stratum containing $p$, called the \emph{cone angle} of that stratum.

We list some examples of constant curvature cone manifolds. 

\subsection*{Cone-surfaces}
A hyperbolic (or Euclidean, spherical) cone-surface $S$ has some isolated singularities, each with a cone angle $\theta \neq 2 \pi$. Simple examples may be constructed by doubling polygons along their boundaries. 

\begin{figure}
\labellist
\small\hair 2pt

\pinlabel $S^0*C_\theta$ at 17 -15
\pinlabel $\theta$ at 30 0
\pinlabel $\theta$ at 30 90

\pinlabel $S^2(\alpha,\beta,\gamma)$ at 137 -15
\pinlabel $\alpha$ at 85 0
\pinlabel $\beta$ at 190 0
\pinlabel $\gamma$ at 145 90

\endlabellist
\centering
\includegraphics[width=5 cm]{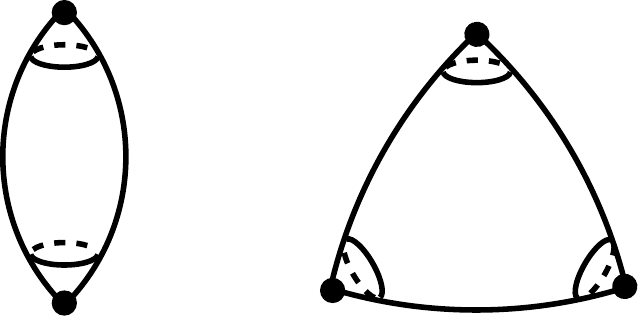}
\vspace{.3 cm}
\nota{The elliptic cone surfaces $S^0 * C_\theta$ and $S^2(\alpha, \beta, \gamma)$.}
\label{cone_surfaces:fig}
\end{figure}

If we double a spherical bigon with inner angles $\frac \theta 2$ we get a cone-sphere with two singular points of angle $\theta$, which is isometric to the join $S^0*C_{\theta}$. If we double a spherical triangle with inner angles $\frac \alpha 2,\frac \beta 2,\frac \gamma 2$ we get a cone-sphere with three singular points of cone angle $\alpha,\beta,\gamma$. This is a prime spherical cone-surface and we denote it by $S^2(\alpha,\beta,\gamma)$. See Figure \ref{cone_surfaces:fig}.

By Gauss-Bonnet, every compact connected orientable spherical cone-surface with cone-angles $< 2\pi$ is a sphere with some singular points (possibly none). 

\subsection*{Cone three-manifolds} On a hyperbolic (or Euclidean, spherical) cone 3-manifold $M$ the singular set $\Sigma = M[0] \cup M[1]$ has dimension $\leq 1$. Each 1-stratum has some cone angle $\theta$, while the unit tangent space at every point $p\in M[0]$ is some prime spherical cone-surface. For instance, it may be $S^2(\alpha,\beta,\gamma)$. 

Some spherical cone 3-manifolds are shown in Figure \ref{cone_3_manifolds:fig}. The join $S^1 * C_\theta$ is $S^3$ with an unknotted closed geodesic of length $2\pi$ and of cone angle $\theta$. The join $S^0 * S^2(\alpha,\beta,\gamma)$ is $S^3$ with singular set $\teta$ and cone angles $\alpha,\beta,\gamma$. If we double a spherical tetrahedron with dihedral angles $\frac \alpha 2, \ldots, \frac \zeta 2$ we get $S^3$ with singular set the 1-skeleton \tetra\ of a tetrahedron and cone angles $\alpha, \ldots, \zeta$: this is a prime spherical cone 3-manifold. 

\begin{figure}
\labellist
\small\hair 2pt

\pinlabel $S^1*C_\theta$ at 40 -15
\pinlabel $\theta$ at 40 65

\pinlabel $S^0*S^2(\alpha,\beta,\gamma)$ at 150 -20
\pinlabel $\alpha$ at 112 35
\pinlabel $\beta$ at 183 12
\pinlabel $\gamma$ at 183 60

\pinlabel $\alpha$ at 225 47
\pinlabel $\beta$ at 250 2
\pinlabel $\delta$ at 263 22
\pinlabel $\zeta$ at 244 24
\pinlabel $\varepsilon$ at 277 50
\pinlabel $\gamma$ at 248 50

\pinlabel $C_\theta*C_\varphi$ at 360 -20
\pinlabel $\theta$ at 312 35
\pinlabel $\varphi$ at 402 60

\endlabellist
\centering
\includegraphics[width=10 cm]{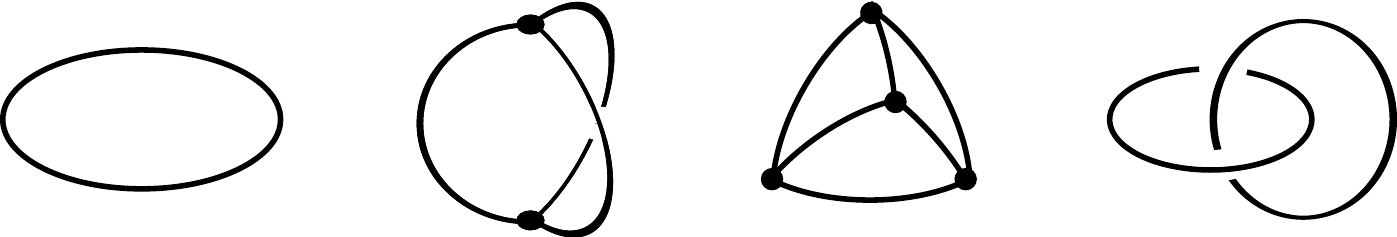}
\vspace{.5 cm}
\nota{Some simple spherical cone 3-manifolds. In all cases the underlying manifold is $S^3$.}
\label{cone_3_manifolds:fig}
\end{figure}

A spherical cone 3-manifold that is crucial in this paper is the join $C_\theta * C_\varphi$ with $\theta,\varphi \neq 2\pi$ shown in Figure \ref{cone_3_manifolds:fig}-(right). This is $S^3$ with singular set the Hopf link: one component of the Hopf link has length $\theta$ and cone angle $\varphi$, while the other has length $\varphi$ and cone angle $\theta$. This is a prime spherical cone 3-manifold (although it decomposes non-trivially as a join). 

If we assume that all cone-angles are $< 2\pi$, then every orientable hyperbolic (or Euclidean, spherical) cone 3-manifold is supported on a manifold.

\subsection*{Cone four-manifolds} 
On a hyperbolic (or Euclidean, spherical) cone 4-manifold $M$ the singular set $\Sigma = M[0] \cup M[1] \cup M[2]$ has dimension $\leq 2$. Each 2-stratum has some cone angle $\theta$. In each 1-stratum the unit tangent space of a point is $S^0 * B$ for some prime spherical cone-surface $B$. At each 0-stratum the unit tangent space is a prime spherical cone 3-manifold.

Figure \ref{foams:fig} shows the types of singularities in a cone 4-manifold that we will encounter in this paper: they are obtained by coning the spherical cone-manifolds shown in Figure \ref{cone_3_manifolds:fig}, and are in some sense the simplest kind of singularities that may occur in dimension four. A hyperbolic cone four-manifold with these types of singularities is topologically a manifold.

\begin{figure}
\labellist
\small\hair 2pt

\pinlabel $\theta$ at 35 40

\pinlabel $\alpha$ at 69 35
\pinlabel $\beta$ at 103 12
\pinlabel $\gamma$ at 98 60

\pinlabel $\alpha$ at 177 22
\pinlabel $\beta$ at 148 24
\pinlabel $\gamma$ at 177 50
\pinlabel $\delta$ at 148 50
\pinlabel $\varepsilon$ at 163 39

\pinlabel $\theta$ at 246 40
\pinlabel $\varphi$ at 228 60

\endlabellist

\begin{center}
\includegraphics[width=10 cm]{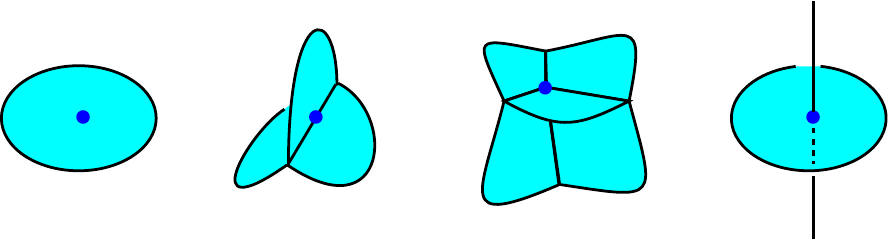}
\nota{Singular points in a cone four-manifold with unit tangent space as in Figure \ref{cone_3_manifolds:fig}. In the right picture we have two discs intersecting transversely in a point, with cone angles $\theta$ and $\varphi$. In all cases the singularity lies locally in a topological four-ball: a cone-manifold with this kind of singularities is topologically a manifold.}
\label{foams:fig}
\end{center}
\end{figure}

\begin{example}
If we pick a compact acute-angled (hence simple) polytope $P\subset \matH^4$ and double it along its boundary, we get a hyperbolic cone-manifold with underlying space $S^4$ and singularities of the first three kinds shown in Figure \ref{foams:fig}. A 2-complex $\Sigma$ with these generic local singularities is sometimes called a \emph{foam}. 
\end{example}

If $M[1] = \emptyset$ and the unit tangent space at every point in $M[0]$ is isometric to $C_\theta * C_\varphi$ (that is, if the only singularities in $M$ are like the first and the last one in Figure \ref{foams:fig}) we say that $\Sigma$ is an \emph{immersed geodesic cone-surface}. In this case we may see $\Sigma$ as the image of a geodesic immersion $\tilde\Sigma \looparrowright M$ of a hyperbolic cone surface $\tilde\Sigma$ obtained by resolving the double points of $\Sigma$ lying in $M[0]$. Every point $p$ in $M[0]$ with unit tangent space $C_\theta * C_\varphi$ is the image of two singular points in $\tilde\Sigma$ with cone angles $\theta$ and $\varphi$. The hyperbolic cone four-manifolds that arise in Theorem \ref{main:teo} are of this kind.

\section{The polytopes.}  \label{polytopes:section}
We now introduce a family of finite-volume polytopes $P_t\subset \matH^4$ that depend on a parameter $t\in (0,1]$, obtained by deforming the \emph{ideal right-angled 24-cell} $P_1$. The family is constructed by truncating the infinite-volume polytopes $\mathcal{F}_t$ built by Kerckhoff and Storm in \cite{KS} with two additional hyperplanes. We try to follow \cite{KS} as much as we can, reproducing all the notation used there. As in \cite{KS}, we will think of
this deformation running backwards from $t=1$, starting with the ideal 24-cell $P_1$ and eventually degenerating to a three-dimensional polyhedron (an ideal cuboctahedron) when $t \to 0$.

In Section \ref{manifolds:section} we will use $P_t$ to construct the deforming hyperbolic cone-manifolds $M_t$ and $W_t$ needed to prove Theorems \ref{main:teo} and \ref{main2:teo}. We warn the reader that the time parameter $t$ used for $P_t$ and $W_t$ differ from that employed to define $M_t$ by a linear rescaling: the manifold $M_t$ of Theorem \ref{main:teo} will be constructed by employing $P_t$ within the segment 
$$t \in \left[\sqrt {\frac 35}, 1 \right].$$
The times $t < \sqrt {\frac 35}$ will not be used to prove Theorem \ref{main:teo}, but only to prove Theorem \ref{main2:teo}. The reader interested only in Theorem \ref{main:teo} may thus ignore our discussion on $P_t$ when $t < \sqrt{ \frac 35}$. 

There are in fact two very important times in the deformation $P_t$ where the polytope changes its combinatorics. These are:

$$t_2 = \sqrt{\frac 12}, \qquad t_1 = \sqrt{ \frac 35}.$$

The combinatorics also changes at the initial time $t=1$, and at the final time $t=0$ where $P_t$ degenerates to a three-dimensional polyhedron. We will sometimes call $0,t_2,t_1,1$ the \emph{critical times} of the family $P_t$.

Many of the results presented in this section were first proved in \cite{KS} and we include them here only for the sake of completeness.

\begin{table}
\begin{eqnarray*}
\p{0} = \left( \sqrt{2},1,1,1,1/t \right) , & &
\m{0} = \left( \sqrt{2},1,1,1,-t \right),\\
\p{1} = \left( \sqrt{2},1,-1,1,-1/t\right),& &
\m{1} = \left( \sqrt{2},1,-1,1,t\right),\\
\p{2} = \left( \sqrt{2},1,-1,-1,1/t\right),& &
\m{2} = \left( \sqrt{2},1,-1,-1,-t\right),\\
\p{3} = \left( \sqrt{2},1,1,-1,-1/t\right),& &
\m{3} = \left( \sqrt{2},1,1,-1,t\right),\\
\p{4} =\left(\sqrt{2},-1,1,-1,1/t\right),& &
\m{4} =\left(\sqrt{2},-1,1,-1,-t\right),\\
\p{5} =\left(\sqrt{2},-1,1,1,-1/t\right),& &
\m{5} = \left( \sqrt{2},-1,1,1,t\right),\\
\p{6} =\left(\sqrt{2},-1,-1,1,1/t\right),& &
\m{6} =\left(\sqrt{2},-1,-1,1,-t\right),\\
\p{7} =\left(\sqrt{2},-1,-1,-1,-1/t\right),& &
\m{7} =\left(\sqrt{2},-1,-1,-1,t\right),\\
\l{A} =\left(1,\sqrt{2},0,0,0\right),& &
\l{B} =\left(1,0,\sqrt{2},0,0\right),\\
\l{C} =\left(1,0,0,\sqrt{2},0\right),& &
\l{D} =\left(1,0,0,-\sqrt{2},0\right),\\
\l{E} =\left(1,0,-\sqrt{2},0,0\right),& &
\l{F} =\left(1,-\sqrt{2},0,0,0\right), \\
\l{G} =\left(1,0,0,0,-\sqrt{2}t\right),& &
\l{H} =\left(1,0,0,0,\sqrt 2 t\right).
\end{eqnarray*}
\caption{The half-spaces that define $P_t$ are the duals of these space-like vectors: we denote vectors and half-spaces by the same letters. These vectors are indeed space-like for all $t\in (0,1]$, except $\l G$ and $\l H$ that are space-like only for $t\in (t_2,1]$.}\label{walls:table}
\end{table}

\subsection{The family $P_t$}\label{poly_def:sec}
We define
$$t_2 = \sqrt{\frac 12}$$
and we consider the 24 half-spaces $\p 0, \m 0, \ldots, \l G, \l H$ listed in Table \ref{walls:table}, that depend on some parameter $t$. The parameter $t$ varies in $(0,1]$ for $\p 0, \m 0, \ldots, \l E, \l F$ and only in $(t_2,1]$ for 
$\l G$ and $\l H$. The reader may check that for these values the vectors listed in the table are indeed space-like and hence determine some half-spaces in $\matH^4$.

For every $t\in (0,1]$ we define $P_t$ as the intersection of all the half-spaces in the table that are present at the time $t$. That is, 
\begin{defn} Let $P_t$ be the intersection of the 24 half-spaces $\p 0, \m 0, \ldots ,\l G, \l H$ when $t\in (t_2,1]$, and of the 22 half-spaces $\p 0, \m 0, \ldots, \l E, \l F$ when $t \in (0,t_2]$. 
\end{defn}

\begin{prop}
The set $P_t$ is a polytope for all $t\in (0,1]$, that deforms continuously in $t\in (0,1]$.
\end{prop}
\begin{proof} 
To prove that $P_t$ is a polytope we only need to check that its interior is non-empty. The set $P_t$ contains a small ball centred at the point $(1,0,0,0,0)$, because the first entry of each vector in Table \ref{walls:table} is positive, for every $t\in (0,1]$.

The deformation is clearly continuous, also at the singular time $t=t_2$ because the half-spaces $\l G$ and $\l H$ tend to the full $\matH^4$ as $t\to t_2$ (the space-like vertices defining them tend to light-like vertices).
\end{proof} 

\subsection*{The walls} The walls of $P_t$ are easily determined. We prove that the set of half-spaces that defines $P_t$ is minimal.

\begin{prop}
The boundary of each half-space $\p 0, \m 0 \ldots, \l G, \l H$ intersects $P_t$ in a wall, for all $t\in (0,1]$ for $\p 0, \m 0, \ldots, \l E, \l F$ and for all $t\in (t_2,1]$ for $\l G$ and $\l H$.
\end{prop}
\begin{proof}
The point $\big(\sqrt 2, \frac 23, \frac 23, \frac 23,0\big)$ belongs to the boundaries of both $\p 0$ and $\m 0$ and lies in the interior of all the other half-spaces: this proves the assertion for $\p 0$ and $\m 0$. By changing the signs of the $\frac 23$ entries we obtain the same for the other positive and negative faces.

The point $(\sqrt 2,1,0,0,0)$ belongs to the boundary of $\l A$ and lies in the interior of the other half-spaces. Similar points work for $\l B, \ldots, \l F$. The points $(\sqrt 2 t, 0,0,0, \mp 1)$ work for $\l G$ and $\l H$ when $t>t_2$.
\end{proof}

The polytope $P_t$ has 24 walls if $t\in (t_2, 1]$ and 22 walls if $t\in (0,t_2]$. We denote the walls of $P_t$ by the same symbols $\p 0, \m 0, \ldots, \l G, \l H$ of the corresponding half-spaces. 

\begin{rem}
Kerckhoff and Storm define for every $t\in (0,1]$ a bigger polytope $\calF_t$ as the intersection of the 22 half-spaces $\p 0, \m 0, \ldots ,\l E, \l F$. The polytope $\calF_t$ coincides with $P_t$ for $t\in (0,t_2]$, it has infinite volume for $t\in (t_2,1]$ and finite volume for $t \in (0,t_2]$. We will soon check that $P_t$ has finite volume for all $t\in (0,1]$.
\end{rem}

\subsection*{The right-angled ideal regular 24-cell}
As remarked in \cite[Section 3]{KS}, the polytope $P_1$ is the regular right-angled ideal 24-cell. The adjacencies between the walls $\p 0, \m 0, \ldots, \l{G}, \l H$ of $P_1$ are nicely codified in \cite[Figure 3.1]{KS}. 

The 24 walls of $P_1$ are subdivided into three octets: the \emph{positive}, the \emph{negative}, and the \emph{letter} walls, see Table \ref{walls:table}. Two walls of the same octet are never adjacent: this is the standard three-colouring of the 24-cell that was used in \cite{KM} to construct many hyperbolic four-manifolds.

\subsection*{The right-angled ideal cuboctahedron}
What happens as $t\to 0$? When $t=0$ the negative $\m 0, \ldots, \m 7$ and letter half-spaces $\l A,\ldots \l F$ are still defined. As $t\to 0$, every positive half-space converges to $(0,0,0,0,\pm 1)$, so they are also still defined (we keep identifying space-like vectors and half-spaces). We may still set $P_0$ to be the intersection of the half-spaces $\p 0, \m 0, \ldots, \l E, \l F$. As $t\to 0$, the polytope $P_t$ converges to $P_0$.

Among the half-spaces defining $P_0$ we find both $(0,0,0,0,1)$ and $(0,0,0,0,-1)$, hence $P_0$ is contained in the hyperbolic hyperplane $\{x_4=0\}$ isometric to $\matH^3$. Therefore $P_0$ is some lower-dimensional object. It is proved in \cite{KS} that $P_0 \subset \matH^3$ is a three-dimensional ideal polyhedron, and more precisely a right-angled ideal cuboctahedron, see also Proposition \ref{cuboct:prop} below. It has 14 faces, defined by the intersections of the 14 walls $\m 0, \m 1, \ldots, \l E, \l F$ with $\matH^3$.

Summing up, the family $P_t$ is a continuous deformation of polytopes that starts with the ideal regular right-angled 24-cell $P_1$ and eventually degenerates to the ideal right-angled cuboctahedron $P_0$.

\subsection{Symmetries}\label{symmetries:sec}
In the next sections, we will determine the combinatorics of the polytope $P_t$ for all times $t\in(0,1)$. Luckily, each $P_t$ has a big group of symmetries that will simplify our arguments significantly.

Consider the half-spaces determined by the space-like vectors
$$\l L=(0,-1,1,0,0),\ \l M=(0,0,-1,1,0),\ \l N=(0,0,-1,-1,0).$$
We denote by the same symbols the half-spaces and the reflections in the corresponding hyperplanes.
These reflections act as follows:
\begin{align*}
\l L\colon & (x_0,x_1,x_2,x_3,x_4) \longmapsto (x_0,x_2,x_1,x_3,x_4), \\
\l M\colon & (x_0,x_1,x_2,x_3,x_4) \longmapsto (x_0,x_1,x_3,x_2,x_4), \\
\l N\colon & (x_0,x_1,x_2,x_3,x_4) \longmapsto (x_0,x_1,-x_3,-x_2,x_4).
\end{align*}
Consider the group
$$H=\langle\l L, \l M,\l N\rangle.$$
The group $H$ is isomorphic to the symmetric group $\mathcal{S}_4$ (note that $(\l M\l N)^2=(\l L\l N)^3=(\l L\l M)^3=1$).
Moreover, in \cite[Section 4]{KS} it is shown that $H$ is the group of symmetries of the 24-cell $P_1$ that preserve:
\begin{itemize}
\item the positive/negative/letter colours of the walls;
\item the even/odd parity of the numbered walls;
\item the walls $\l G$ and $\l H$ (individually).
\end{itemize}
The group $H$ acts on the set of four positive (or negative) even (or odd) walls as its full permutation group.
Up to the action of $H$, the 24 walls $\{\p 0, \m 0, \ldots ,\l G, \l H \}$ reduce to the set
$$\{\p 0, \m 0, \p 3, \m 3, \l A, \l G, \l H\}.$$
Now, consider the order-two rotation 
$$R\colon (x_0,x_1,x_2,x_3,x_4) \longmapsto (x_0,x_1,x_2,-x_3,-x_4).$$ 
This rotation is called the \emph{roll symmetry} in \cite{KS}.
It still preserves $P_1$ and the positive/negative/letter colours of the walls, but it changes the parity of any numbered wall and it exchanges the walls $\l G$ and $\l H$. 
Kerckhoff and Storm prove that the extension
$$K=\langle\l L, \l M,\l N,R\rangle$$
has order 48 and consists precisely of the symmetries of $P_1$ that preserve the colours of the walls and the pair $\{ \l G, \l H \}$. Up to the action of $K$ the set of walls is further reduced to
$$\{\p 3, \m 0, \l A, \l G\}.$$

It is immediate to note that $K$ is also a group of symmetries of $P_t$ for every $t$
(in fact, it will be clear later that $K$ is the full group of symmetries of $P_t$ when $t<1$).
Up to symmetries the polytope $P_t$ has only four types of walls.

\subsection{The quotient polytope $Q_t$}\label{quotient:section}
As in \cite{KS}, we can quotient $P_t$ by the group $H$ of symmetries, and obtain an interesting smaller polytope $Q_t$ with a smaller number of walls. (If we quotient $P_t$ by $K$ we do not get a polytope!)

The quotient polytope $Q_t$ may be identified with the intersection of $P_t$ with the half-spaces $\l L$, $\l M$ and $\l N$. The walls of $Q_t$ are
$$\{\p 0, \m 0, \p 3, \m 3, \l A, \l G, \l H, \l L, \l M, \l N\}$$
when $t \in (t_2,1]$, and the same list with $\l G$ and $\l H$ removed when $t\in (0,t_2]$. The roll symmetry $R$ is a symmetry of $Q_t$ that permutes each pair
$$\{\p 0, \p 3\}, \ \{\m 0, \m 3\}, \ \{\l G, \l H\},\ \{\l M, \l N\}$$
and preserves the walls $\l L$ and $\l A$.  
We introduce another critical time:
$$t_1=\sqrt{\frac 35}.$$

\begin{figure}
\centering
\includegraphics[width=6 cm]{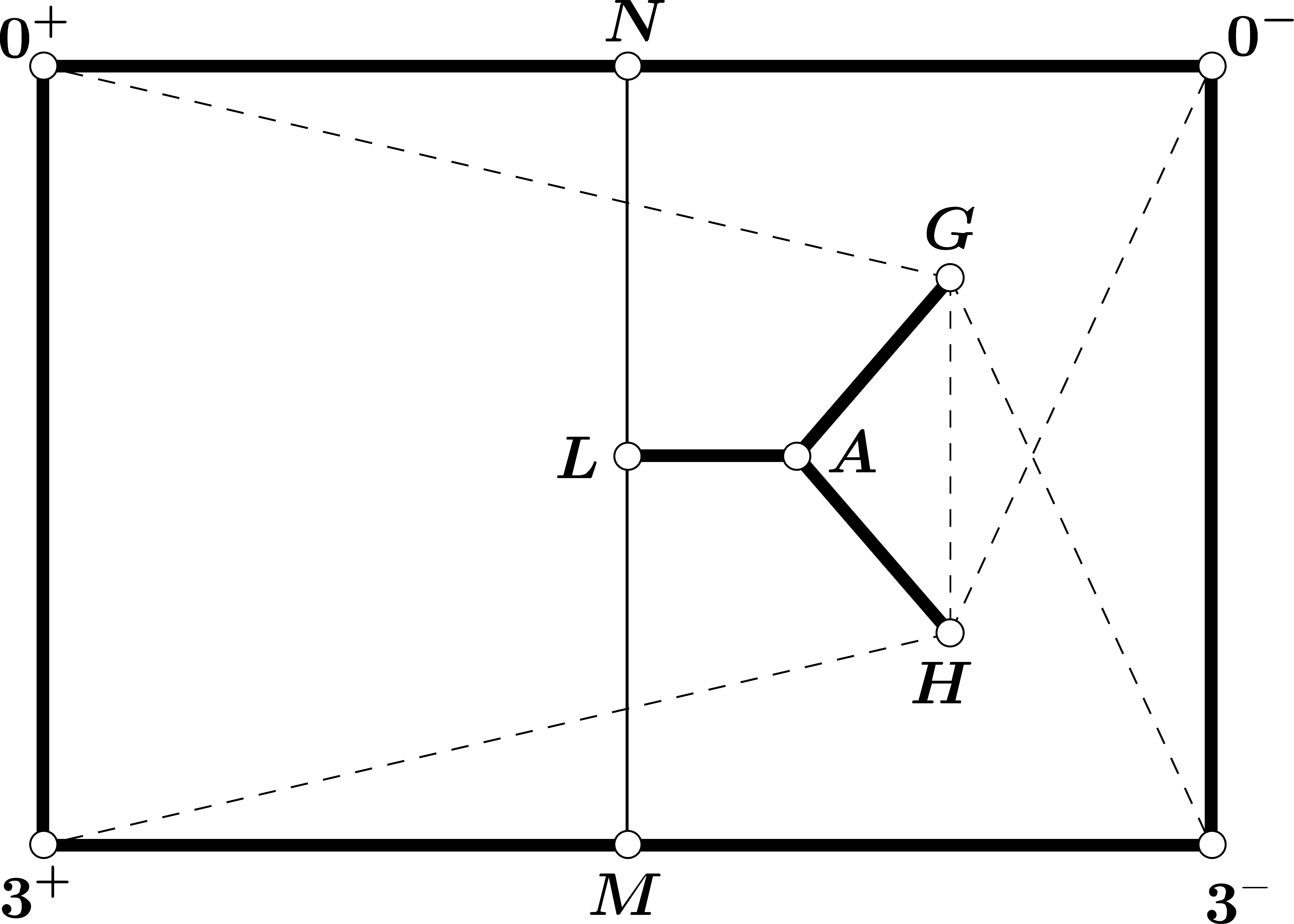}
\hspace{.2 cm}
\includegraphics[width=6 cm]{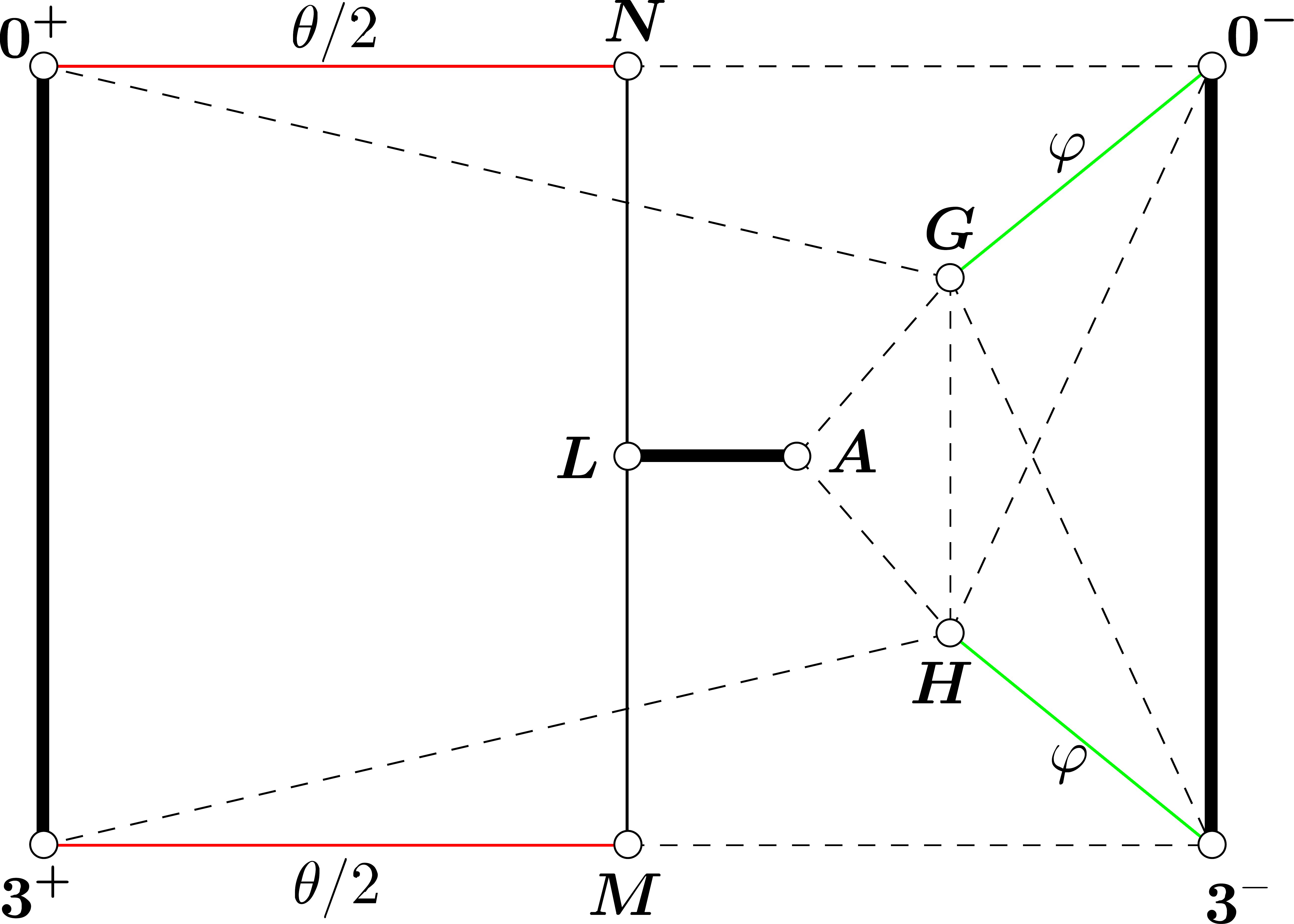}
\vspace{.5 cm}

\includegraphics[width=6 cm]{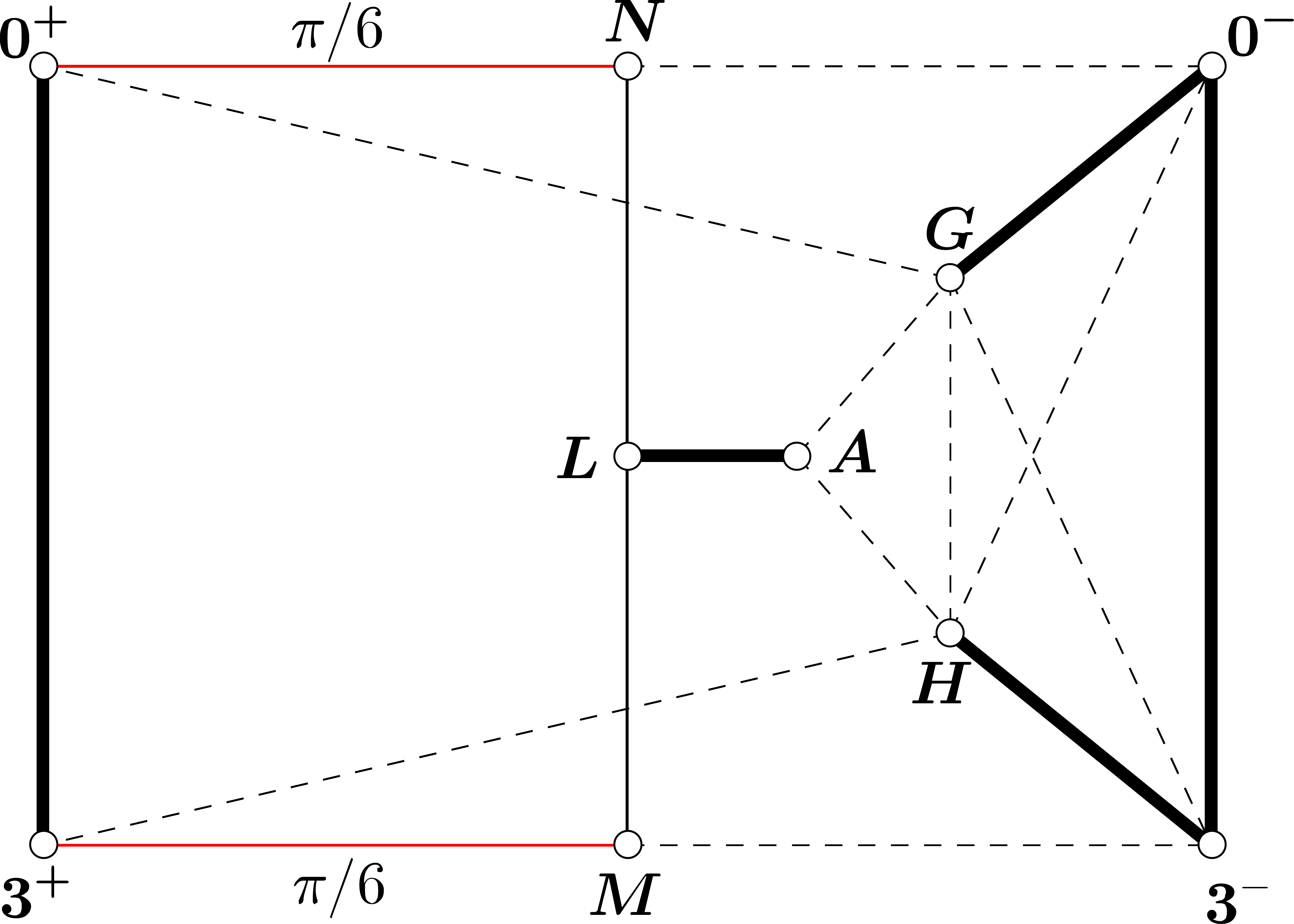}
\hspace{.2 cm}
\includegraphics[width=6 cm]{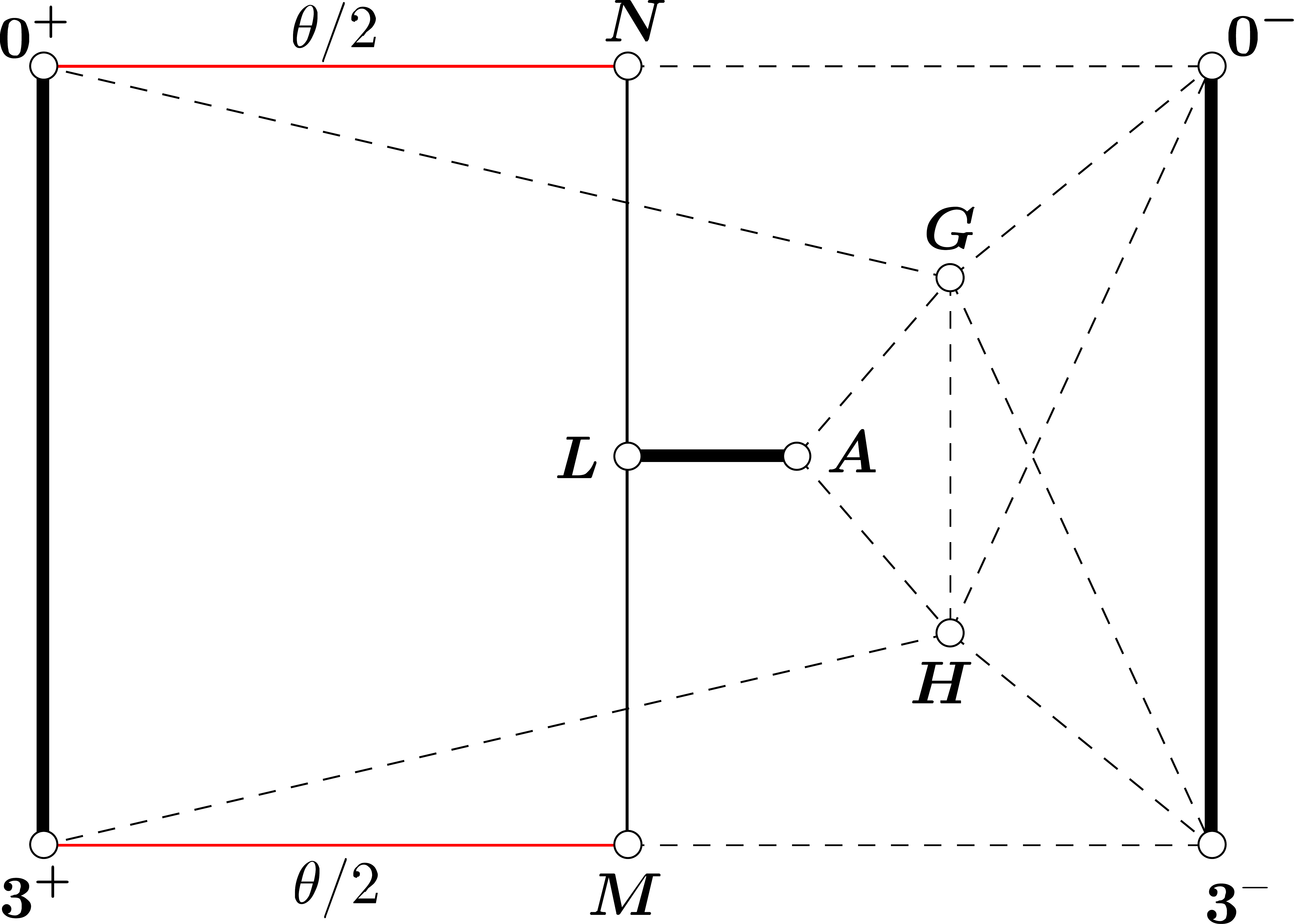}
\vspace{.5 cm}

\includegraphics[width=6 cm]{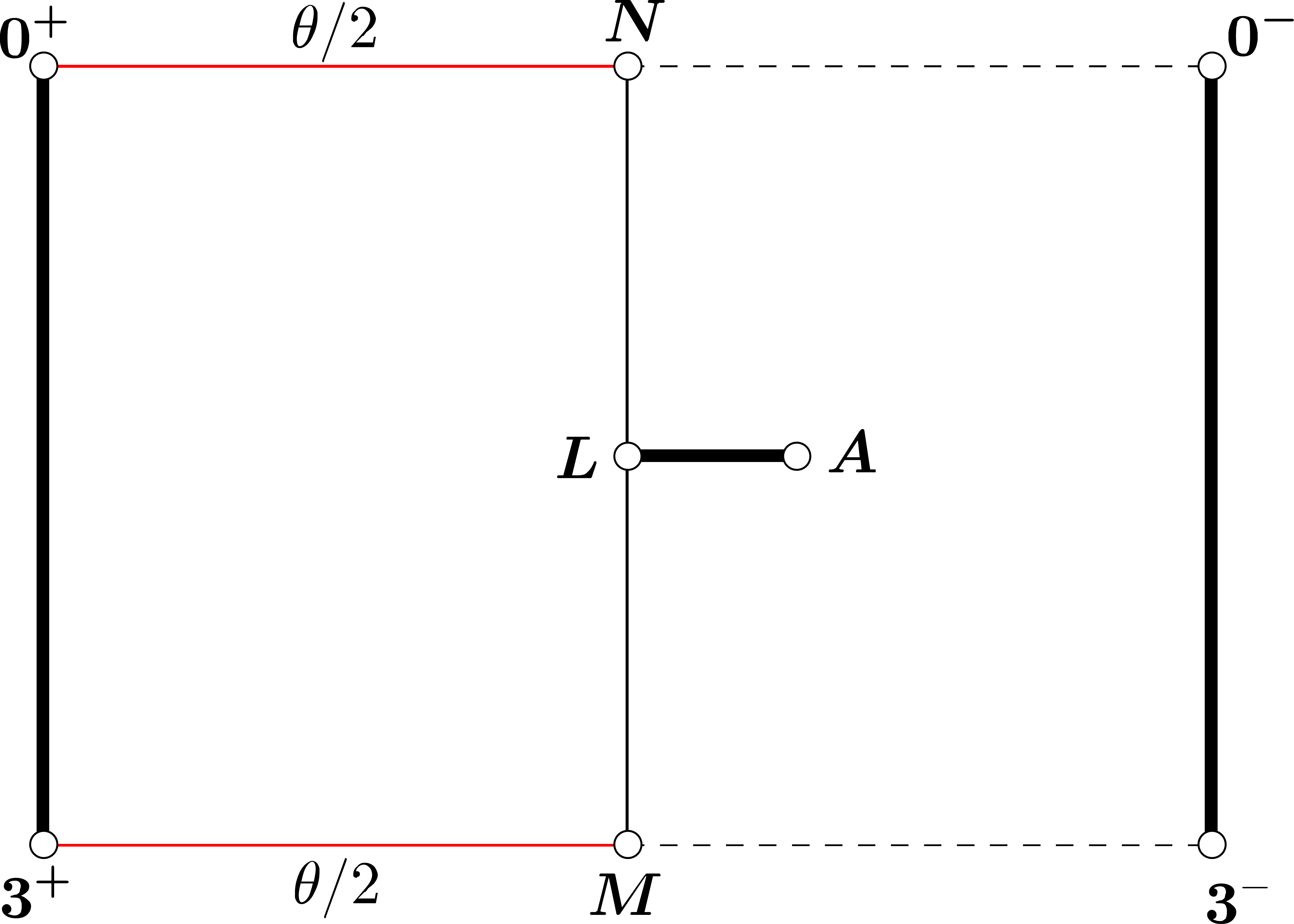}
\nota{The generalised Coxeter diagram $D_t$ of $Q_t$ when $t=1$, $t \in (t_1,1)$, $t=t_1$, $t\in (t_2,t_1)$, and $t\in (0,t_2]$, respectively. The green and red edges indicate the faces with varying dihedral angle $\varphi$ and $\frac \theta 2$. 
}\label{coxeter:fig}
\end{figure}

Note that $0<t_2<t_1<1$. We now show that the quotient polytope $Q_t$ is acute-angled for all $t\in (0,1]$ and may be fully described by some reasonable Coxeter diagrams, whose combinatorics changes at the critical times 1, $t_1$ and $t_2$. 

\begin{prop}\label{Q_t:prop}
The polytope $Q_t$ is acute-angled for all $t\in (0,1]$. Its generalised Coxeter diagram $D_t$ is shown in Figure \ref{coxeter:fig} for all $t \in (0,1]$. The dihedral angles $\frac \theta 2$ and $\varphi$ are such that
$$\cos \theta = \frac{3t^2-1}{1+t^2}, \qquad \cos \varphi = \frac{\sqrt 2 (1-t^2)}{\sqrt{(2t^2-1)(t^2+1)}}.$$
The dihedral angles $\frac\theta 2$ and $\varphi$ are defined for $t\in (0,1)$ and $t\in (t_1,1]$ respectively. They both vary strictly monotonically in $t$. We have:
$$\lim_{t\to 1}\tfrac \theta 2(t) = 0, \quad \tfrac \theta 2(t_1) = \tfrac \pi 6, \quad \lim_{t \to 0} \tfrac \theta 2(t) = \tfrac \pi 2, \qquad \varphi(1) = \tfrac \pi 2, \quad
\lim_{t\to t_1}\varphi (t) = 0.$$
\end{prop}

We plot the functions $\theta(t)$ and $\varphi(t)$ in Figure \ref{plot_theta_phi:fig}.

\begin{figure}
\vspace{.5 cm}
\labellist
\small\hair 2pt
\pinlabel $\theta(t)$ at 250 400
\pinlabel $t_1$ at 480 0 
\pinlabel $t_2$ at 430 0 
\pinlabel $\bar t$ at 350 0 
\pinlabel $\pi$ at 0 450 
\pinlabel $\frac \pi 2$ at 0 240 
\pinlabel $\frac \pi 3$ at 0 170 
\endlabellist
\includegraphics[width=6 cm]{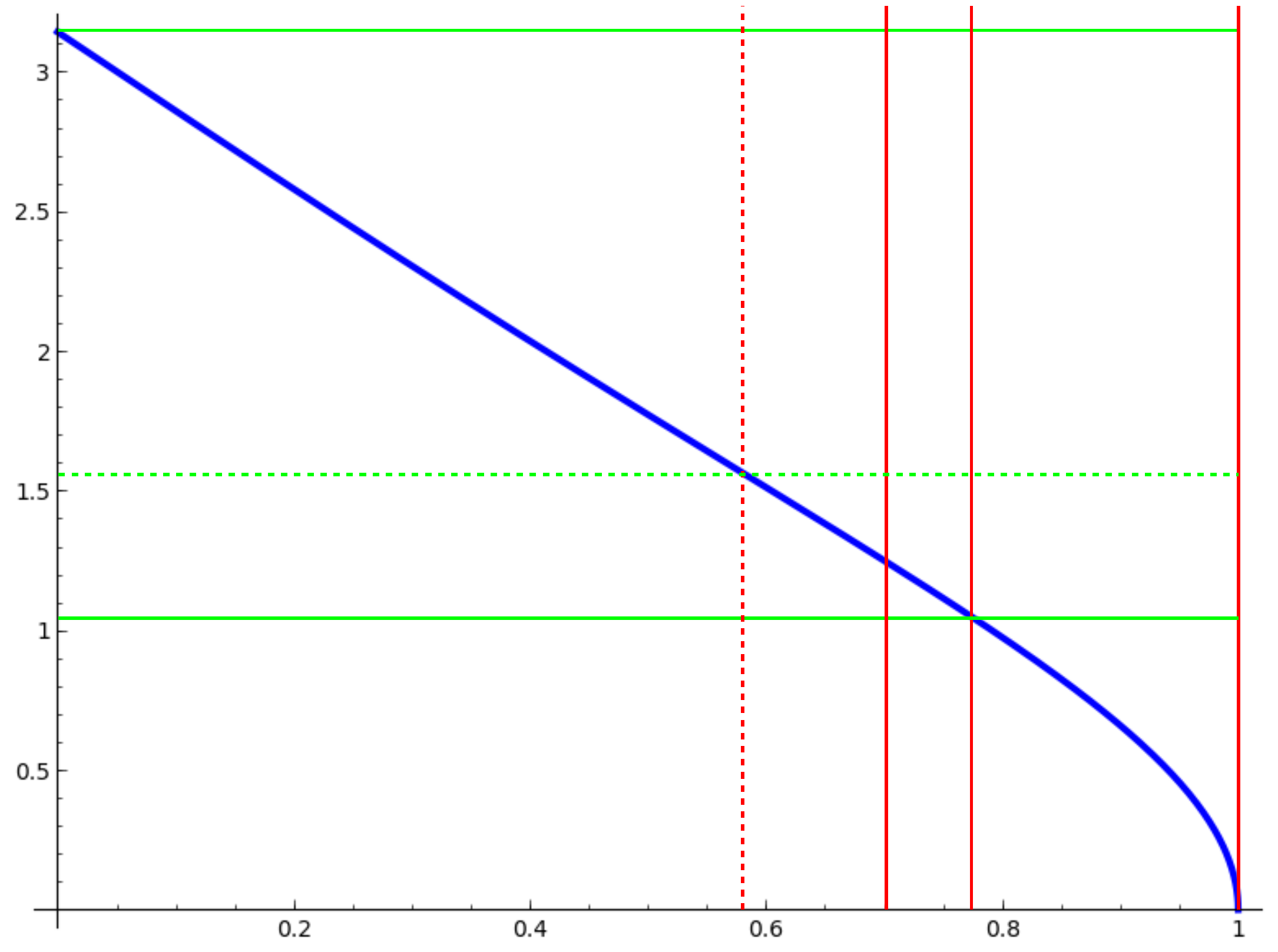}
\labellist
\small\hair 2pt
\pinlabel $\varphi(t)$ at 250 400
\pinlabel $t_1$ at 55 0 
\pinlabel $\frac \pi 2$ at 630 450 
\endlabellist
\includegraphics[width=6 cm]{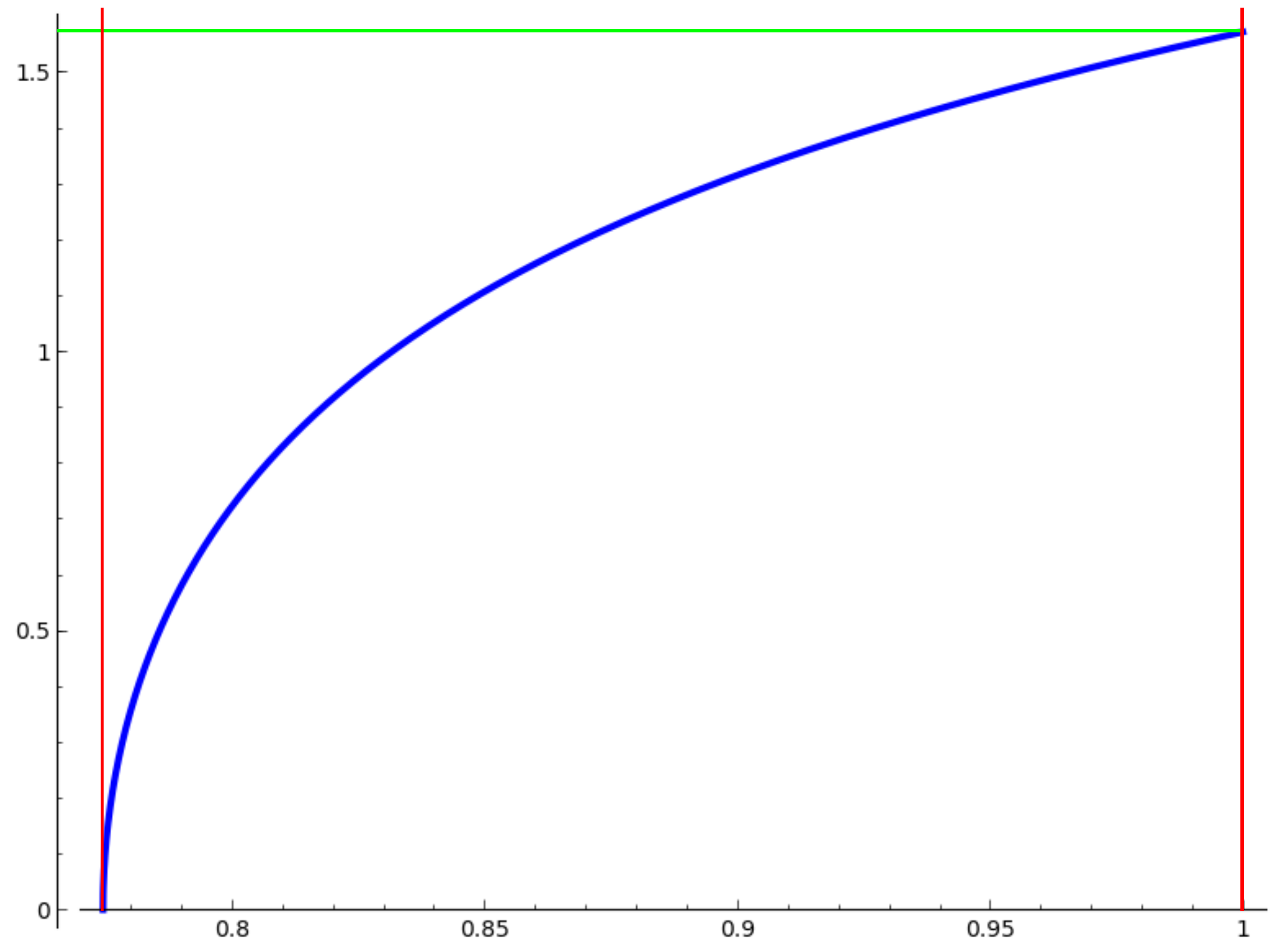}
\vspace{.3 cm}
\nota{The functions $\theta$ and $\varphi$. We note the critical times $t_2$ and $t_1$. At the non-critical time $\bar t = \sqrt{\frac 13}$ we get $\theta = \frac \pi 2$.}\label{plot_theta_phi:fig}
\end{figure}

\begin{proof}
We use the formula \eqref{formula:eqn} for every pair of walls in the set
\begin{equation} \label{defining:eqn}
\{\p 0, \m 0, \p 3, \m 3, \l A, \l G, \l H, \l L, \l M, \l N\}.
\end{equation}
We use the roll symmetry $R$ to reduce the number of pairs to be investigated. A simple inspection shows that we get $\alpha \geq 0$ for every pair and at every time $t\in (0,1]$. More precisely, for most pairs we get $\alpha >1$, $\alpha = 1$, $\alpha = \frac 12$, or $\alpha = 0$ for all $t\in (0,1]$, except (up to the roll symmetry) for the following:
\begin{enumerate}
\item with the pair $\{\p 0,\l N\}$ we get
$$\alpha = \frac{\sqrt 2 t}{\sqrt{1+t^2}}>0;$$
\item with the pair $\{\m 0,\l G\}$ we get 
$$\alpha = \frac{\sqrt 2 (1-t^2)}{\sqrt{(2t^2-1)(t^2+1)}} \geq 0;$$
recall that $\l G$ exists only for $t > t_2 = \sqrt{\frac 12}$;
\item with the pairs $\{\m 0, \l N\}$ and $\{\l A, \l G\}$ we get $\alpha = 1$ at $t=1$ and $\alpha>1$ for all $t<1$.
\end{enumerate}
Therefore $Q_t$ is acute-angled for all $t\in (0,1]$. Concerning the Coxeter diagrams, we note that:
\begin{enumerate}
\item with the pair $\{\p 0,\l N\}$, we get $\alpha =1$ at $t=1$ and $\alpha < 1$ for all $t<1$. Therefore when $t<1$ the walls intersect with dihedral angle $\frac \theta 2$ such that $\cos \frac \theta 2 = \alpha$, that is
$$\cos \theta = 2 \cos^2 \tfrac \theta 2 - 1 = 2 \alpha^2 - 1 = 2 \frac {2t^2}{1+t^2} -1  =\frac{3t^2 -1}{1+t^2}.$$
In particular when $t=t_1$ we get $\cos \theta = \frac 12$ and hence $\frac \theta 2 = \frac \pi 6$. By calculating the derivative one sees that $\theta$ varies strictly monotonically in $t$.
\item with the pair $\{\m 0,\l G\}$, we get $\alpha = 0$ at $t=1$. When $t\in (t_1,1)$ we get $0<\alpha <1$ and the half-spaces intersect with dihedral angle $\varphi$ such that $\cos \varphi = \alpha$. By calculating the derivative we see that $\varphi$ varies monotonically in $t$. When $t=t_1$ we get $\alpha = 1$ and when $t<t_1$ we get $\alpha > 1$. 
\end{enumerate}
The proof is complete.
\end{proof} 

The roll symmetry $R$ acts on the Coxeter diagram of $Q_t$ as a reflection with horizontal axis. 
The polytopes $Q_t$ are remarkable because they are acute-angled and have only few non-right dihedral angles, for every $t$.

\subsection*{Coxeter polytopes}
Recall that a Coxeter polytope is a polytope whose dihedral angles divide $\pi$. As noted in \cite{KS}, the polytope $Q_t$ is Coxeter at the times:
$$1,\quad t_1 = \sqrt{\frac 35}, \quad \frac{\cos \frac \pi 5}{\sqrt{1+\sin^2\frac \pi 5}} , 
\quad \sqrt {\frac 13}, \quad \sqrt{\frac 17}.$$
For these times, the dihedral angle $\frac \theta 2$ is respectively
$$0, \quad \frac \pi 6, \quad \frac \pi 5, \quad \frac \pi 4, \quad \frac \pi 3.$$
The dihedral angle $\varphi$ is $\frac \pi 2$ and $0$ in the first two cases. 
We get five Coxeter polytopes overall in the family $Q_t$. Using Vinberg's criterion, in \cite{KS} it is proved that they are all arithmetic, except the one with $\frac \theta 2 = \frac \pi 5$.

\subsection*{The walls}
We now describe the 3-dimensional walls of $Q_t$. 
Up to the roll symmetry $R$, there are only six walls to analyse in $Q_t$, namely
$$\m 0, \p 3, \l A, \l G, \l L, \l M.$$
Each such wall is an acute-angled polyhedron, because $Q_t$ is acute-angled. We are only interested in the first four $\m 0, \p 3, \l A$, $\l G$, that are quotients of some walls in $P_t$: understanding these will be enough to determine the combinatorics of all the walls in the original polytope $P_t$. We ignore the case $t=1$ for simplicity: we already know that $P_1$ is the ideal regular 24-cell.

\begin{figure}
\labellist
\small\hair 2pt
\pinlabel $\l A$ at -200 550
\endlabellist
\centering
\includegraphics[height=3.8 cm]{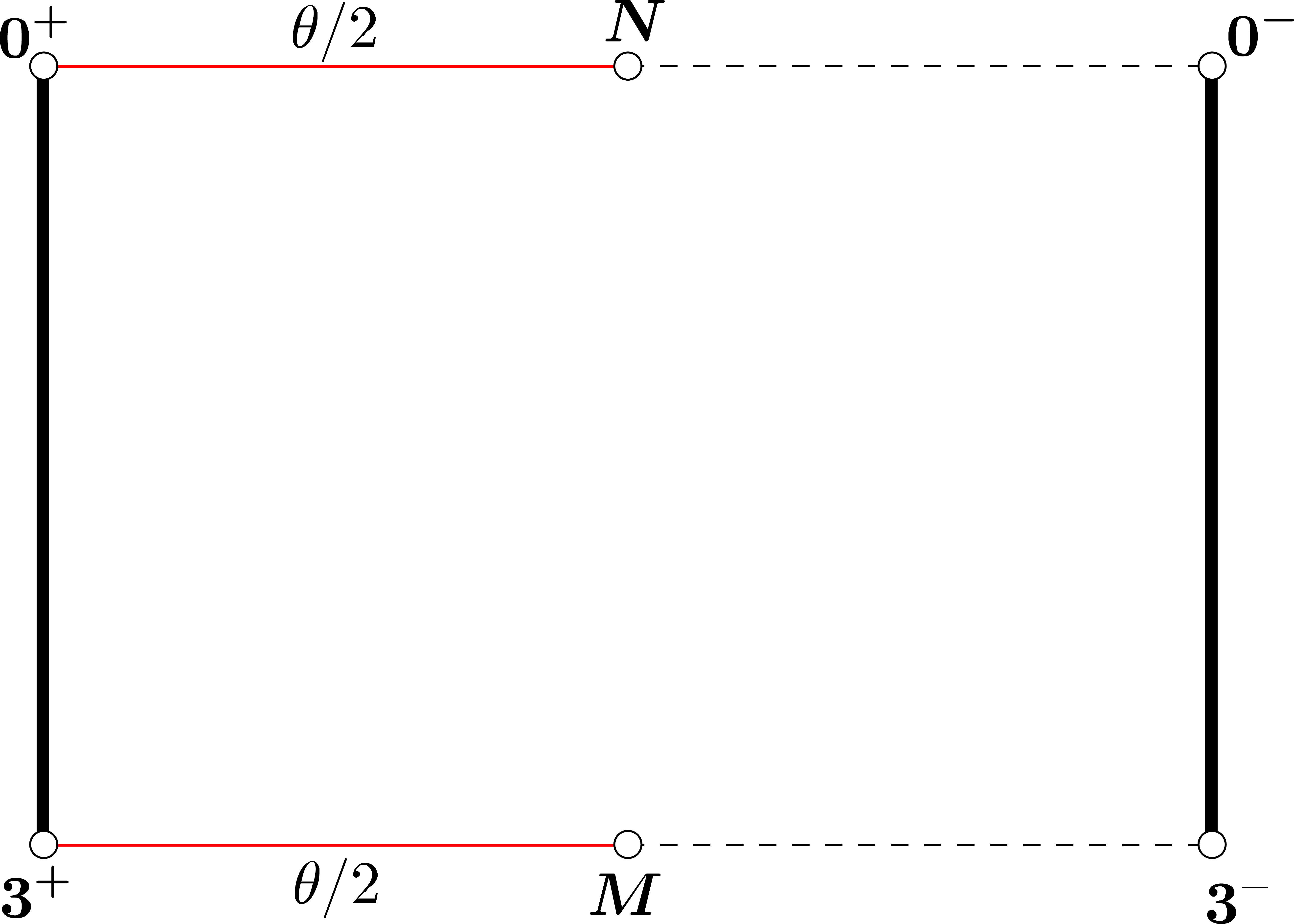}

\vspace{.5 cm}
\labellist
\small\hair 2pt
\pinlabel $\l G$ at -200 550
\pinlabel $(t_1,1)$ at -200 450
\pinlabel $\l G$ at 2900 550
\pinlabel $(t_2,t_1]$ at 2900 450
\endlabellist
\centering
\includegraphics[height=3.8 cm]{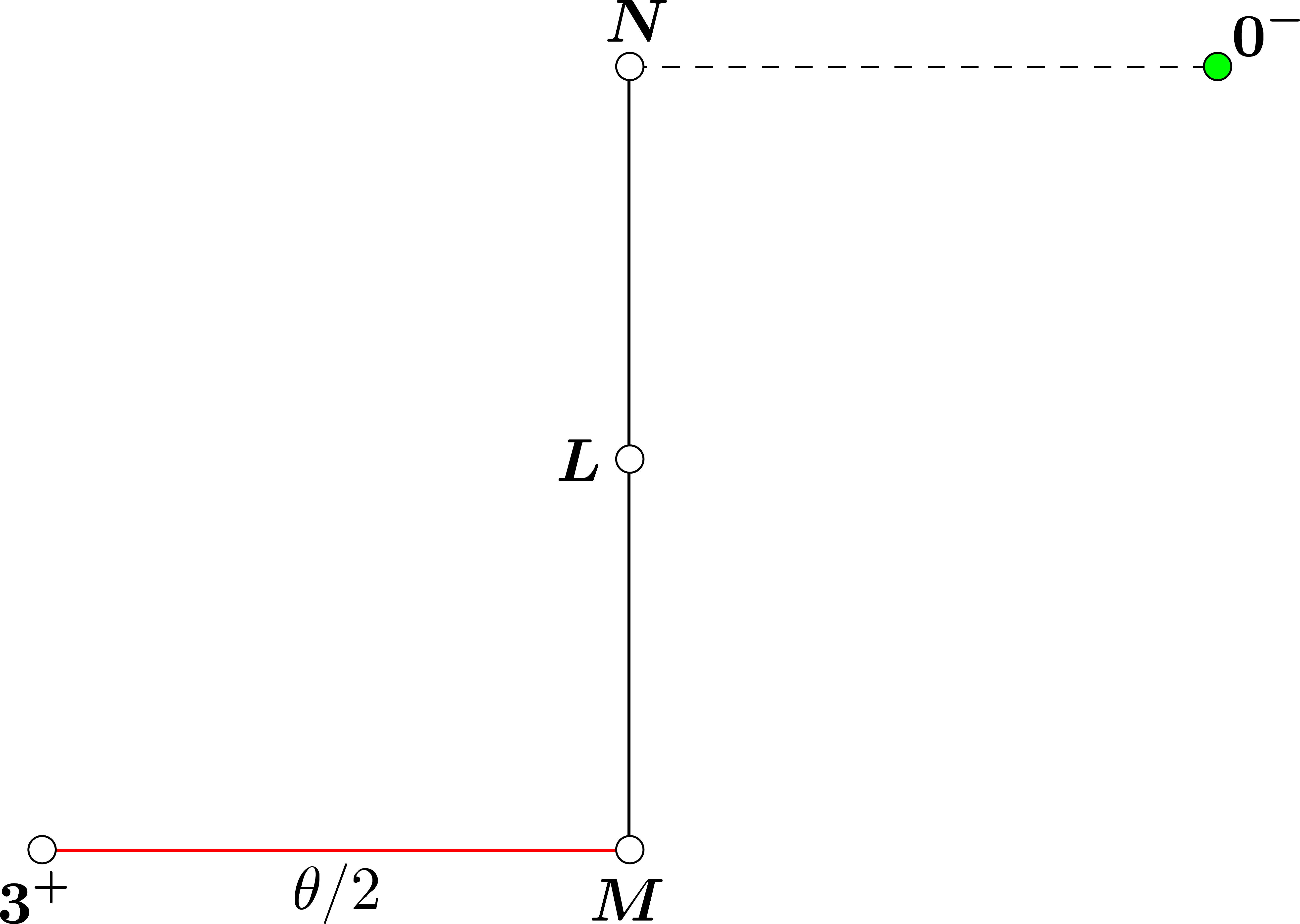}
\qquad
\includegraphics[height=3.8 cm]{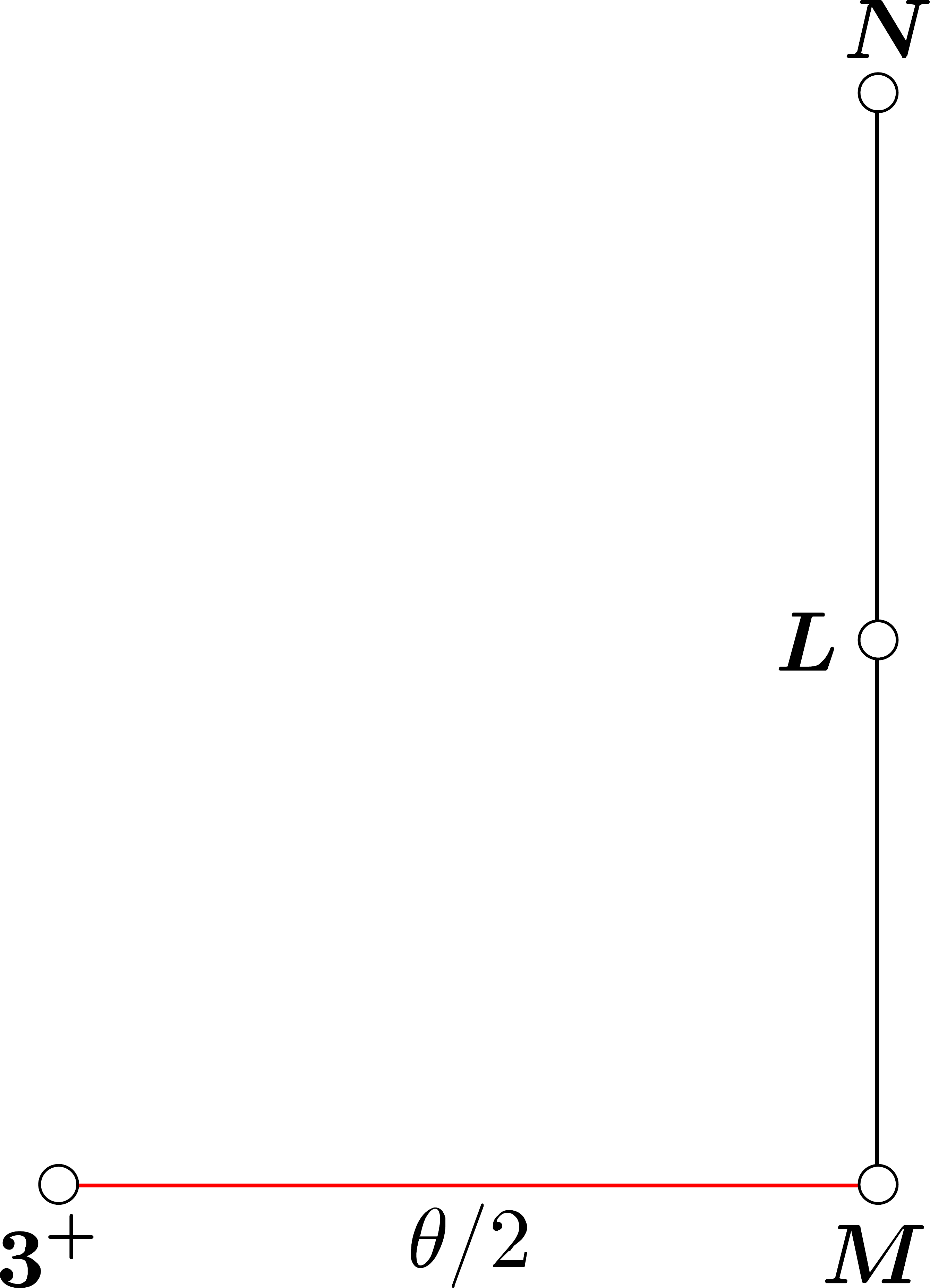}

\vspace{.5 cm}
\labellist
\small\hair 2pt
\pinlabel $\m 0$ at -220 550
\pinlabel $(t_1,1)$ at -220 450
\pinlabel $\m 0$ at 2800 550
\pinlabel $(0,t_1]$ at 2800 450
\endlabellist
\centering
\includegraphics[height=3.8 cm]{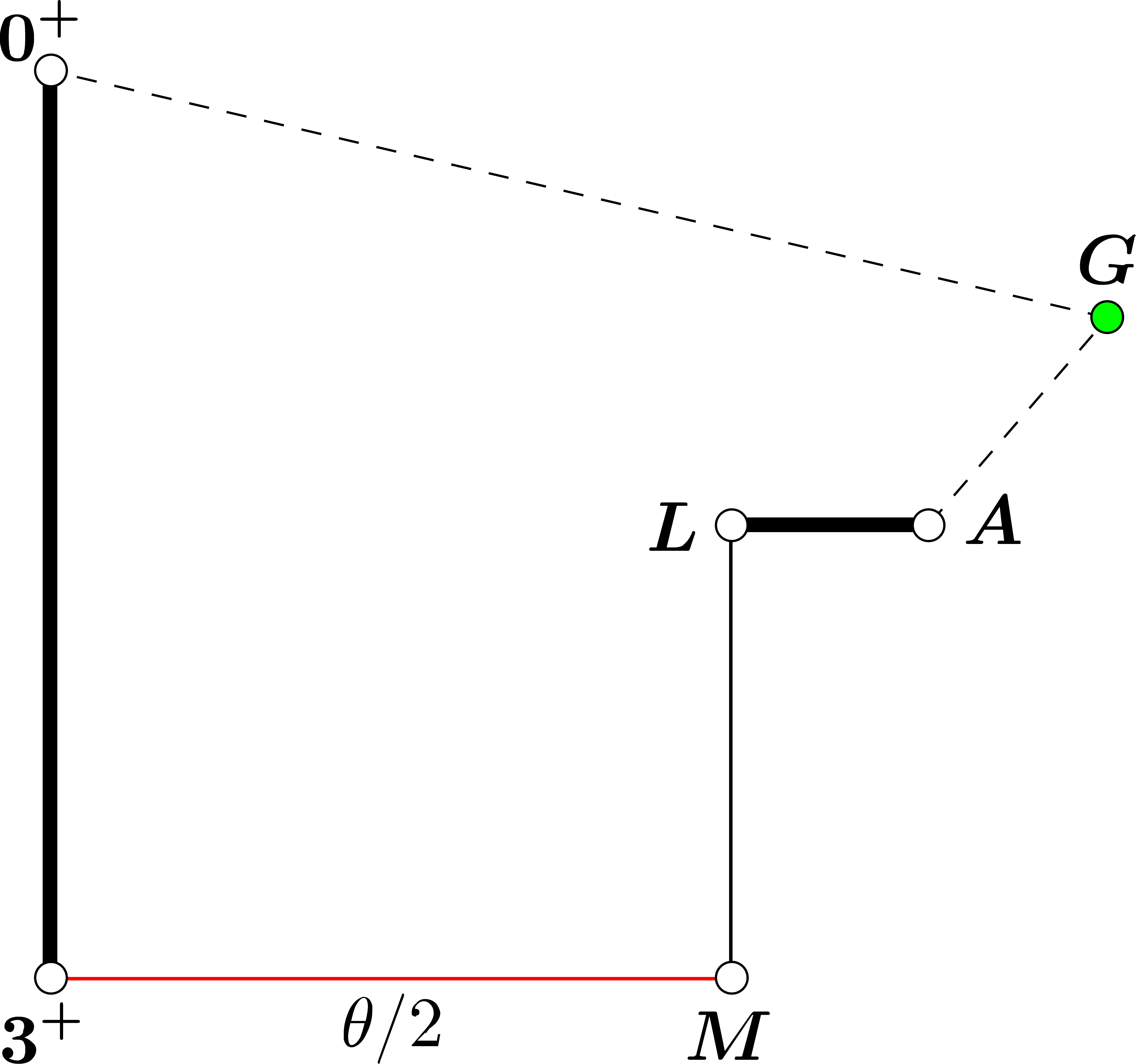}
\qquad
\includegraphics[height=3.8 cm]{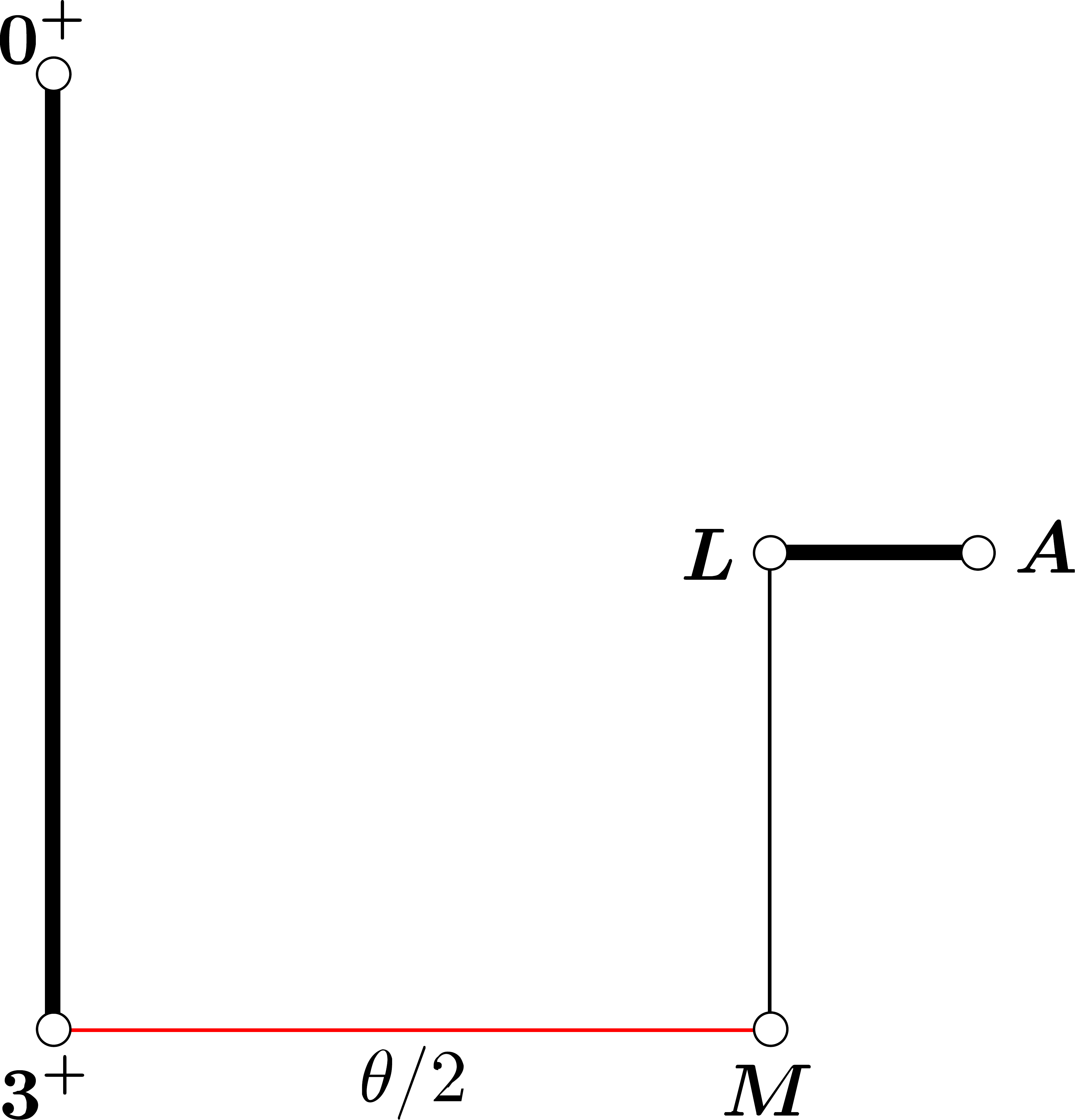}

\vspace{.5 cm}
\labellist
\small\hair 2pt
\pinlabel $\p 3$ at 400 -50
\pinlabel $(t_1,1)$ at 400 -150
\pinlabel $\p 3$ at 1300 -50
\pinlabel $t_1$ at 1300 -150
\pinlabel $\p 3$ at 2200 -50
\pinlabel $(t_2,t_1)$ at 2200 -150
\pinlabel $\p 3$ at 3100 -50
\pinlabel $(0,t_2]$ at 3100 -150
\endlabellist
\centering
\includegraphics[height=3.8 cm]{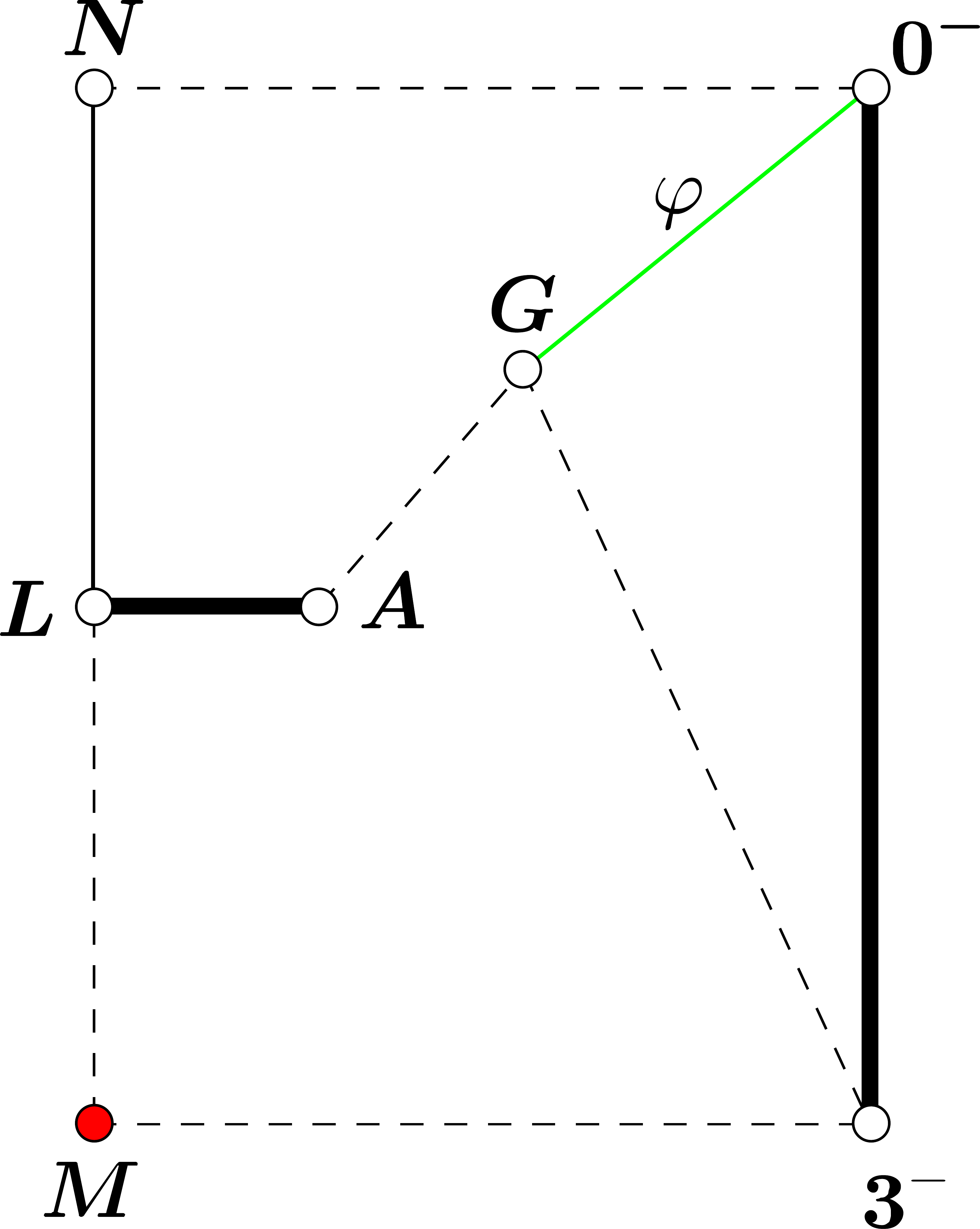}
\includegraphics[height=3.8 cm]{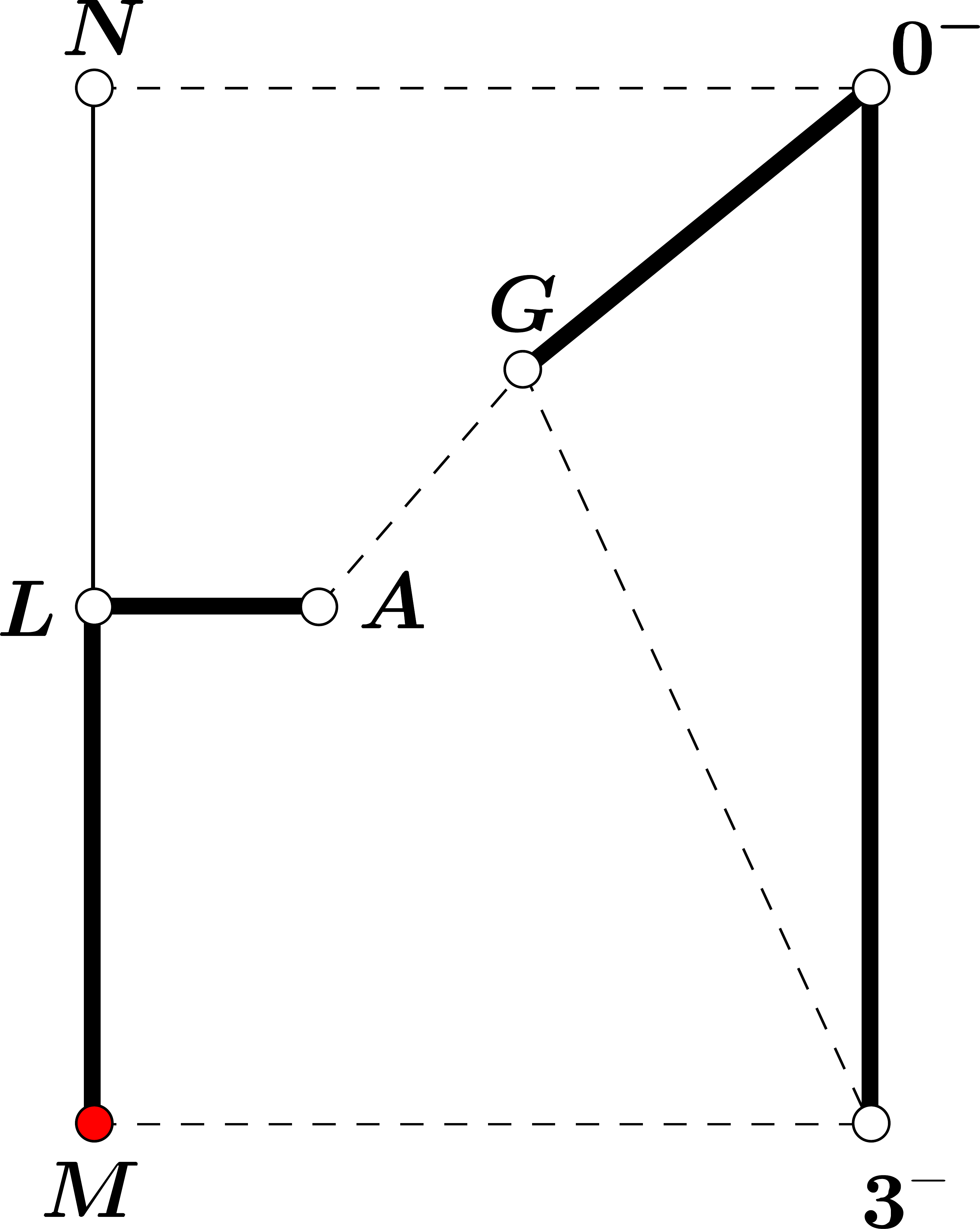}
\includegraphics[height=3.8 cm]{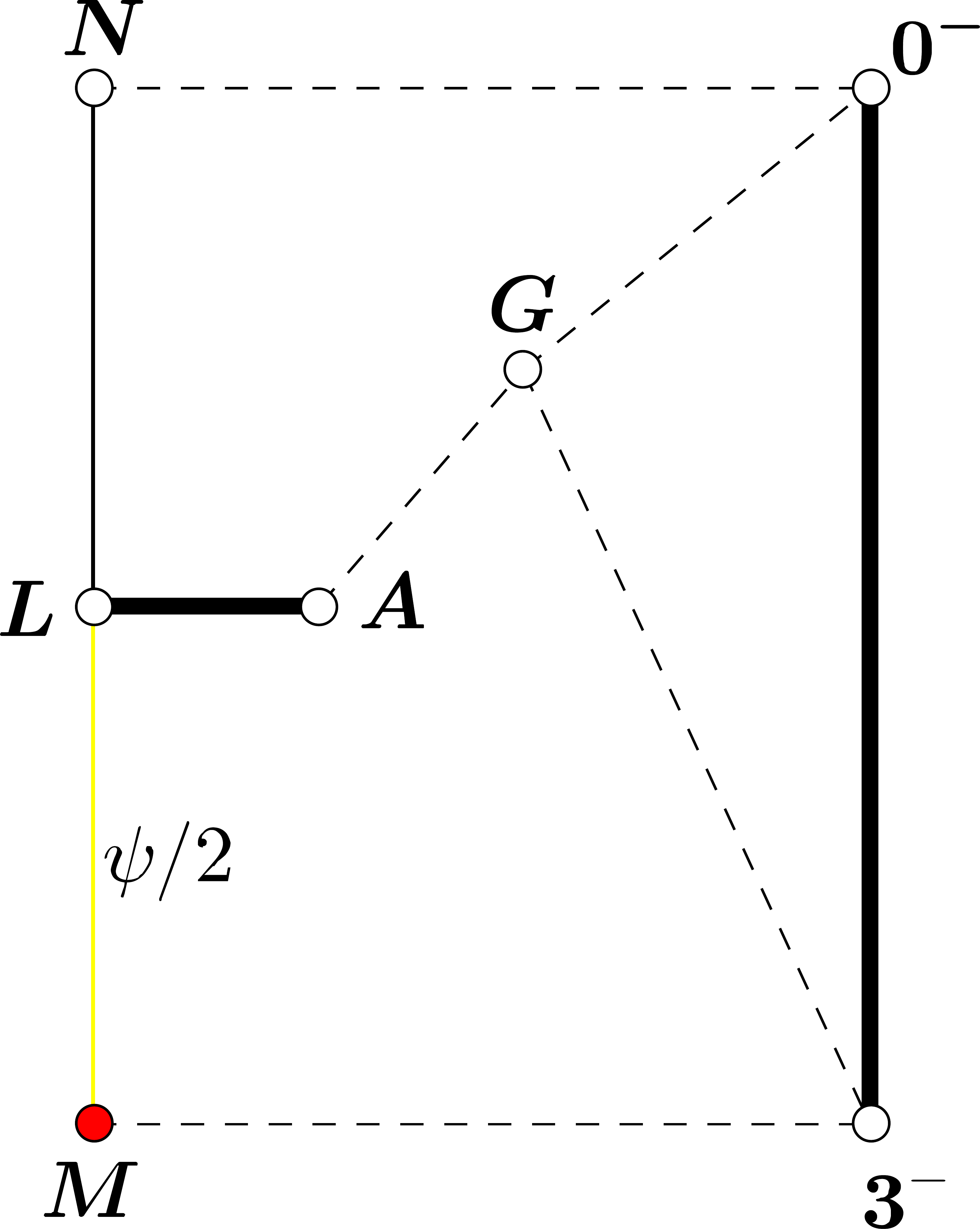}
\includegraphics[height=3.8 cm]{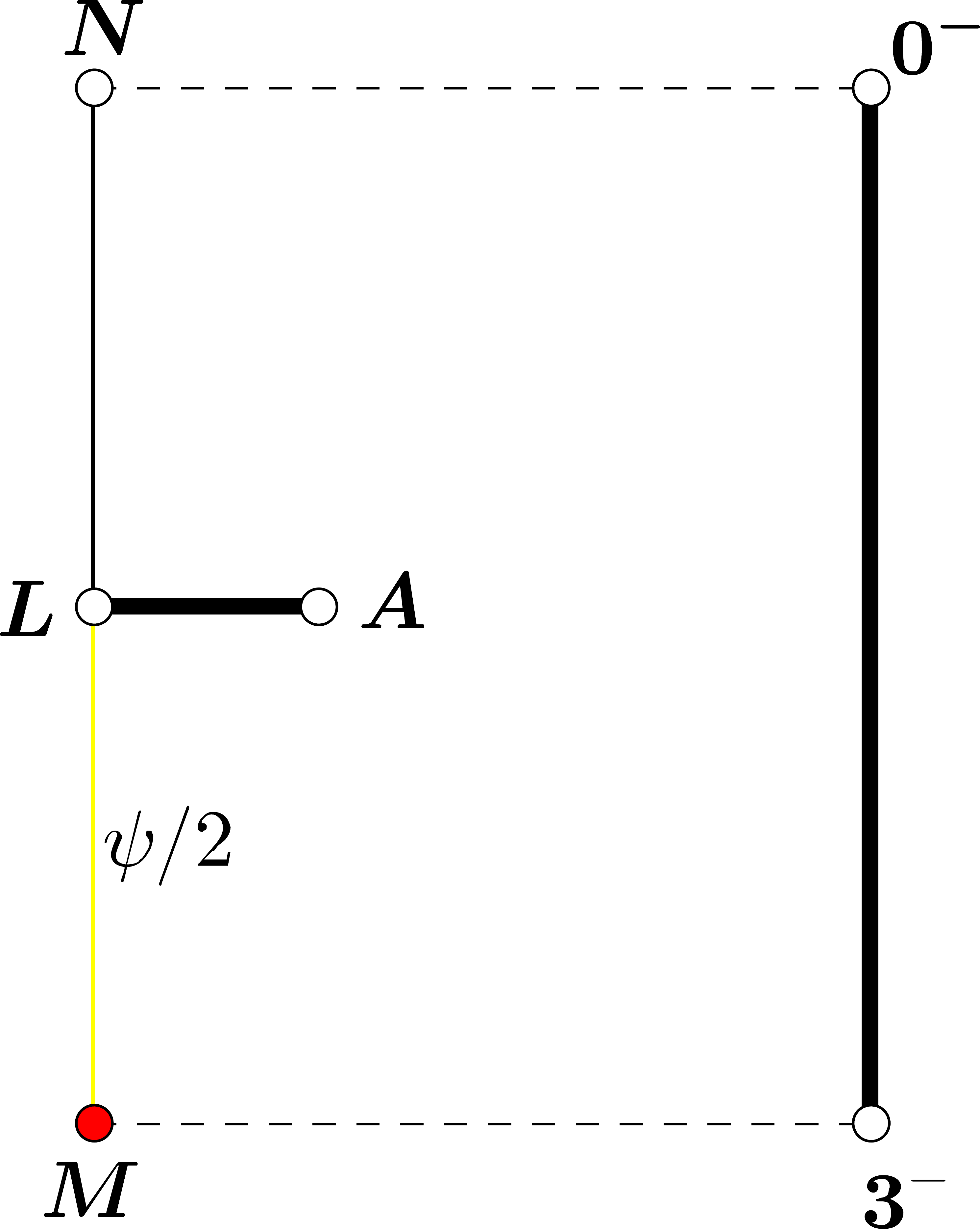}

\vspace{.5 cm}
\nota{The generalised Coxeter diagrams of some walls of $Q_t$. Specifically, that of
$\l A$ for $t\in (0,1)$;
$\l G$ for $t \in (t_1,1)$ and $t\in (t_2,t_1]$;
$\m 0$ for $t \in (t_1,1)$ and $t\in (0, t_1]$;
$\p 3$ for $t \in (t_1,1)$, $t=t_1$, $t\in (t_2,t_1)$, and $t\in (0,t_2]$.
The red and green vertices indicate the (two-dimensional) faces in $Q_t$ with non-right dihedral angles $\frac \theta 2$ and $\varphi$, coherently with Figure \ref{coxeter:fig}. The green, red, and yellow edges indicate the edges of the wall with varying dihedral angle $\varphi$, $\frac \theta 2$, and $\frac \psi 2$.
}\label{coxeter_walls:fig}
\end{figure}

\begin{lemma}\label{quotient_walls_coxeter:lemma}
The generalised Coxeter diagrams of the acute-angled polyhedra $\l A$, $\l G$, $\m 0$, and $\p 3$ are shown in Figure \ref{coxeter_walls:fig} for all $t\in (0,1)$.
The (yellow) dihedral angle $\frac \psi 2$ of $\p 3$ is defined for $t\in (0,t_1]$ and is such that
$$\cos\psi=\frac{\cos\theta}{1-\cos\theta}=\frac{1-3t^2}{2(t^2-1)}.$$
In particular, the angle $\frac\psi 2$ varies strictly monotonically in $t$. Its extremal values are $$\lim_{t\to t_1} \tfrac\psi 2 (t)=0,\quad\lim_{t\to 0}\tfrac\psi 2(t)=\tfrac\pi 3.$$
\end{lemma}
\begin{proof}
For every $\l W\in\lbrace\l A,\ \l G,\ \m 0,\ \p 3\rbrace$ and every time $t$, we construct the Coxeter diagram $D_{\l W,t}$ of $\l W$ at time $t$ following the instructions of Section \ref{acute:subsection}. 

The diagram $D_{\l W,t}$ is built from $D_t$ by removing $\l W$ and all the vertices that are connected to $\l W$ by either a dashed or a thickened edge.
We need then to recompute $\alpha$ from formula (\ref{formula:eqn}) for every pair of vectors. To do so we must substitute each space-like vector 
$$v\in \{\p 0, \m 0, \p 3, \m 3, \l A, \l G, \l H, \l L, \l M, \l N\}$$
with its projection $P(v)$ in the time-like hyperplane $\l W^\perp$, using the formula
$$P( v)=v-\frac{\langle v,\l W\rangle}{\langle\l W,\l W\rangle}\l W.$$
We then calculate the new values of $\alpha$ on every pair $P(v), P(w)$ instead of $v, w$. This will determine the labels on the edges of $D_{\l W, t}$.

Given the abundance of right-angles, in most cases $\alpha$ remains unaffected.
More specifically: 
\begin{itemize}
\item $\l A$ is orthogonal to all the incident walls, hence $P(v)=v$ for every such wall $v$ and all the values $\alpha$ remain unaffected: the diagram $D_{\l  A,t}$ is just a subdiagram of $D_t$ and is shown in Figure \ref{coxeter_walls:fig}-(first line) for all $t\in (0,1)$;
\item $\l G$ is orthogonal to all the incident walls except $\m 0$, which is however orthogonal to all the walls incident to both $\l G$ and $\m 0$: this implies easily that all the values $\alpha \leq 1$ remain unaffected also in this case; hence $D_{\l G,t}$ is just a subdiagram of $D_t$ as in Figure \ref{coxeter_walls:fig}-(second line) for the times $(t_1,1)$ and $(t_2,t_1]$ respectively;
\item $\m 0$ is orthogonal to all the incident walls except $\l G$, which is orthogonal to all the walls incident to both $\m 0$ and $\l G$: again the values $\alpha \leq 1$ are unaffected and $D_{\m 0,t}$ is a subdiagram of $D_t$ as in Figure \ref{coxeter_walls:fig}-(third line) for the times $(t_1,1)$ and $(0,t_1]$ respectively;
\item $\p 3$ is orthogonal to all the incident walls except $\l M$, which is in turn not orthogonal to $\l L$: this is the only label that changes from Figure \ref{coxeter:fig} to \ref{coxeter_walls:fig}, namely that of the edge connecting $\l M$ and $\l L$. We have 
$$P(\l L)=\l L, \qquad P(\l M) = \l M + \frac{2t^2}{t^2+1} \p 3$$
and we easily deduce that
$$\langle P(\l M),P(\l L)\rangle=-1,\quad \langle P(\l M),P(\l M)\rangle=2\frac{1-t^2}{1+t^2},\quad\langle P(\l L),P(\l L)\rangle=2$$
and therefore
$$\alpha = \frac{\sqrt{1+t^2}}{2\sqrt{1-t^2}} =\frac{1+t^2}{2\sqrt{1-t^4}}.$$
In particular:
\begin{itemize}
\item when $t\in (t_1, 1)$ we have $\alpha >1$ and the faces are ultraparallel;
\item when $t=t_1$ we have $\alpha = 1$ and the faces are asymptotically parallel;
\item when $t\in (0,t_1)$ the faces meet at a dihedral angle $\frac \psi 2$ that satisfies
$$\cos\tfrac\psi 2=\frac{\sqrt{1+t^2}}{2\sqrt{1-t^2}} =\frac{1+t^2}{2\sqrt{1-t^4}}.$$
\end{itemize}
The diagram $D_{\p 3, t}$ is shown in Figure \ref{coxeter_walls:fig}-(fourth line) at all times.
\end{itemize}
We note that
$$\cos \psi = 2 \cos^2 \tfrac \psi 2 - 1 = \frac{1+t^2}{2(1-t^2)} -1=\frac{1-3t^2}{2(t^2-1)}.$$
The proof is complete.
\end{proof}

We can now easily draw the walls $\l A$, $\l G$, $\m 0$, and $\p 3$ of $Q_t$ at all times.

\begin{figure}
\labellist
\small\hair 2pt
\pinlabel $\l A$ at 10 10
\pinlabel \textcolor{blue}{$\m 0$} at 10 100
\pinlabel {$\p 3$} at 32 100
\pinlabel \textcolor{blue}{$\p 0$} at 20 60
\pinlabel {$\m 3$} at 29 45
\pinlabel \textcolor{blue}{$\l M$} at 44 80
\pinlabel {$\l N$} at 42 60

\pinlabel $\l G$ at 90 10
\pinlabel $\m 0$ at 120 92
\pinlabel \textcolor{blue}{$\l N$} at 130 45
\pinlabel $\p 3$ at 110 70
\pinlabel $\l L$ at 138 70
\pinlabel \textcolor{blue}{$\l M$} at 127 75

\pinlabel $\m 0$ at 185 10
\pinlabel $\l G$ at 212 23
\pinlabel \textcolor{blue}{$\l M$} at 230 57
\pinlabel $\p 3$ at 215 45
\pinlabel \textcolor{blue}{$\l L$} at 202 57
\pinlabel $\l A$ at 225 78
\pinlabel $\p 0$ at 210 90

\pinlabel $\p 3$ at 275 10
\pinlabel $\l M$ at 320 50
\pinlabel \textcolor{blue}{$\l N$} at 320 75
\pinlabel $\l A$ at 305 73
\pinlabel $\m 0$ at 297 45
\pinlabel \textcolor{blue}{$\l G$} at 303 25
\pinlabel $\m 3$ at 300 90
\pinlabel \textcolor{blue}{$\l L$} at 289 55

\pinlabel $(t_1,1)$ at 340 110
\endlabellist
\centering
\includegraphics[width=12.5 cm]{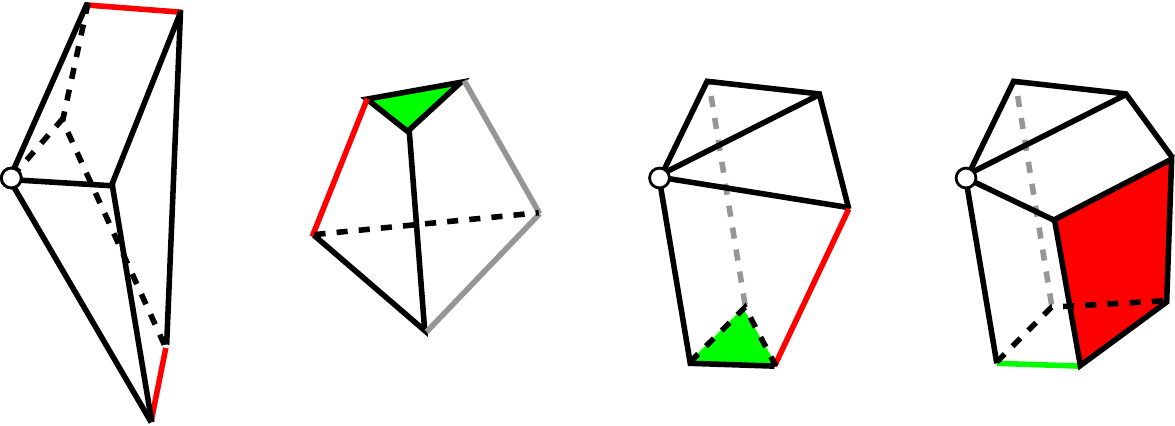}
\vspace{.5 cm}

\labellist
\small\hair 2pt
\pinlabel $\l A$ at 10 10
\pinlabel \textcolor{blue}{$\m 0$} at 10 100
\pinlabel {$\p 3$} at 32 100
\pinlabel \textcolor{blue}{$\p 0$} at 20 60
\pinlabel {$\m 3$} at 29 45
\pinlabel \textcolor{blue}{$\l M$} at 44 80
\pinlabel {$\l N$} at 42 60

\pinlabel $\l G$ at 90 10
\pinlabel \textcolor{blue}{$\l N$} at 130 45
\pinlabel $\p 3$ at 110 70
\pinlabel $\l L$ at 138 70
\pinlabel \textcolor{blue}{$\l M$} at 127 75

\pinlabel $\m 0$ at 185 10
\pinlabel \textcolor{blue}{$\l M$} at 230 57
\pinlabel $\p 3$ at 215 45
\pinlabel \textcolor{blue}{$\l L$} at 202 57
\pinlabel $\l A$ at 225 78
\pinlabel $\p 0$ at 210 90

\pinlabel $\p 3$ at 275 10
\pinlabel $\l M$ at 320 50
\pinlabel \textcolor{blue}{$\l N$} at 320 75
\pinlabel $\l A$ at 305 73
\pinlabel $\m 0$ at 292 45
\pinlabel \textcolor{blue}{$\l G$} at 305 30
\pinlabel $\m 3$ at 300 90
\pinlabel \textcolor{blue}{$\l L$} at 289 55

\pinlabel $t_1$ at 340 110
\endlabellist
\centering
\includegraphics[width=12.5 cm]{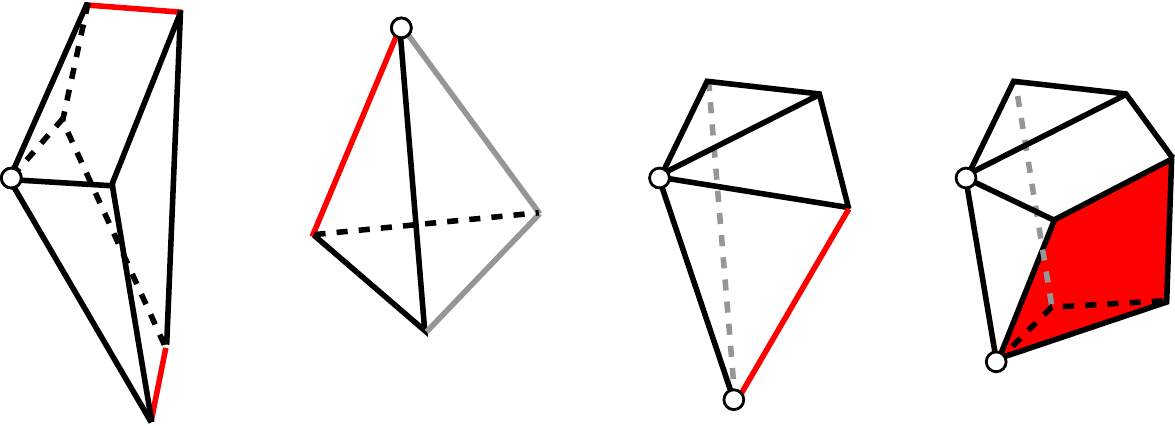}
\vspace{.5 cm}

\labellist
\small\hair 2pt
\pinlabel $\l A$ at 10 10
\pinlabel \textcolor{blue}{$\m 0$} at 10 100
\pinlabel {$\p 3$} at 32 100
\pinlabel \textcolor{blue}{$\p 0$} at 20 60
\pinlabel {$\m 3$} at 29 45
\pinlabel \textcolor{blue}{$\l M$} at 44 80
\pinlabel {$\l N$} at 42 60

\pinlabel $\l G$ at 90 10
\pinlabel \textcolor{blue}{$\l N$} at 130 45
\pinlabel $\p 3$ at 110 70
\pinlabel $\l L$ at 138 70
\pinlabel \textcolor{blue}{$\l M$} at 127 75

\pinlabel $\m 0$ at 185 10
\pinlabel \textcolor{blue}{$\l M$} at 230 57
\pinlabel $\p 3$ at 215 45
\pinlabel \textcolor{blue}{$\l L$} at 202 57
\pinlabel $\l A$ at 225 78
\pinlabel $\p 0$ at 210 90

\pinlabel $\p 3$ at 275 10
\pinlabel $\l M$ at 320 50
\pinlabel \textcolor{blue}{$\l N$} at 320 75
\pinlabel $\l A$ at 305 73
\pinlabel $\m 0$ at 290 55
\pinlabel \textcolor{blue}{$\l G$} at 307 29
\pinlabel $\m 3$ at 300 90
\pinlabel \textcolor{blue}{$\l L$} at 293 40

\pinlabel $(t_2,t_1)$ at 340 110
\endlabellist
\centering
\includegraphics[width=12.5 cm]{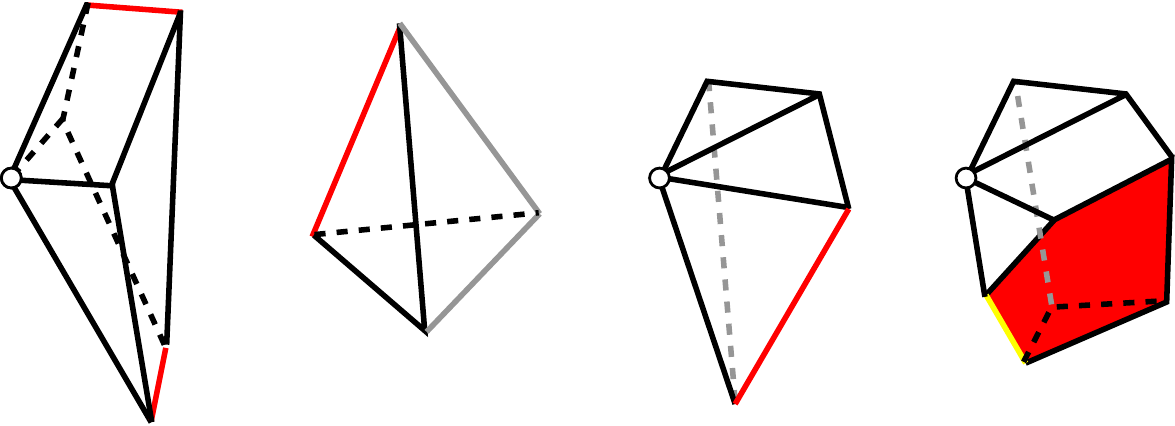}
\vspace{.5 cm}

\labellist
\small\hair 2pt
\pinlabel $\p 3$ at -5 0
\pinlabel \textcolor{blue}{$\l L$} at 30 20
\pinlabel \textcolor{blue}{$\l N$} at 50 50
\pinlabel $\l M$ at 50 30
\pinlabel $\l A$ at 25 45
\pinlabel $\m 0$ at 15 28
\pinlabel $\m 3$ at 22 62

\pinlabel $\p 3$ at 105 0
\pinlabel \textcolor{blue}{$\l L$} at 140 20
\pinlabel \textcolor{blue}{$\l N$} at 160 50
\pinlabel $\l M$ at 160 30
\pinlabel $\l A$ at 135 45
\pinlabel $\m 0$ at 125 28
\pinlabel $\m 3$ at 132 62

\pinlabel $t_2$ at 70 70
\pinlabel $(0,t_2)$ at 180 70
\endlabellist
\centering
\includegraphics[width=6.5 cm]{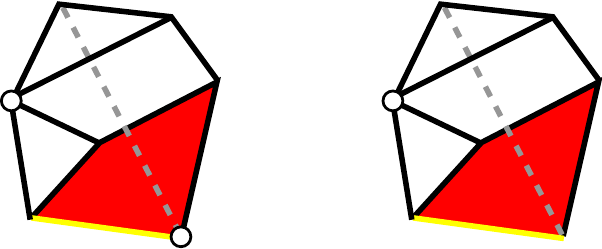}
\vspace{.3 cm}
\nota{The walls $\l A$, $\l G$, $\m 0$, and $\p 3$ of the quotient polytope $Q_t$ at the times $(t_1,1)$ in the first line, $t_1$  in the second line, and $(t_2,t_1)$ in the third line. The combinatorics of $\l A$ and $\m 0$ is constant in $(0, t_1)$, but that of $\p 3$ changes further at the times $t_2$ and $(0,t_2)$ as shown in the fourth line. Every face is labeled with the name of the adjacent wall: front faces are labeled in black, and back faces in blue. On each wall, the red, green, black, grey, and yellow edges have dihedral angle respectively $\frac \theta 2$, $\varphi$, $\frac \pi 2$, $\frac \pi 3$, and $\frac\psi 2$. Similarly, on the polytope $Q_t$ the red, green, and white faces have dihedral angle $\frac \theta 2$, $\varphi$, and $\frac \pi 2$. The ideal vertices are indicated as white dots. }\label{Q_walls:fig}
\end{figure}

\begin{cor}
The combinatorics and geometry of the polyhedra $\l A$, $\l G$, $\m 0$, and $\p 3$ of $Q_t$ is shown in Figure \ref{Q_walls:fig}. In particular, they all have finite volume.
\end{cor}
\begin{proof}
All the strata of each acute-angled polyhedron are easily deduced from its corresponding Coxeter diagram, using the algorithms described in Section \ref{acute:subsection}, that allow one to determine first the edges and then the vertices of each polyhedron. 

Recall in particular that every finite vertex arises from a triple of nodes of the Coxeter diagram of elliptic type, and every ideal vertex arises from a triple or 4-uple of vertices of Euclidean type. The reader is invited to check that the vertices are those shown in Figure \ref{Q_walls:fig}, and in particular the crucial fact that every edge has two vertices as its endpoints: hence the polyhedra have all finite volume (there are no hyperideal vertices, see Theorem \ref{Vinberg:teo}).

For instance, one checks that the polyhedron $\l A$ contains $6$ finite vertices, that correspond to elliptic Coxeter subdiagram with tree nodes, and an ideal vertex, that corresponds to the Euclidean Coxeter subgraph with four nodes $\lbrace\m0,\p0,\m3,\p3\rbrace$, that represents a rectangle. 

Similarly, the polyhedron $\l G$ contains some finite vertices, and one ideal vertex only at the time $t=t_1$ corresponding to the subdiagram with nodes $\lbrace\p 3,\l L,\l M\rbrace$, which represents a Euclidean triangle with angles $\frac \pi 2$, $\frac \pi 3$, and $\frac \theta 2 = \frac \pi 6$. When $t<t_1$ we get $\frac \theta 2 > \frac \pi 6$ and the triple represents a finite vertex instead. The polyhedra $\m 0$ and $\p 3$ are treated similarly.
\end{proof}

Figure \ref{Q_walls:fig} shows both the four-dimensional dihedral angles along the faces and the three-dimensional dihedral angles of the single walls along the edges: on each wall, the red, green, black, grey, and yellow edges have dihedral angle respectively $\frac \theta 2$, $\varphi$, $\frac \pi 2$, $\frac \pi 3$, and $\frac \psi 2$. Similarly, on the polytope $Q_t$ the red, green, and white faces have dihedral angle $\frac \theta 2$, $\varphi$, and $\frac \pi 2$. The ideal vertices are indicated as white dots.

\begin{cor}
The polytope $Q_t$ has finite volume for all $ t\in (0,1]$. Its combinatorics is constant on each of the time intervals
$$(0,t_2), \quad (t_2,t_1), \quad (t_1,1)$$
and changes precisely at the critical times $t_2$, $t_1$, and $1$.
\end{cor}
\begin{proof}
We only need to prove that $Q_t$ has finite volume. By Theorem \ref{Vinberg:teo} it suffices to check that every edge of $Q_t$ has two (finite or ideal) vertices as endpoints. All the edges that belong to one of the walls $\l A$, $\l G$, $\m 0$, or $\p 3$ have this property, as already checked. There is yet one last edge to investigate in Figure \ref{coxeter:fig}, determined by the triple
$\lbrace\l L,\l M,\l N\rbrace$. That edge joins the finite vertices $\{ \l L, \l M, \l N, \l G\}$ and $\{ \l L, \l M, \l N, \l H\}$ when $t > t_2$, and the vertices $\{ \l L, \l M, \l N, \p 3\}$ and $\{ \l L, \l M, \l N, \p 0\}$ when $t \leq t_2$, which are ideal at $t=t_2$ and finite when $t<t_2$.
\end{proof}

We now finally use all the information that we gathered on the quotient polytope $Q_t$ to analyse the original polytope $P_t$.

\subsection{Back to the original polytope $P_t$}\label{polytope:sec}
We recall that $P_t$ has 24 walls when $t>t_2$ and 22 when $t\leq t_2$, and up to the action of its symmetry group these walls reduce to four elements only:
$$\{\p 3, \m 0, \l A, \l G\}$$
where $\l G$ exists only for $t>t_2$.
We start by showing the following.
\begin{prop}\label{combinatoria:prop}
For all $t\in(0,1]$, the polytope $P_t$ has finite volume.
Moreover, its combinatorics is constant on each of the time intervals
$$(0,t_2), \quad (t_2,t_1), \quad (t_1,1)$$
and changes precisely at the critical times $t_2$, $t_1$, and $1$.
The combinatorics and geometry of the walls $\p 3, \m 0, \l A, \l G$
is fully described in Figures \ref{walls:fig}, \ref{walls1:fig}, \ref{walls_medi:fig}, \ref{walls_plus:fig}.
\end{prop}
\begin{proof}
The walls of $P_t$ are obtained by mirroring the corresponding walls of $Q_t$ from Figure \ref{Q_walls:fig} along the faces $\l L$, $\l M$, and $\l N$.
\end{proof}

\begin{figure}
\labellist
\small\hair 2pt
\pinlabel $\l A$ at 80 0
\pinlabel \textcolor{blue}{$\m 0$} at 24 104
\pinlabel $\p 3$ at 50 90
\pinlabel $\m 2$ at 100 106
\pinlabel \textcolor{blue}{$\p 1$} at 60 110
\pinlabel \textcolor{blue}{$\m 1$} at 60 80
\pinlabel \textcolor{blue}{$\p 0$} at 30 50
\pinlabel \textcolor{blue}{$\p 2$} at 70 50
\pinlabel $\m 3$ at 54 52

\pinlabel $\l G$ at 252 0
\pinlabel $\p 7$ at 192 70
\pinlabel $\p 5$ at 232 42
\pinlabel \textcolor{blue}{$\p 1$} at 202 40
\pinlabel \textcolor{blue}{$\p 3$} at 222 72
\pinlabel \textcolor{blue}{$\m 2$} at 167 53
\pinlabel \textcolor{blue}{$\m 0$} at 253 67
\pinlabel $\m 6$ at 215 18
\pinlabel $\m 4$ at 205 104
\endlabellist
\centering
\includegraphics[width=12 cm]{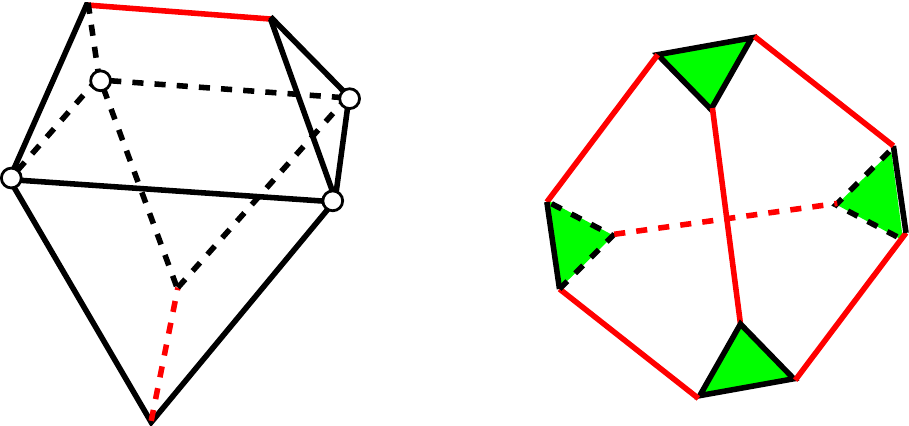}
\vspace{1 cm}

\labellist
\small\hair 2pt
\pinlabel $\m 0$ at 80 0
\pinlabel \textcolor{blue}{$\l A$} at 37 83
\pinlabel $\p 0$ at 38 92
\pinlabel $\l B$ at 16 75
\pinlabel $\l C$ at 60 75
\pinlabel \textcolor{blue}{$\p 3$} at 24 54
\pinlabel \textcolor{blue}{$\p 1$} at 50 54
\pinlabel $\p 5$ at 31 42

\pinlabel $\p 3$ at 250 0
\pinlabel \textcolor{blue}{$\G$} at 36 32
\pinlabel $\l B$ at 150 83
\pinlabel $\l A$ at 227 83
\pinlabel \textcolor{blue}{$\l D$} at 189 88
\pinlabel $\m 3$ at 189 97
\pinlabel $\m 0$ at 189 37
\pinlabel \textcolor{blue}{$\p 7$} at 189 60
\pinlabel $\p 5$ at 150 30
\pinlabel $\p 1$ at 227 30
\pinlabel \textcolor{blue}{$\m 4$} at 142 60
\pinlabel \textcolor{blue}{$\m 2$} at 235 60
\pinlabel \textcolor{blue}{$\G$} at 189 20
\endlabellist
\centering
\includegraphics[width=11 cm]{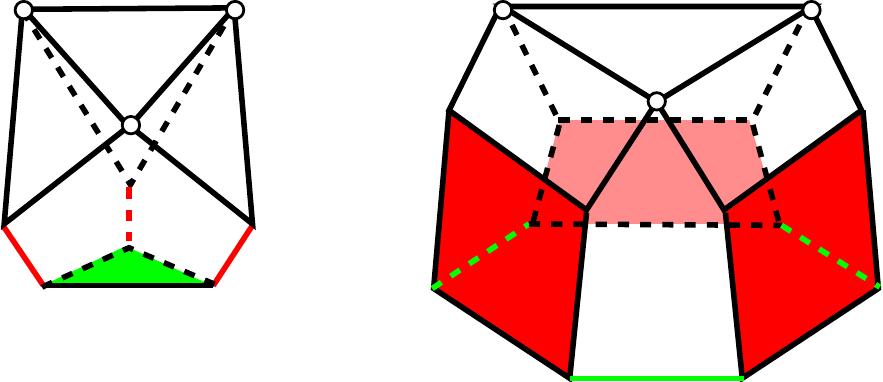}
\nota{Combinatorial pictures of the walls $\l A$, $\l G$, $\m 0$, and $\p 3$ of $P_t$ at the times $t\in(t_1,1)$. Every face is labeled with the name of the adjacent wall: front faces are labeled in black, and back faces in blue. On each wall, the red, green, and black edges have dihedral angle respectively $\theta$, $\varphi$, and $\frac \pi 2$. Similarly, on the polytope $P_t$ the red, green, and white faces have dihedral angle $\theta$, $\varphi$, and $\frac \pi 2$. The ideal vertices are indicated as white dots.}\label{walls:fig}
\end{figure}

\begin{figure}
\labellist
\small\hair 2pt
\pinlabel $\l A$ at 80 0
\pinlabel \textcolor{blue}{$\m 0$} at 24 104
\pinlabel $\p 3$ at 50 90
\pinlabel $\m 2$ at 100 106
\pinlabel \textcolor{blue}{$\p 1$} at 60 110
\pinlabel \textcolor{blue}{$\m 1$} at 60 80
\pinlabel \textcolor{blue}{$\p 0$} at 30 50
\pinlabel \textcolor{blue}{$\p 2$} at 70 50
\pinlabel $\m 3$ at 54 52

\pinlabel $\l G$ at 252 0
\pinlabel $\p 7$ at 192 70
\pinlabel $\p 5$ at 227 42
\pinlabel \textcolor{blue}{$\p 1$} at 202 40
\pinlabel \textcolor{blue}{$\p 3$} at 222 72
\endlabellist
\centering
\includegraphics[width=12 cm]{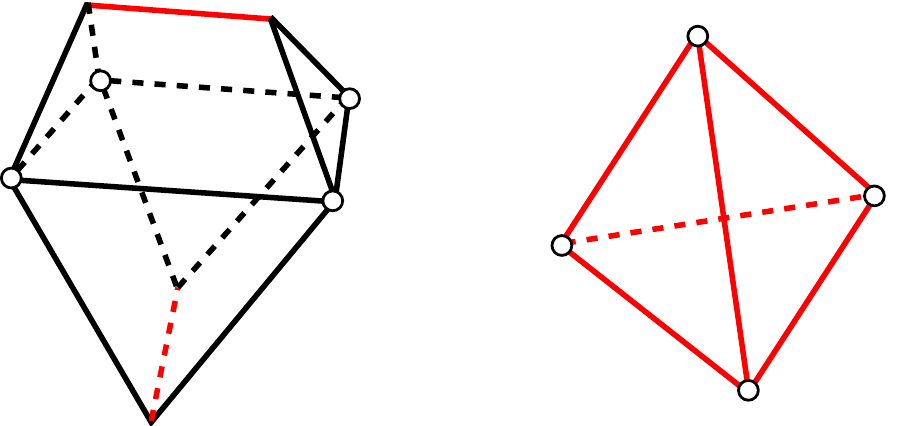}
\vspace{1 cm}

\labellist
\small\hair 2pt
\pinlabel $\m 0$ at 80 0
\pinlabel \textcolor{blue}{$\l A$} at 37 83
\pinlabel $\p 0$ at 38 92
\pinlabel $\l B$ at 16 75
\pinlabel $\l C$ at 60 75
\pinlabel \textcolor{blue}{$\p 3$} at 24 54
\pinlabel \textcolor{blue}{$\p 1$} at 50 54
\pinlabel $\p 5$ at 31 42

\pinlabel $\p 3$ at 250 0
\pinlabel $\l B$ at 163 85
\pinlabel $\l A$ at 215 85
\pinlabel \textcolor{blue}{$\l D$} at 189 88
\pinlabel $\m 3$ at 189 97
\pinlabel $\m 0$ at 189 34
\pinlabel \textcolor{blue}{$\p 7$} at 189 58
\pinlabel $\p 5$ at 160 47
\pinlabel $\p 1$ at 217 47
\pinlabel \textcolor{blue}{$\m 4$} at 150 75
\pinlabel \textcolor{blue}{$\m 2$} at 232 75
\pinlabel \textcolor{blue}{$\G$} at 189 20
\endlabellist
\centering
\includegraphics[width=11 cm]{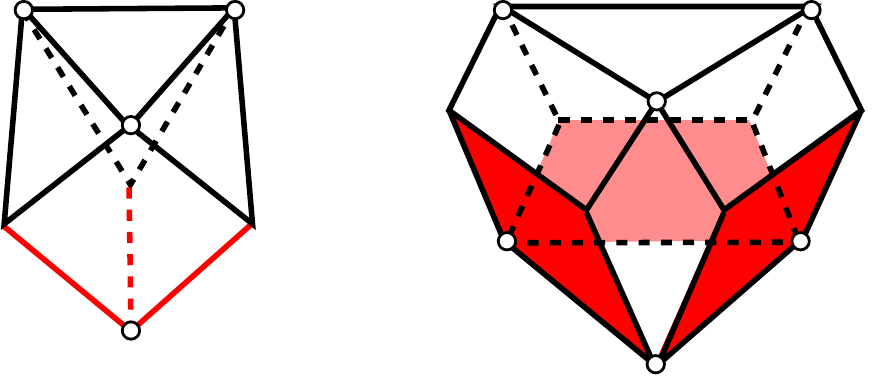}
\nota{Combinatorial pictures of the walls $\l A$, $\l G$, $\m 0$, and $\p 3$ at the critical time $t_1$. We use the same notations as in Figure \ref{walls:fig}. The dihedral angles are either $\frac \pi 3$ (on the red faces and edges) or $\frac \pi 2$ (on the rest).}\label{walls1:fig}
\end{figure}

\begin{figure}
\labellist
\small\hair 2pt
\pinlabel $\l A$ at 80 0
\pinlabel \textcolor{blue}{$\m 0$} at 24 104
\pinlabel $\p 3$ at 50 90
\pinlabel $\m 2$ at 100 106
\pinlabel \textcolor{blue}{$\p 1$} at 60 110
\pinlabel \textcolor{blue}{$\m 1$} at 60 80
\pinlabel \textcolor{blue}{$\p 0$} at 30 50
\pinlabel \textcolor{blue}{$\p 2$} at 70 50
\pinlabel $\m 3$ at 54 52

\pinlabel $\l G$ at 252 0
\pinlabel $\p 7$ at 192 70
\pinlabel $\p 5$ at 227 42
\pinlabel \textcolor{blue}{$\p 1$} at 202 40
\pinlabel \textcolor{blue}{$\p 3$} at 222 72
\endlabellist
\centering
\includegraphics[width=12 cm]{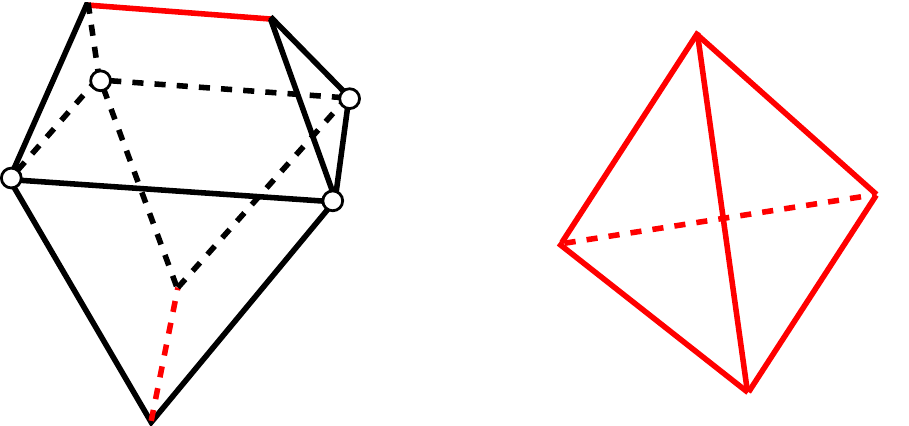}
\vspace{1 cm}

\labellist
\small\hair 2pt
\pinlabel $\m 0$ at 80 0
\pinlabel \textcolor{blue}{$\l A$} at 37 83
\pinlabel $\p 0$ at 38 92
\pinlabel $\l B$ at 16 75
\pinlabel $\l C$ at 60 75
\pinlabel \textcolor{blue}{$\p 3$} at 24 54
\pinlabel \textcolor{blue}{$\p 1$} at 50 54
\pinlabel $\p 5$ at 31 42

\pinlabel $\p 3$ at 250 0
\pinlabel $\l B$ at 163 85
\pinlabel $\l A$ at 215 85
\pinlabel \textcolor{blue}{$\l D$} at 189 88
\pinlabel $\m 3$ at 189 97
\pinlabel $\m 0$ at 189 50
\pinlabel \textcolor{blue}{$\p 7$} at 189 58
\pinlabel $\p 5$ at 160 40
\pinlabel $\p 1$ at 217 40
\pinlabel \textcolor{blue}{$\m 4$} at 150 75
\pinlabel \textcolor{blue}{$\m 2$} at 232 75
\pinlabel \textcolor{blue}{$\G$} at 195 15
\endlabellist
\centering
\includegraphics[width=11 cm]{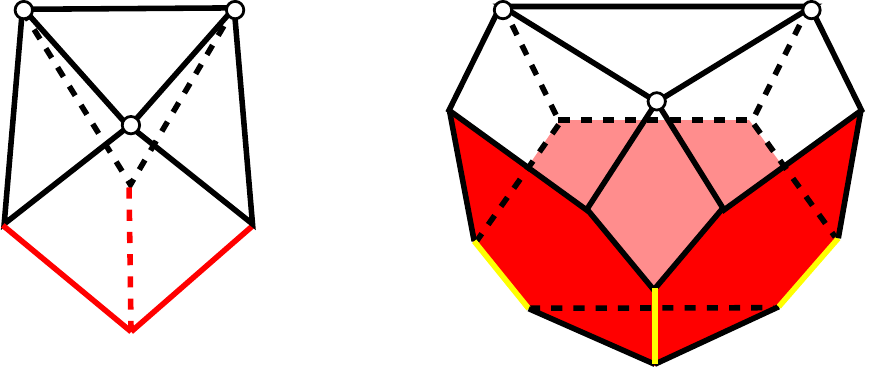}
\nota{Combinatorial pictures of the walls $\l A$, $\l G$, $\m 0$, and $\p 3$ at the times $t \in(t_2,t_1)$. We use the same notations as in Figure \ref{walls:fig}. The dihedral angles are either $\theta$ (on the red faces and edges), $\psi$ (on the yellow edges) or $\frac \pi 2$ (on the rest).}\label{walls_medi:fig}
\end{figure}

\begin{figure}
\labellist
\small\hair 2pt
\pinlabel $\p 3$ at 122 0
\pinlabel $\l B$ at 35 85
\pinlabel $\l A$ at 97 85
\pinlabel \textcolor{blue}{$\l D$} at 61 88
\pinlabel $\m 3$ at 61 97
\pinlabel $\m 0$ at 61 50
\pinlabel \textcolor{blue}{$\p 7$} at 61 58
\pinlabel $\p 5$ at 31 40
\pinlabel $\p 1$ at 89 40
\pinlabel \textcolor{blue}{$\m 4$} at 22 75
\pinlabel \textcolor{blue}{$\m 2$} at 104 75

\pinlabel $\p 3$ at 257 0
\pinlabel $\l B$ at 170 85
\pinlabel $\l A$ at 232 85
\pinlabel \textcolor{blue}{$\l D$} at 196 88
\pinlabel $\m 3$ at 196 97
\pinlabel $\m 0$ at 196 50
\pinlabel \textcolor{blue}{$\p 7$} at 196 58
\pinlabel $\p 5$ at 166 40
\pinlabel $\p 1$ at 224 40
\pinlabel \textcolor{blue}{$\m 4$} at 157 75
\pinlabel \textcolor{blue}{$\m 2$} at 239 75
\endlabellist
\centering
\includegraphics[width=11 cm]{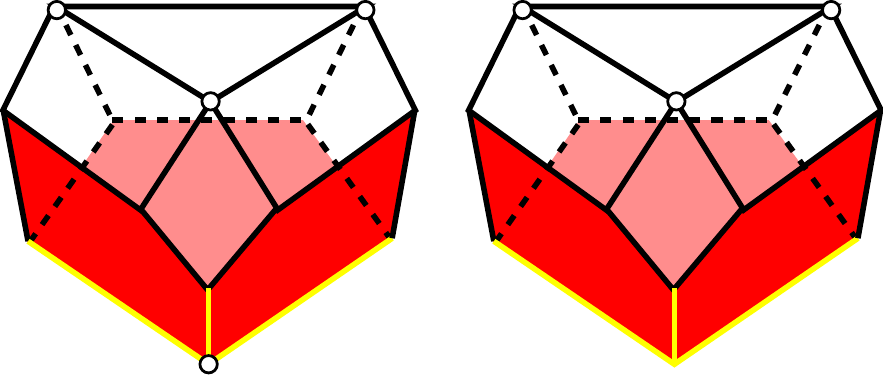}
\nota{Combinatorial pictures of the wall $\p 3$ at the critical time $t_2$ and in the interval $t\in(0,t_2)$. We use the same notations as in Figure \ref{walls:fig}. The four-dimensional and three-dimensional dihedral angles are either $\theta$ (on the red faces), $\psi$ (on the yellow edges) or $\frac \pi 2$ (on the rest). The only non-right angle of each red pentagon is at the bottom vertex. The only difference between the two figures is the bottom vertex which is either ideal (left) or finite (right). }\label{walls_plus:fig}
\end{figure}

\begin{figure}
\vspace{.5 cm}
\labellist
\small\hair 2pt
\pinlabel $(t_1,1)$ at 90 550 
\pinlabel $t_1$ at 250 550 
\pinlabel $(t_2,t_1)$ at 405 550 
\pinlabel $t_2$ at 565 550 
\pinlabel $(0,t_2)$ at 720 550 
\pinlabel $\l A$ at 150 403 
\pinlabel $\l G$ at 150 280 
\pinlabel $\m 0$ at 150 160 
\pinlabel $\p 3$ at 150 20 
\endlabellist
\centering
\includegraphics[width=12.5 cm]{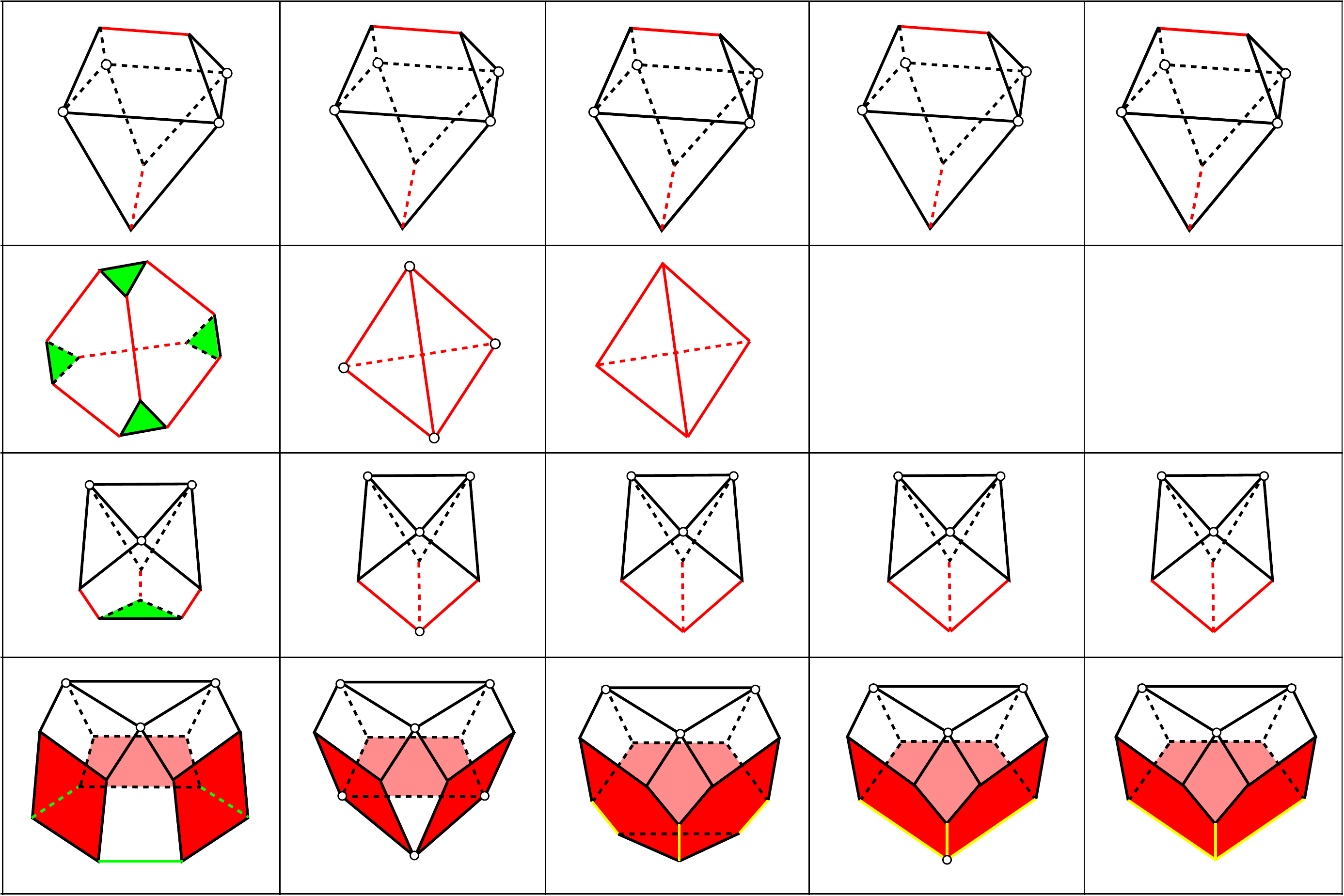}
\nota{An overview of the combinatorics of the evolving walls. At the initial time $t=1$ all the walls are regular ideal octahedra.}\label{walls_overview:fig}
\end{figure}

\begin{figure}
\labellist
\small\hair 2pt
\pinlabel $(t_1,1)$ at 90 300 
\pinlabel $t_1$ at 270 300 
\pinlabel $(t_2,t_1)$ at 445 300 
\pinlabel $t_2$ at 620 300
\pinlabel $(0,t_2)$ at 790 300 

\pinlabel $\theta$ at 62 190 
\pinlabel $\theta$ at 116 190 
\pinlabel $\theta$ at 88 236

\pinlabel $\varphi$ at 50 102
\pinlabel $\varphi$ at 135 102
\pinlabel $\ell$ at 92 28

\pinlabel $\ell$ at 272 28

\pinlabel $\ell$ at 452 28
\pinlabel $d$ at 452 135

\pinlabel $\ell$ at 620 28

\pinlabel $\eta$ at 795 114
\pinlabel $\ell$ at 795 28
\endlabellist
\centering
\vspace{.5 cm}
\includegraphics[width=12 cm]{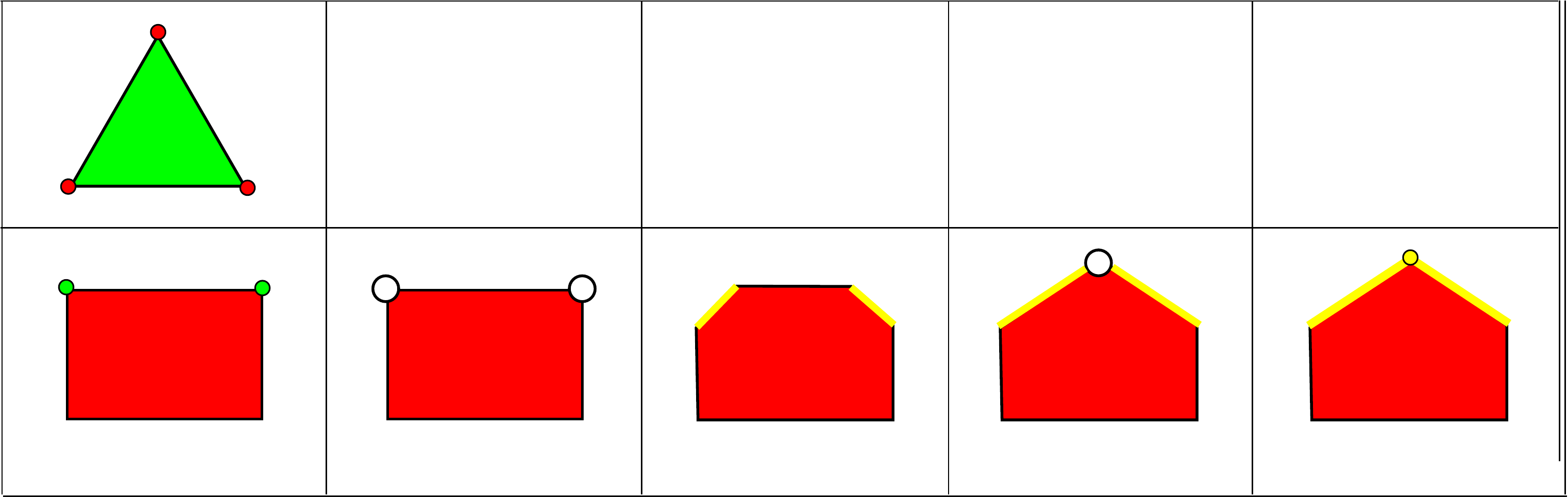}
\nota{The evolving green and red faces. The green face is an equilateral triangle and exists only for $t\in(t_1,1)$. Vertices with big white dots are ideal. All the finite vertices are right-angled, except those labeled with some explicit angle $\theta$, $\varphi$, or $\eta$. The angles $\theta, \varphi, \eta$ and the lengths $\ell$, $d$ depend on $t$. All the red faces are symmetric with respect to a vertical axis. A small green or red dot indicates the presence of an incident green or red face.}\label{polygons0:fig}
\end{figure}

The figures show both the four-dimensional dihedral angles along the faces and the three-dimensional dihedral angles of the single walls along the edges.
An overview of the evolving walls is shown in Figure \ref{walls_overview:fig}.

\subsection*{Dihedral angles}
A remarkable aspect of the deformation $P_t$ is that most of the dihedral angles stay constantly right during the whole process.  
In the following proposition we denote a face of $P_t$ as a pair of intersecting walls.

\begin{prop} \label{dihedral:prop}
All the faces of $P_t$ have right dihedral angles, except:
\begin{itemize}
\item the 8 green triangles
$$ \{\l{G}, \m 0\}, \  \{\l{G}, \m 2\}, \ \{\l{G}, \m 4\}, \ \{\l{G}, \m 6\},\  
\{\l{H}, \m 1\}, \  \{\l{H}, \m 3\}, \ \{\l{H}, \m 5\}, \ \{\l{H}, \m 7\}$$
have dihedral angle $\varphi$ when $t\in(t_1,1)$,
\item the 12 red polygons
$$\{\p{1},\p{3}\},\{\p{3},\p{5}\},\{\p{5},\p{7}\},\{\p{7},\p{1}\},\{\p{1},\p{5}\},\{\p{3},\p{7}\},$$
$$\{\p{2},\p{0}\},\{\p{0},\p{4}\}, \{\p{4},\p{6}\},\{\p{6},\p{2}\},\{\p{2},\p{4}\},\{\p{0},\p{6}\}$$  
have dihedral angle $\theta$ for all $t\in (0,1)$.
\end{itemize}
The evolution of the green and red faces is shown in Figure \ref{polygons0:fig}.
\end{prop}

It is remarkable that for all $t\in(t_1,1)$ the non right-angled faces intersect only in pairs at some vertices. Where this happens, the dihedral angle $\varphi$ or $\theta$ of one face equals the interior angle of the other, see Figure \ref{polygons0:fig}.

\begin{cor}
The polytope $P_t$ is acute-angled precisely when $t\geq \bar t = \sqrt{\frac 13}$. 
\end{cor}

The polytope $P_{\bar t}$ is right-angled. We will soon determine the Coxeter polytopes in the family $P_t$.

\subsection*{Simple polytopes}
During our analysis we have also proved the following.

\begin{prop}
The polytope $P_t$ is simple for all $t\in (0,1]$.
\end{prop}
\begin{proof}
The polytope $P_t$ is acute-angled and hence \cite[Section 3]{Vin}
simple for all $t \geq \sqrt {\frac 13}$. If $t < \sqrt{\frac 13}$ the polytope $P_t$ has the same combinatorics of $P_{t_2-\varepsilon}$ and is hence also simple.
\end{proof}

We are now interested in the links of the vertices of the polytope $P_t$. The initial polytope $P_1$ is the ideal 24-cell: it has 24 ideal vertices, each with a Euclidean cube as a link. We now study separately the first time interval $(t_1,1)$, the first critical time $t_1$, the second time interval $(t_2,t_1)$, and the last time interval $(0,t_2]$. (The discussion for $(0,t_2]$ also includes the second critical time $t_2$.)

\begin{figure} 
\labellist
\small\hair 2pt

\pinlabel \rotatebox{90}{$I_{\frac\pi2}*I_\theta$} at 700 300
\pinlabel \rotatebox{90}{\textcolor{blue}{N}} at 450 400
\pinlabel \rotatebox{90}{L} at 300 300
\pinlabel \rotatebox{90}{P} at 200 470
\pinlabel \rotatebox{90}{\textcolor{blue}{P}} at 150 300
\pinlabel \rotatebox{90}{(1)} at 700 550

\pinlabel \rotatebox{90}{$I_\varphi*I_\theta$} at 700 960
\pinlabel \rotatebox{90}{\textcolor{blue}{N}} at 450 1060
\pinlabel \rotatebox{90}{L} at 300 960
\pinlabel \rotatebox{90}{P} at 200 1130
\pinlabel \rotatebox{90}{\textcolor{blue}{P}} at 150 960
\pinlabel \rotatebox{90}{(2)} at 700 1200

\pinlabel \rotatebox{90}{\textcolor{blue}{P}} at 450 1760
\pinlabel \rotatebox{90}{P} at 300 1660
\pinlabel \rotatebox{90}{L} at 200 1830
\pinlabel \rotatebox{90}{\textcolor{blue}{P}} at 150 1660
\pinlabel \rotatebox{90}{(3)} at 700 1900

\pinlabel \rotatebox{90}{$\Delta^3(\theta)$} at 700 2360
\pinlabel \rotatebox{90}{\textcolor{blue}{P}} at 450 2460
\pinlabel \rotatebox{90}{P} at 300 2360
\pinlabel \rotatebox{90}{P} at 200 2530
\pinlabel \rotatebox{90}{\textcolor{blue}{P}} at 150 2360
\pinlabel \rotatebox{90}{(4)} at 700 2600

\endlabellist

\centering

\rotatebox{-90}{\includegraphics[width=2.5 cm]{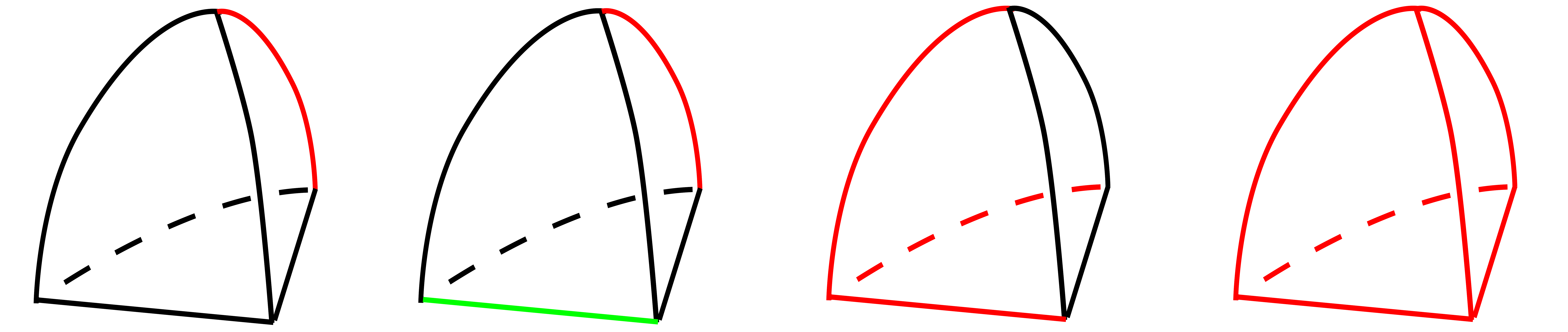}}

\vspace{1cm}

\nota{The links of the finite vertices of the polytope $P_t$ are spherical tetrahedra. The black, red and green edges have dihedral angle respectively $\frac\pi2$, $\theta$ and $\varphi$. The faces of these tetrahedra are labeled (front faces in black, back faces in blue) with the type of the corresponding wall of $P_t$: P for positive, N for negative and L for letter. The first two tetrahedra are spherical joins of segments $I_{\frac \pi 2} * I_\theta$ and $I_\varphi * I_\theta$, where $I_\alpha\subset S^1$ indicates a circular arc of length $\alpha$. The latter is the regular spherical tetrahedron $\Delta^3(\theta)$ with dihedral angles $\theta$.
}\label{links:fig}
\end{figure}

\begin{figure}
\definecolor{qqqqff}{rgb}{0,0,1}
\definecolor{ffqqqq}{rgb}{1,0,0}
\begin{tikzpicture}[line cap=round,line join=round,>=triangle 45,x=0.7cm,y=0.7cm]
\clip(-4.5,-1.92) rectangle (10.81,4.32);
\draw [line width=1.2pt] (-2,1)-- (-0.5,1);
\draw [line width=1.2pt] (-0.5,1)-- (-0.5,-1);
\draw [line width=1.2pt] (-0.5,-1)-- (-2,-1);
\draw [line width=1.2pt] (-2,1)-- (-4,2);
\draw [line width=1.2pt] (-4,2)-- (-2.5,2);
\draw [line width=1.2pt] (-2.5,2)-- (-0.5,1);
\draw [line width=1.2pt] (-2,-1)-- (-4,0);
\draw [line width=1.2pt] (-4,0)-- (-4,2);
\draw [line width=1.2pt,dash pattern=on 2pt off 2pt] (-4,0)-- (-2.5,0);
\draw [line width=1.2pt,dash pattern=on 2pt off 2pt] (-2.5,0)-- (-0.5,-1);
\draw [line width=1.2pt,dash pattern=on 2pt off 2pt] (-2.5,0)-- (-2.5,2);
\draw [line width=1.2pt] (-2,1)-- (-2,-1);
\draw [line width=1.2pt] (4,0.4)-- (3,-1);
\draw [line width=1.2pt] (3,-1)-- (5,-1);
\draw [line width=1.2pt] (5,-1)-- (4,0.4);
\draw [line width=1.2pt,color=ffqqqq] (4,0.4)-- (4,3.4);
\draw [line width=1.2pt] (4,3.4)-- (5,2);
\draw [line width=1.2pt,color=ffqqqq] (5,2)-- (5,-1);
\draw [line width=1.2pt] (4,3.4)-- (3,2);
\draw [line width=1.2pt,color=ffqqqq] (3,2)-- (3,-1);
\draw [line width=1.2pt,dash pattern=on 2pt off 2pt] (3,2)-- (5,2);
\draw [line width=1.2pt,color=ffqqqq] (8.37,-1.14)-- (8.81,2.03);
\draw [line width=1.2pt,color=ffqqqq] (8.81,2.03)-- (10,0);
\draw [line width=1.2pt,color=ffqqqq] (10,0)-- (8.37,-1.14);
\draw [line width=1.2pt,color=ffqqqq] (8.37,-1.14)-- (7.3,0.35);
\draw [line width=1.2pt,color=ffqqqq] (7.3,0.35)-- (8.81,2.03);
\draw [line width=1.2pt,dash pattern=on 2pt off 2pt,color=ffqqqq] (7.3,0.35)-- (10,0);

\draw [shift={(-.15 cm, .15 cm)}](-1.28,0.21) node[anchor=north west] {N};
\draw [color=qqqqff, shift={(-.15 cm, .15 cm)}](-3.38,1.15) node[anchor=north west] {N};
\draw [shift={(-.15 cm, .15 cm)}](-3.05,0.66) node[anchor=north west] {P};
\draw [color=qqqqff, shift={(-.15 cm, .15 cm)}](-1.7,0.7) node[anchor=north west] {P};
\draw [shift={(-.15 cm, .15 cm)}](-2.32,1.64) node[anchor=north west] {L};
\draw [color=qqqqff, shift={(-.15 cm, .15 cm)}](-2.48,-0.28) node[anchor=north west] {L};
\draw [shift={(-.15 cm, .15 cm)}](4.37,1.3) node[anchor=north west] {P};
\draw [shift={(-.15 cm, .15 cm)}](3.46,1.32) node[anchor=north west] {P};
\draw [color=qqqqff, shift={(-.15 cm, .15 cm)}](4.19,0.86) node[anchor=north west] {P};
\draw [shift={(-.15 cm, .15 cm)}](3.92,-0.28) node[anchor=north west] {L};
\draw [color=qqqqff, shift={(-.15 cm, .15 cm)}](3.66,2.72) node[anchor=north west] {N};
\draw [shift={(-.15 cm, .15 cm)}] (8.12,0.99) node[anchor=north west] {P};
\draw [color=qqqqff, shift={(-.15 cm, .15 cm)}](8.88,1.1) node[anchor=north west] {P};
\draw [color=qqqqff, shift={(-.15 cm, .15 cm)}](8.21,0.03) node[anchor=north west] {P};
\draw[shift={(-.15 cm, .15 cm)}] (9.02,0.60) node[anchor=north west] {P};

\draw [shift={(.8 cm, -.8 cm)}](-1.28,0.21) node[anchor=north west] {(1)};
\draw [shift={(4.7 cm, -.8 cm)}](-1.28,0.21) node[anchor=north west] {(2)};
\draw [shift={(7.7 cm, -.8 cm)}](-1.28,0.21) node[anchor=north west] {(3)};
\end{tikzpicture}

\nota{The Euclidean links of the ideal vertices of the polytope $P_t$.
The conventions as the same of Figure \ref{links:fig}. The link (1) is a rectangular parallelepiped, whose edge lengths vary smoothly on $t$. The link (2) is a prism with equilateral base and appears only at the time $t_1$, when the red edges have dihedral angle $\theta = \frac \pi 3$. The link (3) is a regular tetrahedron and it appears only at the time $t_2$ when the red edges have dihedral angle $\theta$ with $\cos \theta = \frac 13$. 
}\label{links_eucl:fig}
\end{figure}

\subsection*{The first time interval.}
When $t\in(t_1,1)$, the combinatorial change from the 24-cell $P_1$ consists in the substitution of 12 ideal vertices with 12 quadrilateral red faces. Each of these new 12 red faces is the intersection (with dihedral angle $\theta$) of two positive walls that were asymptotically parallel in $P_1$. 

Geometrically, all the other faces remain right-angled except six green triangles that were right-angled in the ideal 24-cell $P_1$ and have now dihedral angle $\varphi$.

\begin{prop}\label{combpiccoli}
When $t\in(t_1,1)$, the polytope $P_t$ has 24 walls, 108 faces, 144 edges and 60 vertices. The combinatorics can be recovered from Figure \ref{walls:fig}. In particular, the vertices are of three kinds:

\begin{enumerate}
\item 12 ideal vertices (which actually exist for all $t\in(0,1]$), whose link is a Euclidean rectangular parallelepiped,
represented in Figure \ref{links_eucl:fig}-(1).
For every odd $\l i\in \{\l0,\ldots,\l7\}$ there are three ideal vertices of type
$$\partial_\infty\p i\cap\partial_\infty\m i\cap\partial_\infty\p j\cap\partial_\infty\m j\cap\partial_\infty\l X\cap\partial_\infty\l Y$$ 
for some even $\l j$ and some letter walls $\l X,\l Y$ of type $\l A,\ldots,\l F$.
\item 24 finite vertices, whose link is the spherical tetrahedron
represented in Figure \ref{links:fig}-(1).
Each of these vertices is the intersection of two positive walls, a negative wall, and a letter wall of type $\l{A},\ldots,\l{F}$.
\item 24 finite vertices, whose link is the spherical tetrahedron
represented in Figure \ref{links:fig}-(2).
Each of these vertices is the intersection of two positive walls, a negative wall, and a wall $\l{G}$ or $\l{H}$.
\end{enumerate}
\end{prop}
\begin{proof}
The 48 finite vertices are the $4\times 12$ vertices of the new 12 red quadrilateral faces; among these, $8\times 3=24$ are also vertices of the 8 triangular green faces. Recall that the polytope $P_t$ is simple. The links of the finite vertices are therefore tetrahedra, whose dihedral angles are all right except those corresponding to red or green faces.

The ideal vertex of the quotient polytope $Q_t$ is (see Figure \ref{coxeter:fig}) $$\partial_\infty\p3\cap\partial_\infty\m3\cap\partial_\infty\p0\cap\partial_\infty\m0\cap\partial_\infty\l A\cap\partial_\infty\l L.$$
Its link is a product of three intervals, that is, a Euclidean rectangular parallelepiped. Letting the group of symmetries $K$ act, we get the $4\times3=12$ ideal vertices of $P_t$.
Note that since that ideal vertex exists in $Q_t$ for all $t\in(0,1]$, these 12 ideal vertices of $P_t$ exist for all $t\in(0,1]$.
\end{proof}

We note in particular that the green and red faces intersect only at the 24 finite vertices of type (3).

\subsection*{The first critical time.}
At the critical time $t=t_1$, the 8 green triangular faces collapse into 8 new ideal vertices. The only non-right dihedral angle is now $\theta=\frac \pi 3$, hence $P_{t_1}$ is a Coxeter polytope.

\begin{prop}\label{combcrit1}
The Coxeter polytope $P_{t_1}$ has 24 walls, 100 faces, 120 edges and 44 vertices.
The combinatorics can be recovered from Figure \ref{walls1:fig}.
In particular, the vertices are of three kinds:
\begin{enumerate}
\item 12 ideal vertices, whose link is a Euclidean rectangular parallelepiped
represented in Figure \ref{links_eucl:fig}-(1). 
\item 8 ideal vertices, whose link is a Euclidean right prism over an equilateral triangle,
represented in Figure \ref{links_eucl:fig}-(2).
Each of these vertices is the ideal vertex of a negative wall, three positive walls, and a wall $\l G$ or $\l H$. 
\item 24 finite vertices, whose link is the spherical tetrahedron
represented in Figure \ref{links:fig}-(1).
Each of these vertices is the intersection of two positive walls, a negative wall, and a letter wall of type $\l{A},\ldots,\l{F}$.
\end{enumerate}
\end{prop}

\subsection*{The second time interval.}
When $t\in(t_2,t_1)$, the combinatorial change from the Coxeter polytope $P_{t_1}$ consists in the substitution of 8 ideal vertices with 8 new edges, drawn in yellow in Figure \ref{walls_medi:fig}. Each yellow edge is the intersection of three positive walls, and also of three red faces. Each red face is now a right-angled hexagon.

\begin{prop}\label{combmedi}
When $t\in(t_2,t_1)$, the polytope $P_t$ has 24 walls, 100 faces, 128 edges and 52 vertices. The combinatorics can be recovered from Figure \ref{walls_medi:fig}.
In particular, the vertices are of three kinds:
\begin{enumerate}
\item 12 ideal vertices, whose link is a Euclidean rectangular parallelepiped
represented in Figure \ref{links_eucl:fig}-(1).
\item 24 finite vertices, whose link is the spherical tetrahedron
represented in Figure \ref{links:fig}-(1).
Each of these vertices is the intersection of two positive walls, a negative wall, and a letter wall of type $\l{A},\ldots,\l{F}$.
\item 16 finite vertices, whose link is the spherical tetrahedron
represented in Figure \ref{links:fig}-(3).
Each of these vertices is the intersection of three positive walls and a negative wall, or three positive walls and a wall $\l{G}$ or $\l{H}$.
\end{enumerate}
\end{prop}

\subsection*{The last time interval.}
When $t\in(0,t_2]$, the polytope $P_t$ coincides with the $\calF_t$ of \cite{KS}. 
At the critical time $t_2$ the walls $\l G$ and $\l H$ collapse into two new ideal vertices, that become finite as soon as $t<t_2$. Indeed, the vectors defining $\l G$ and $\l H$ transform from space-like to light-like and then time-like. The combinatorial change at $t_2$ is the inverse operation of a truncation.

The two new vertices in $P_t$ are quadruple intersections of positive walls. Their link is a regular tetrahedron with dihedral angles $\theta$.
At $t=t_2$ the two new vertices are ideal, we have $\cos\theta=\frac 1 3$ and the link is a regular Euclidean tetrahedron; as soon as $t<t_2$ the angle $\theta$ increases and the link is a regular spherical tetrahedron.

\begin{prop}\label{combgrandi}
When $t\in(0,t_2]$, the polytope $P_t$ has 22 walls, 92 faces, 116 edges and 46 vertices. The combinatorics can be recovered from Figure \ref{walls_plus:fig} for positive walls and from Figure \ref{walls_medi:fig} for the other walls.
In particular, the vertices are of four kinds:
\begin{enumerate}
\item 12 ideal vertices, whose link is a Euclidean rectangular parallelepiped
represented in Figure \ref{links_eucl:fig}-(1).
\item 24 finite vertices, whose link is the spherical tetrahedron
represented in Figure \ref{links:fig}-(1).
Each of these vertices is the intersection of two positive walls, a negative wall, and a letter wall of type $\l{A},\ldots,\l{F}$.
\item 8 finite vertices, whose link is the spherical tetrahedron
represented in Figure \ref{links:fig}-(3).
Each of these vertices is the intersection of three positive walls and a negative wall.
\item 2 vertices, ideal for $t=t_2$ and finite for $t<t_2$, whose link is the regular tetrahedron 
represented respectively in Figure \ref{links_eucl:fig}-(3) and Figure \ref{links:fig}-(4).
Each of these vertices is the intersection of four positive walls of the same parity.
\end{enumerate}
\end{prop}

Note that in this time interval, the (yellow) angle
$$\psi=\arccos\left(\frac{\cos\theta}{1-\cos\theta}\right)$$
of Lemma \ref{quotient_walls_coxeter:lemma} equals the inner angle of a face of a regular spherical tetrahedron with dihedral angles $\theta$.
In the polytope $P_t$, the red faces are now pentagons with four right angles and a new angle $\eta$, that must equal the length of an edge of such a spherical tetrahedron.

\begin{prop}\label{eta:prop}
When $t\in(0,t_2]$, the inner angle $\eta$ between the two yellow edges of each red face is such that
$$\cos\eta=\frac{\cos\theta}{1-2\cos\theta}=\frac{3t^2-1}{3-5t^2}.$$ 
\end{prop}
\begin{proof}
\emph{First way.} Denote by $P$ the orthogonal projection of $\mathbb{R}^{1,4}$ onto the vector subspace $W^\perp=(\p 3)^\perp\cap(\p 7)^\perp$, where $W$ is generated by the vectors $\p 3$ and $\p 7$. An orthogonal basis for $W$ is given by $\mathbf u_1=\p 3$ and $\mathbf u_2=\p 7+\cos\theta\ \p 3$.
Therefore, denoting by $P_i$ the orthogonal projection onto the subspace $\mathbb{R}\mathbf u_i$ ($i=1,2$), for every $\mathbf v\in\mathbb{R}^{1,4}$
$$P(\mathbf v)=\mathbf v-P_1(\mathbf v)-P_2(\mathbf v).$$
The angle $\eta$ is thus given applying Formula (\ref{formula:eqn}) to the vectors
$$P(\p 1)=\p 1+\cos\psi\ \p 3+\cos\psi\ \p7,\ P(\p 5)=\p 5+\cos\psi\ \p 3+\cos\psi\ \p7.$$
\emph{Second way.} For every $n>1$, denote by $G_n$ the Gram matrix of a regular spherical $n$-simplex with dihedral angles $\theta$, that is the $(n+1)\times(n+1)$ matrix with 1's on the diagonal and $-\cos\theta$ on the other entries.
As we said,  $\eta$ is the length of an edge of a regular spherical 3-simplex with dihedral angles $\theta$.
By the sine law \cite{DMP} we get
$$\frac{\sin^2\eta}{\sin^2\theta}=\frac{\mathrm{det}(G_3)}{\mathrm{det}(G_2)^2}=\frac{1-3\cos\theta}{(1-2\cos\theta)^2(1+\cos\theta)}.$$
This easily implies the statement.
\end{proof}
The angle $\eta$ tends to $\arccos(-\frac 13)$ as $t\rightarrow0$.

\subsection*{The fixed ideal cuboctahedron}
Let $\matH^3\subset\mathbb{H}^4$ be the hyperplane $\{x_4=0\}$ defined by the space-like vector $(0,0,0,0,1)$.

\begin{lemma}
The 12 ideal vertices of $P_t$ that exist for all $t\in(0,1]$ are all in $\partial_\infty \matH^3$ and do not depend on $t$.
\end{lemma}
\begin{proof}
Recall Section \ref{symmetries:sec} and the quotient polytope $Q_t$.
The fixed points of the roll symmetry $R$ form a 2-plane contained in $\matH^3$.
The roll symmetry $R$ fixes the ideal vertex of $Q_t$ that exists for all $t$.
The hyperplanes $\l L$, $\l M$ and $\l N$ are orthogonal to $\matH^3$.
Therefore, letting the group $K$ act, we get that the 12 ideal vertices are contained in $\partial_\infty \matH^3$.

Now, by solving a simple linear system in $\mathbb{R}^{1,4}$, we get $$\p3^\perp\cap\m3^\perp\cap\p0^\perp\cap\m0^\perp\cap\l A^\perp\cap\l B^\perp=(\sqrt2,1,1,0,0)\mathbb{R},$$
showing that the ideal vertex
$$\partial_\infty\p3\cap\partial_\infty\m3\cap\partial_\infty\p0\cap\partial_\infty\m0\cap\partial_\infty\l A\cap\partial_\infty\l B$$
does not depend on $t$, nor hence the other 11 by symmetry.
\end{proof}

\begin{prop}\label{cuboct:prop}
The intersection $P_t\cap \matH^3$ does not depend on $t$ and is an ideal, right-angled cuboctahedron.
The quadrilateral faces are $\l X\cap \matH^3$ for every letter wall $\l X\in\lbrace\l A,\ldots,\l F\rbrace$,  while the triangular faces are the 2-faces of $P_t$ given by $\p i\cap\m i$ for every $\l i\in\lbrace\l0,\ldots,\l7\rbrace$.
Moreover, we have $P_t\cap \matH^3=P_0=\bigcap_s P_s$.
\end{prop}
\begin{proof}
For every $t\in(0,1]$
we have $\partial_\infty\l A\subset\partial_\infty \matH^3$.
Thus $\l A\cap \matH^3$ must be the ideal quadrilateral containing the ideal points of $\l A$.
It is easy to see that the hyperplanes containing the walls $\m0$, $\p0$ and $\matH^3$ intersect in the same 2-plane. Therefore the ideal triangle $\m0\cap\p0 $ is contained in $\matH^3$.

By the previous lemma, such ideal polygons do not depend on $t$.
As before, since $\matH^3$ is orthogonal to the hyperplanes $\l L$, $\l M$ and $\l N$, it suffices to let the group $K$ act to conclude the same for the other walls.

Finally, since $P_0$ is the convex hull of its (ideal) vertices, that are fixed, the last statement is proved.
\end{proof}

All these intersections do not depend on $t$. Moreover, $\l X\perp \matH^3$ for all $t\in(0,1]$ and every $\l X\in\lbrace\l A,\ldots,\l F\rbrace$. What varies is the (acute) angle of intersection between $\matH^3$ and the numbered hyperplanes:

\begin{prop}\label{angle_V:prop}
The letter hyperplanes are orthogonal to the hyperplane $\matH^3$ for all $t\in[0,1]$.
Moreover, for every $\l i\in\lbrace\l0,\ldots,\l7\rbrace$, the functions $\mathrm{Angle}(\p i,\matH^3)$ and $\mathrm{Angle}(\m i,\matH^3)$ are strictly monotone in $t$, they take the value $\frac{\pi}{4}$ at $t=1$, and
$$\lim_{t\to0}\mathrm{Angle}(\p i,\matH^3)=0,\quad\lim_{t\to0}\mathrm{Angle}(\m i,\matH^3)=\tfrac{\pi}2.$$
\end{prop}
\begin{proof}
These assertions can be verified as usual by Formula (\ref{formula:eqn}).
\end{proof}

\subsection{Coxeter polytopes}
The dihedral angles $\theta$ and $\varphi$ are strictly monotone in $t$. We have
$$
\lim_{t \to 1}\theta (t) = 0, \qquad \theta (t_1) = \tfrac \pi 3, \qquad \theta \Big(\sqrt{\tfrac 13}\Big) = \tfrac \pi 2, \qquad \lim_{t \to 0} \theta (t) = \pi, 
$$
$$
\varphi(1) = \tfrac \pi 2, \qquad \lim_{t\to t_1}\varphi (t) = 0.\qquad\qquad\qquad\qquad\qquad\qquad\qquad\quad
$$
In particular the polytope $P_t$ is Coxeter at the times 
$$1, \quad t_1 = \sqrt{\frac 35}, \quad \bar t=\sqrt{\frac 13}.$$
The polytope $P_t$ is right-angled both at times $t=1$ and $t=\bar t$. Note that in $P_1$ all vertices are ideal, while $P_{\bar t}$ contains both ideal and finite vertices and is quite interesting. The Coxeter polytope $P_{t_1}$ has dihedral angles $\frac \pi 2$ and $\frac \pi 3$.

The orbifold Euler characteristic of these Coxeter polytopes is calculated below (for the 24-cell $P_1$, it is well-known that $\chi(P_1)=1$).

\begin{prop}\label{euler_t1:prop}
The Coxeter polytope $P_{t_1}$ has Euler characteristic $\chi(P_{t_1})=1$.
\end{prop}
\begin{proof}
The isomorphism classes of the stabilizers are obtained from the information about the dihedral angles of the faces of every dimension, that are either $\frac{\pi}3$, $\frac{\pi}2$ or $0$.
Precisely, Figure \ref{walls1:fig} and Proposition \ref{combcrit1} give:
\begin{itemize}
\item 24 walls (with stabilizer $\mathbb{Z}/_{2\mathbb{Z}}$);
\item 88 faces with stabilizer $\mathbb{Z}/_{2\mathbb{Z}}\times\mathbb{Z}/_{2\mathbb{Z}}$;
\item 12 faces with stabilizer the dihedral group $D_3$ (of order 6);
\item 72 edges with stabilizer $\mathbb{Z}/_{2\mathbb{Z}}\times\mathbb{Z}/_{2\mathbb{Z}}\times\mathbb{Z}/_{2\mathbb{Z}}$;
\item 48 edges with stabilizer $D_3\times\mathbb{Z}/_{2\mathbb{Z}}$;
\item 24 finite vertices with stabilizer $D_3\times\mathbb{Z}/_{2\mathbb{Z}}\times\mathbb{Z}/_{2\mathbb{Z}}$;
\item 20 ideal vertices (with infinite stabilizer).
\end{itemize}
Therefore, we get
\begin{align*}
\chi & = 1 + 24\cdot \frac{-1}2 + 88\cdot \frac 14 + 12\cdot \frac 16 + 72\cdot \frac{-1}8 + 48\cdot \frac{-1}{12} + 24\cdot \frac 1{24} = \\
 & = 1 - 12 + 22 + 2 - 9 - 4 + 1 = 1.
\end{align*}
The proof is complete.
\end{proof}

We will re-prove that $\chi(P_{t_1}) = \chi(P_1) = 1$ later on using two more different arguments.

\begin{prop}\label{euler_bart:prop}
The Coxeter polytope $P_{\bar t}$ has Euler characteristic $\chi(P_{\bar t})=\frac 58$.
\end{prop}
\begin{proof}
More easily than above: since the polytope is right-angled, the stabilizer of a $k$-dimensional face is isomorphic to $\left(\mathbb{Z}/_{2\mathbb{Z}}\right)^{4-k}$. Therefore Proposition \ref{combgrandi} gives:
\begin{align*}
\chi & = 1 + 22\cdot \frac{-1}2 + 92\cdot \frac 14 + 116\cdot \frac{-1}8 + 34\cdot \frac 1{16} = \\
 & = \frac 18\left(8 - 88 + 184 - 116 + 17\right) = \frac 58,
\end{align*}
and the proof is complete.
\end{proof}

There are also two more interesting times $t$ when $\theta$ equals $\frac{2\pi}{5}$ and $\frac{2\pi} 3$. In both cases the resulting $P_t$ is however not a Coxeter polytope, because the angles do not divide $\pi$.

\subsection{Volume.}
We now study the volume $\Vol(P_t)$ of the polytope $P_t$.
Instead of a long computation using the Poincar\'e formula, we just exhibit the value of the volume and verify it by the Schl\"afli formula. 
Recall that the Schl\"afli formula can be applied only while the combinatorics stays constant, therefore we need to consider three cases separately, for the first, second, and last time interval.
We know the initial data of these three differential equations, because the Gauss-Bonnet formula for 4-orbifolds 
$$\mathrm{Vol}(O)=\frac{4\pi^2}{3}\chi(O)$$
furnishes the volume of the Coxeter polytopes $P_1$ and of $P_{t_1}$. 

Instead of using $t$ as a parameter, it is much more convenient to write $\Vol(P_t)$ in function of the angles $\theta$ and $\varphi$.

\begin{prop}\label{vol_piccoli}
When $t\in [t_1,1]$, the volume of $P_t$ depends on the
dihedral angles $\theta$ and $\varphi$ as follows:
$$\mathrm{Vol}(P_t)=\frac{4\pi^2}{3}\left(2-\frac{3}{\pi}\theta-\frac{2}{\pi}\varphi+\frac{6}{\pi^2}\theta\varphi\right).$$
\end{prop}
\begin{proof}
By Proposition \ref{dihedral:prop} the only non-constant dihedral angles are:
\begin{itemize}
\item $\theta$ at 12 red quadrilateral faces with angles $\frac{\pi}2$, $\frac{\pi}2$, $\varphi$, $\varphi$;
\item $\varphi$ at 8 green triangular faces with angles $\theta$, $\theta$, $\theta$.
\end{itemize}
Therefore, the Schl\"afli formula gives
$$\frac 1 8 {d\mathrm{Vol}}=\left(\varphi-\frac{\pi}2\right)d\theta+\left(\theta-\frac{\pi}3\right)d\varphi.$$
The orbifold Euler characteristic of the extremes is $\chi(P_1)=1=\chi(P_{t_1})$. The first equality is well-known, the second is proved in Proposition \ref{euler_t1:prop}. (Actually, we only need the first, and we re-obtain the second now, providing a new proof of Proposition \ref{euler_t1:prop}.) Hence, by Gauss-Bonnet, the initial and final value of the volume is $\frac{4\pi^2}3$.

It is easy to check that the formula in the statement of the proposition satisfies this Cauchy problem (recall that at the extremes the values of the angles are respectively $\theta=0,\ \varphi=\frac{\pi}{2}$ and $\theta=\frac{\pi}{3},\ \varphi=0$). By uniqueness of the solution, the statement is proved.
\end{proof}

In the second and last time intervals, the only non-constant dihedral angle is $\theta$, therefore the volume decreases with $\theta$ by the Schl\"afli formula. In the second time interval, the formula for the volume simplifies and becomes linear in $\theta$. 

\begin{prop}\label{vol_medi}
When $t\in [t_2,t_1]$, the volume of $P_t$ depends on the
dihedral angle $\theta$ as follows:
$$\mathrm{Vol}(P_t)=\frac{4\pi^2}{3}\left(2-\frac{3}{\pi}\theta\right).$$
\end{prop}
\begin{proof}
The non-constant dihedral angle is $\theta$ at 12 right-angled red hexagons. Therefore, the Schl\"afli formula gives
$$d\mathrm{Vol}=-4\pi d\theta.$$
Moreover, we know that $\mathrm{Vol}(P_{t_1})=\frac{4\pi^2}{3}$ and $\theta(t_1) = \frac \pi 3$.
\end{proof}

We now analyse the last time interval. Recall the final collapse as $t\to0$.

\begin{prop}\label{vol_grandi}
When $t\in[0,t_2]$, 
The volume of $P_t$ depends on the dihedral angle $\theta$, as follows:
$$\mathrm{Vol}(P_t)=\frac{4\pi^2}{3}\left(2-\frac 3{\pi}\theta+\frac 3{\pi^2}\int_a^\theta\eta(\tilde\theta)d\tilde\theta\right),$$
where $a=\arccos\frac 13$ and $\eta$ depends on $\theta$ as prescribed by Proposition \ref{eta:prop}.
Moreover, the volume tends to zero as $t\to0$.
\end{prop}
\begin{proof}
Looking at Figure \ref{walls_plus:fig}, the non-constant dihedral angle is $\theta$ at the 12 red pentagons of Proposition \ref{eta:prop}.
Therefore, the Schl\"afli formula gives
$$d\mathrm{Vol}=-4(\pi-\eta)d\theta.$$
We know the initial datum at $t=t_2$ from Proposition \ref{vol_medi}.
The Schl\"afli formula is satisfied and the first statement is proved.

The last statement may be proved geometrically by showing that $P_t$ collapses onto the three-dimensional $P_0$, with its ideal vertices staying fixed and the finite ones converging to $\matH^3$. Alternatively, we can show that the value of the following \emph{Coxeter integral} is
$$\int_a^\pi\arccos\left(\frac{\cos\theta}{1-2\cos\theta}\right)d\theta=\frac{\pi^2}{3}.$$
This integral is not easy to compute directly; we rather give a geometric argument.
The Schl\"afli formula for a spherical polyhedron $P$ is
$$d\mathrm{Vol}(P)=\frac 12\sum_il_id\alpha_i.$$
We apply that formula to the regular spherical tetrahedron $T$ with dihedral angles $\theta$.
Recall that $\eta$ is the length of an edge of $T$. Therefore, denoting by $V(\theta)$ the volume of $T$, the formula becomes 
$$dV(\theta)=3\eta\ d\theta.$$
Now, to get the initial and final data of the last differential equation, we analyse the limit cases where $\theta=a=\arccos\frac13$ and $\theta=\pi$.
In the first case, the tetrahedron is a point, thus
$$V(\theta)=3\int_a^\theta\eta(\tilde\theta)d\tilde\theta$$
(this is not so surprising: compare with Poincar\'e formula in Section \ref{polytopes_prelim:section}).
When $\theta=\pi$, instead, the tetrahedron becomes a halfspace of $S^3$ (the surface of the tetrahedron becomes $S^2$ tessellated by four regular spherical triangles with inner angles $\psi=\frac{2\pi}{3}$), therefore
$$V(\pi)=\frac12\mathrm{Vol}(S^3)=\pi^2$$ 
which gives the desired value for the Coxeter integral.
\end{proof}

\begin{figure}
\vspace{.5 cm}
\labellist
\small\hair 2pt
\pinlabel $t_1$ at 480 0 
\pinlabel $t_2$ at 430 0 
\pinlabel $\bar t$ at 350 0 
\pinlabel $\frac 43\pi^2$ at 0 420 
\pinlabel $\frac 56\pi^2$ at 0 270 
\endlabellist
\centering
\includegraphics[width=8 cm]{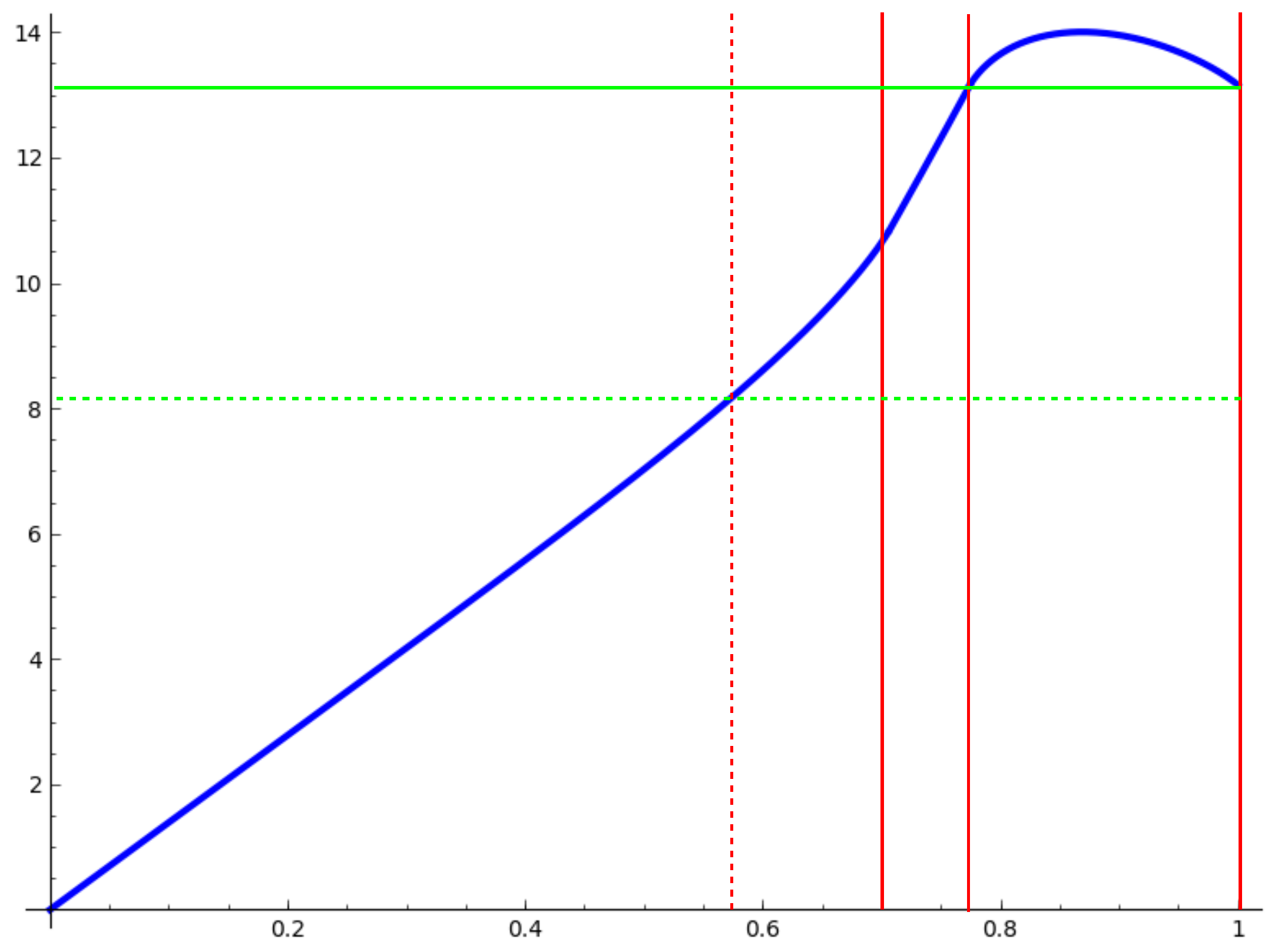}
\vspace{.3 cm}
\nota{The function $\Vol(P_t)$. The polytope $P_t$ changes its combinatorics at the times $t_2, t_1, 1$ and is Coxeter at the times $\bar t, t_1, 1$.}\label{vol_plot_thick:fig}
\end{figure}

\begin{cor}
The function $t \mapsto \Vol(P_t)$ is of class $C^1$ and shown in Figure \ref{vol_plot_thick:fig}.
\end{cor} 

\section{The manifolds.} \label{manifolds:section}
We now use the deforming polytopes $P_t$ to construct some deforming hyperbolic cone four-manifolds $W_t$, $N_t$, and $M_t$, each tessellated into a fixed number of copies of $P_t$. The manifolds $W_t$ and $M_t$ are those needed for Theorems \ref{main2:teo} and \ref{main:teo}.

\subsection*{Overview}
We first construct a hyperbolic cone-manifold $W_t$ tessellated into eight copies of $P_t$. The manifold $W_t$ is constructed by mirroring $P_t$ three times, one for each wall octet: this is a particularly simple application of a colouring technique that we introduce in Section \ref{8copie:sec}. In fact $W_t$ is the simplest interesting cone-manifold that we can construct from $P_t$.

The deforming cone manifold $W_t$ has many symmetries and is relatively easy to analyse, so we do this with some detail. As usual, we think of $t$ moving backwards from the initial time $1$ in the interval $(0,1]$. Along the path in $(0,1]$ we discover various types of hyperbolic Dehn surgeries, and a final degeneration at $t\to 0$ similar to the one described by Thurston in his notes \cite{bibbia}. This proves Theorem \ref{main2:teo}.

When $t$ varies in the interval $[t_1,1]$, the manifold $W_t$ is quite like the one needed for Theorem \ref{main:teo}, except that it interpolates between a manifold and an \emph{orbifold}. To promote the orbifold to a manifold, we need to modify the construction: we build a new cone-manifold deformation $N_t$ via a more complicated pattern, and then further quotient it to get the $M_t$ of Theorem \ref{main:teo}.

The cone-manifolds $W_t, M_t, N_t$ that we construct here are not special in any sense: there are many ways one can modify their construction to produce different deforming cone-manifolds from $P_t$ with different types of behaviour. By taking finite covers one can also get infinitely many examples of various kinds. The only difficulty in the overall process is, of course, that we are working in dimension four and hence the combinatorial patterns are more complicated than in dimension three. 

\subsection{The colouring technique} \label{8copie:sec}
How can we construct a hyperbolic cone manifold from a single polytope $P$? A simple method consists of colouring its walls and then mirroring $P$ iteratively along them.

That is, we take a palette $\{c_1,\ldots, c_k\}$ of colours and assign arbitrarily a colour to every wall of $P$ (we suppose that each colour $c_i$ is assigned to $P$ at least once); then we mirror $P$ iteratively $k$ times along its walls, one colour at a time. 

More specifically, for every $I = (i_1,\ldots, i_k) \in \{0,1\}^k$ we fix a copy $P^I$ of $P$, and we identify every point in a wall of $P^I$ coloured with $c_i$ with the corresponding point in $P^{I'}$ where $I'$ differs from $I$ only in its $i$-th coordinate.

The resulting space is a hyperbolic cone manifold $M$ tessellated into $2^k$ copies of $P$. If $P$ is right-angled, and every pair of adjacent walls have different colours, then $M$ is a hyperbolic manifold (with no singularities). 

This construction works in all dimensions and was used for instance in \cite{KM} with the standard three-colouring of the ideal 24-cell $P_1$. It is now natural to extend it to $P_t$ for all $t\in (0,1]$.

\subsection{A family $W_t$ of hyperbolic cone four-manifolds} \label{Wt:subsection}
We now apply the colouring technique to our family $P_t$ of deforming polytopes, for all $t\in (0,1]$.

Each polytope $P_t$ in the family has either 24 or 22 walls, partitioned into \emph{letter}, \emph{negative}, and \emph{positive} walls. We interpret this as a colouring of the walls of $P_t$ with three colours $\{$L, N, P$\}$,
and we define $W_t$ to be the space obtained from $P_t$ by mirroring it as prescribed by this colouring, as explained above.

The space $W_t$ is a hyperbolic cone-manifold for all $t\in (0,1]$. It is tessellated into $2^3=8$ copies $P_t^{ijk}$ of $P_t$, whose walls are identified according to the following cubic scheme:
$$
\xymatrix@!0{ 
& P_t^{000} \ar@{-}^{\rm N}[rr]\ar@{-}'[d][dd]_{\rm L} 
& & P_t^{001} \ar@{-}[dd]^{\rm L} 
\\ 
P_t^{100} \ar@{-}[ur]^{\rm P}\ar@{-}[rr]^{\ \ \rm N}\ar@{-}[dd]_{\rm L} 
& & P_t^{101} \ar@{-}[ur]^{\rm P}\ar@{-}[dd]^<<<{{\rm L}}
\\
& P_t^{010} \ar@{-}'[r]^{\rm N}[rr]
& & P_t^{011} 
\\
P_t^{110} \ar@{-}[rr]_{\rm N}\ar@{-}[ur]^{\rm P} 
& & P_t^{111} \ar@{-}[ur]_{\rm P} 
}
$$
When $t=1$ the polytope $P_1$ is the right-angled ideal 24-cell and $W_1$ is a nice and very symmetric hyperbolic four-manifold with 24 cusps, each cusp having a cubic 3-torus section: this hyperbolic four-manifold was first described in \cite[Example 2.9]{KM}. We now study $W_t$ when $t < 1$.

\subsection*{The singular set $\Sigma$}
When $t<1$ the polytope $P_t$ is not right-angled anymore, hence some singularities appear in $W_t$. Luckily, only few faces in $P_t$ are not right-angled, so the singularities are easily detected. 

\begin{prop} \label{solo:prop}
The singular set $\Sigma$ of $W_t$ is the union of the green and red faces of the eight copies of $P_t$.
\end{prop}
\begin{proof}
At every point $x\in \partial P$ that does not lie in a green or red face, the polytope is locally right-angled and the adjacent walls have distinct colours. Therefore $x$ becomes a smooth point in $W_t$.
\end{proof}

\begin{figure}
\vspace{.5 cm}
\labellist
\small\hair 2pt
\pinlabel $(t_1,1)$ at 120 400 
\pinlabel $t_1$ at 370 400 
\pinlabel $(t_2,t_1)$ at 605 400 
\pinlabel $t_2$ at 760 400
\pinlabel $(0,t_2)$ at 875 400 

\pinlabel $2\theta$ at 145 298 
\pinlabel $2\theta$ at 83 298
\pinlabel $2\theta$ at 115 317 

\pinlabel $4\varphi$ at 52 110 
\pinlabel $4\varphi$ at 170 110 

\pinlabel $\eta$ at 885 92 
\pinlabel $\eta$ at 885 144 
\pinlabel $\eta$ at 860 47 
\pinlabel $\eta$ at 860 189 
\endlabellist
\centering
\includegraphics[width=12.5 cm]{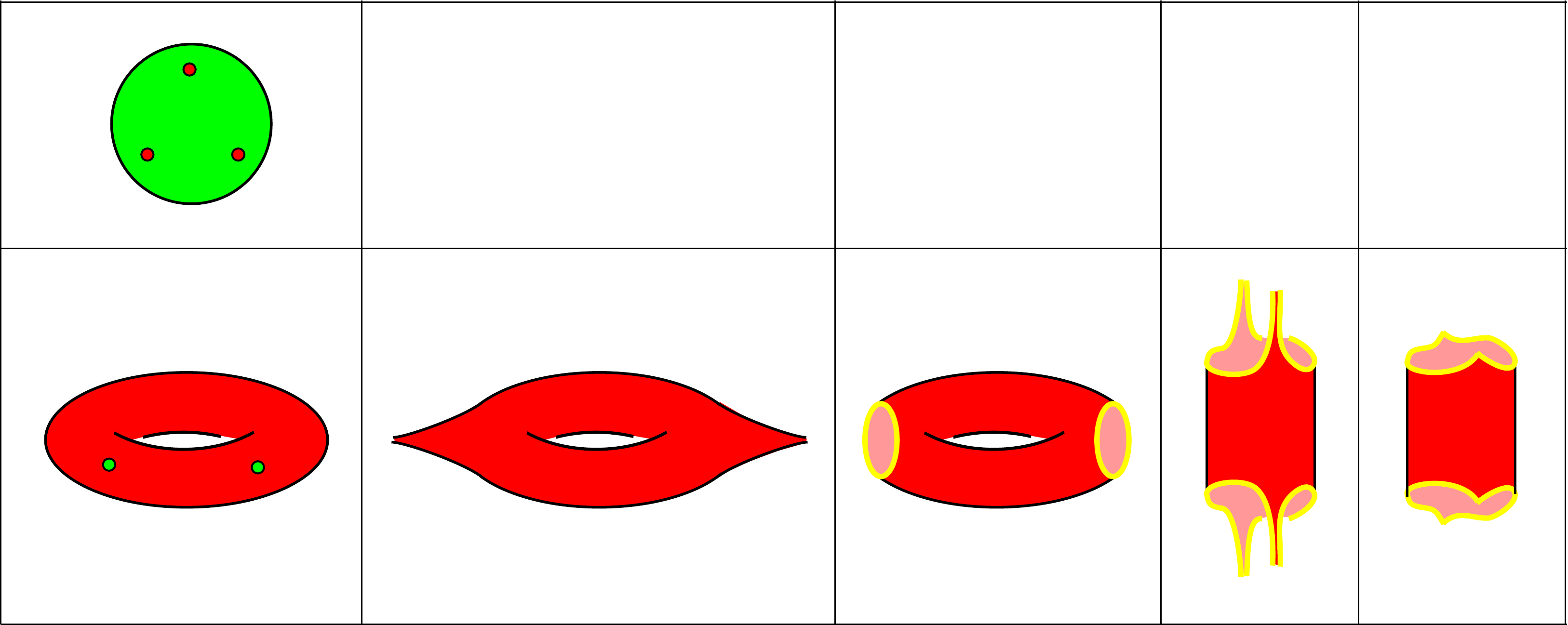}
\vspace{.3 cm}
\nota{Each closed 2-stratum of the singular set $\Sigma$ of $W_t$ is either green or red and its topology is shown here, depending on $t\in (0,1)$. The green closed stratum is a sphere with 3 cone points of angle $2\theta$ (the cone points are 0-strata) and arises only when $t>t_1$. The red closed stratum is a cone torus for $t>t_1$, a twice-punctured torus for $t=t_1$, and a compact twice-holed torus with geodesic boundary for $t\in (t_2,t_1)$; the topology of the red closed stratum changes at $t=t_2$ into an annulus: the geodesic boundary is non compact at $t=t_2$, and two boundary cone points arise when $t\in (0,t_2)$ with some angle $\eta$.}\label{polygons1:fig}
\end{figure}

In particular $\Sigma$ is the closure of its 2-strata and we can describe it quite easily. 
Recall from Figure \ref{cone_3_manifolds:fig} the names of some elliptic cone three-manifolds. We will also use the following terminology.

\begin{defn} \label{Sn:defn}
We denote by $S^n(\alpha)$ the (hyperbolic, Euclidean, or spherical) cone $n$-manifold obtained by doubling the regular (hyperbolic, Euclidean, or spherical) $n$-simplex with dihedral angle $\frac\alpha 2$ (when it exists). All the $(n-2)$-dimensional strata in $S^n(\alpha)$ have cone angle $\alpha$. In the Euclidean case we have $\cos \frac \alpha 2= \frac 1n$ and $S^n(\alpha)$ is defined only up to rescaling.
\end{defn}

We call a \emph{closed $k$-stratum} the closure of a $k$-stratum.

\begin{prop} \label{Sigma:W:prop}
Each closed 2-stratum of $\Sigma \subset W_t$ is either a green or red hyperbolic surface as shown in Figure \ref{polygons1:fig}. Its cone angle is respectively $4\varphi$ and $2\theta$.

There are 1-strata only when $t\in (0,t_1)$. The unit tangent space at a point in a 1-stratum is $S^0*S^2(2\theta)$.

There are 0-strata only in two disjoint time intervals, and these are the following:
\begin{itemize}
\item when $t\in(t_1,1)$, there are 24 points with unit tangent space $C_{2 \theta} * C_{4\varphi}$;
\item when $t\in (0,t_2)$, there are 8 points with unit tangent space $S^3(2\theta)$.
\end{itemize}
\end{prop}
\begin{proof}
To understand $\Sigma$, we analyse all the vertices $v$ of $P_t$ and determine the unit tangent space of their images in $W_t$. The vertices of $P_t$ are fully described in Propositions \ref{combpiccoli}, \ref{combcrit1}, \ref{combmedi}, and \ref{combgrandi}, and we refer to them.

We analyse the finite vertices $v$ of $P_t$ case by case. The link of $v$ in $P_t$ is always some spherical tetrahedron $\Delta$ whose four faces are naturally coloured like the walls they are contained in.
We refer to Figure \ref{links:fig}.

The unit tangent space of $v$ in $M_t$ is obtained by mirroring $\Delta$ along its faces according to the colours.

We note that a spherical tetrahedron with 4 right dihedral angles $\frac \pi 2$ and two opposite edges with dihedral angles $\alpha$ and $\beta$ is a spherical join $I_\alpha * I_\beta$ of two circle arcs of length $\alpha$ and $\beta$.

\begin{enumerate}
\item For every $t\in(0,1)$ the polytope $P_t$ has 24 finite vertices $v$ with link the spherical join $\Delta = I_\theta * I_{\frac{\pi}{2}}$. The 4 faces of $\Delta$ are coloured as P, P, N, L, with:
\begin{itemize}
\item the edge $I_{\frac \pi 2}$ lying between the two faces coloured by P, that form a dihedral angle $\theta$, and 
\item the edge $I_\theta$ lying between N and L, that form a dihedral angle $\frac \pi 2$. 
\end{itemize}
By mirroring $\Delta$ along L we get $I_\theta*I_\pi$ and by then mirroring along N we get $I_\theta*S^1$. Finally, by mirroring the result along P we get $C_{2\theta}*S^1$. Therefore the vertex $v$ in $W_t$ is an interior point of some 2-stratum of $\Sigma$.
\item When $t\in(t_1,1)$ the polytope $P_t$ has 24 vertices $v$ with link $I_\theta * I_\varphi$. Similarly as before, the resulting unit tangent space in $W_t$ is $C_{2\theta} * C_{4\varphi}$.
\item When $t\in(0,t_1)$, the polytope $P_t$ contains some (either 16 or 8) vertices $v$ with link a spherical tetrahedron with three edges sharing a vertex having dihedral angle $\theta$, while the other three have dihedral angle $\frac \pi 2$. Three faces are coloured with P and one with either N or L. By mirroring along N or L we get $S^0*T$, where $T$ is the equilateral spherical triangle with inner angles $\theta$. By mirroring the result along P we get $S^0*S^2(2\theta)$. Therefore $v$ in $W_t$ belongs to the 1-stratum of $\Sigma$.
\item When $t\in(0,t_2]$, the polytope $P_t$ contains 2 vertices $v$ with link a spherical regular tetrahedron with all dihedral angles $\theta$ and all faces coloured by P. By mirroring it we get $S^3(2\theta)$. 
\end{enumerate}
This discussion determines the possible unit tangent spaces at every point of $W_t$ for all times $t\in (0,1)$, since the vertices contain all the relevant information.

The 2-strata in Figure \ref{polygons1:fig} are obtained by analyzing the effect of the mirroring to the green and red polygons of Figure \ref{polygons0:fig}. Each side $e$ of every green or red polygon $f$ is naturally coloured by the colour of the unique wall that is incident to $e$ but does not contain $f$ (every edge in a simple polytope is incident to three walls). By applying the mirroring technique we get the 2-stratum. Here are the details:
\begin{itemize}
\item the three sides of the green triangles are coloured with P, the triangle is mirrored and gives a green sphere $S^2(2\theta)$ with three cone points of angle $2\theta$, and this is a closed 2-stratum;
\item the horizontal and vertical sides of the red polygon in Figure \ref{polygons0:fig} are coloured by L and N, so at $t> t_1$ the polygon is a quadrilateral and is mirrored twice to give a torus with two cone points of angle $4\varphi$, and each torus is tessellated by four rectangles and forms a closed stratum; when $t<t_1$, the diagonal sides are coloured with P and are not mirrored: they form the (yellow) boundary of the 2-stratum (which consists of closed 1-strata).
\end{itemize}
The proof is complete.
\end{proof}

\begin{cor}
When $t\in (t_1,1)$ the singular set $\Sigma$ is an an immersed geodesic surface made of 12 cone-tori and 8 cone-spheres, intersecting in 24 points.
\end{cor}

The intersection pattern of the red cone-tori and green cone-spheres is shown in Figure \ref{sphere_tori:fig}-(left). The figure then shows the evolution of $\Sigma$ when $t>t_2$.

Note that for all $t\in(0,1]$ the unit tangent spaces are cone-manifolds always supported on the sphere $S^3$. Therefore the cone-manifold $W_t$ is always supported on a four-manifold.

Here is another important consequence of Proposition \ref{Sigma:W:prop}.

\begin{cor}
When $t=t_1$ the hyperbolic cone manifold $W_{t_1}$ is an orbifold. Its singular set $\Sigma$ consists of 12 red twice-punctured tori with cone angle $\frac {2\pi} 3$.
\end{cor}
\begin{proof}
At $t=t_1$ we have $2\theta = \frac {2\pi} 3$.
\end{proof}

We have shown that the family $W_t$ with $t\in[t_1,1]$ interpolates between a manifold for $t=1$ and an orbifold for $t=t_1$. We now analyse the cusps of the whole family.

\begin{figure}
\labellist
\small\hair 2pt
\pinlabel $(t_1,1)$ at 100 420
\pinlabel $t_1$ at 355 420
\pinlabel $(t_2,t_1)$ at 595 420
\pinlabel $W_t$ at 100 200
\pinlabel $W_{t_1}$ at 355 200
\pinlabel $W_t$ at 595 200
\endlabellist
\centering
\vspace{.5 cm}
\includegraphics[width=11 cm]{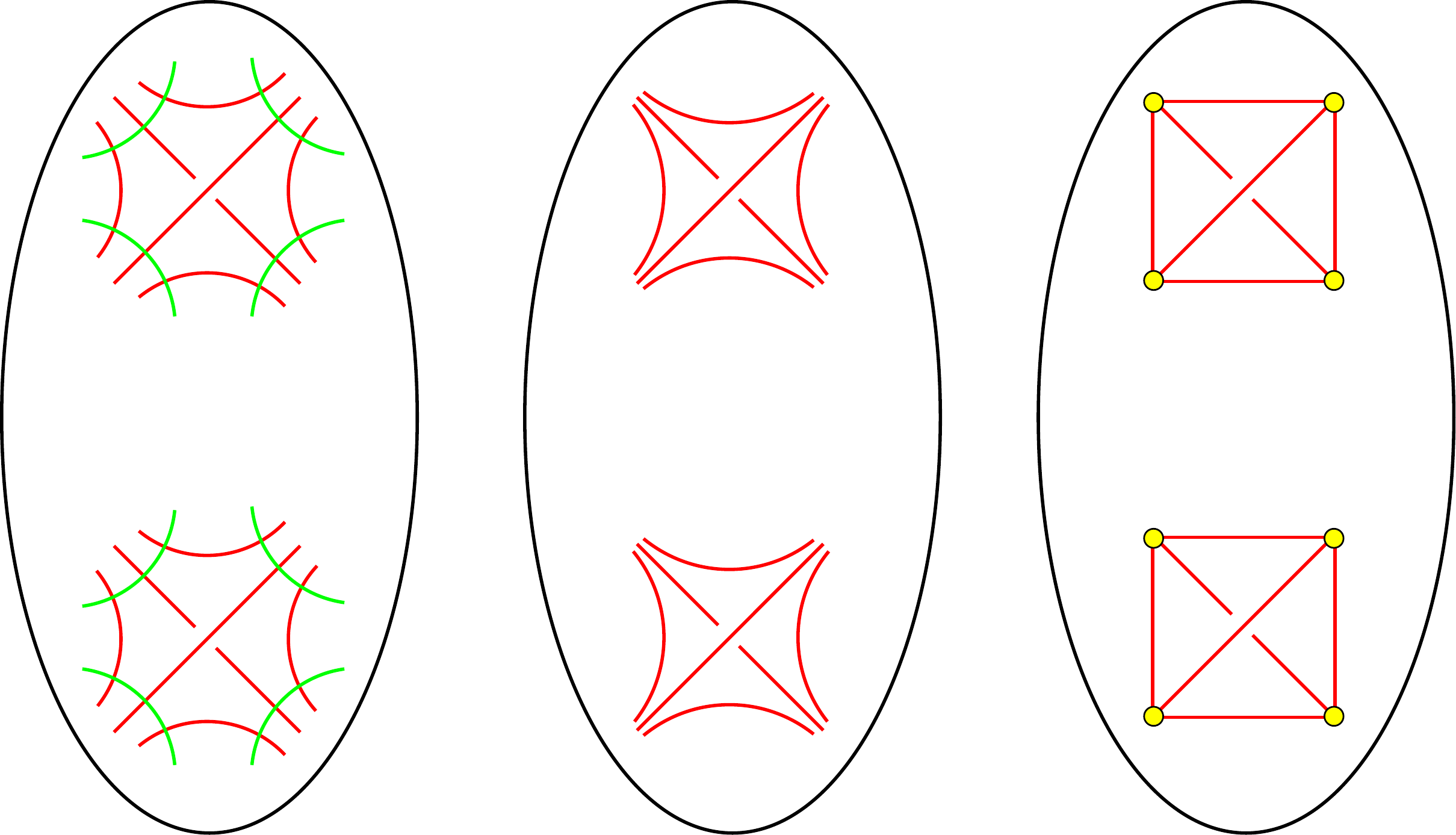}
\nota{The evolution of the singular locus $\Sigma$ of $W_t$. When $t\in (t_1,1)$ the singular locus $\Sigma$ consists of 12 red cone-tori (with two singular points) and 8 green cone-spheres (with three singular points) that intersect transversely precisely in their 24 singular points (left). When $t=t_1$ the cone-spheres disappear to infinity and the 12 cone-tori transform into punctured tori: triples of punctures of distinct tori go to the same cusp in $W_{t_1}$ (centre). When $t\in (t_2,t_1)$ the cusps in $W_{t_1}$ are filled with small simple closed geodesics and each twice-punctured torus transforms into a twice-holed compact torus with geodesic boundary consisting of two of these small geodesics; twice-holed tori and closed geodesics are represented as red edges and yellow vertices, respectively (right). The evolution continues with the interval $(0,t_2]$, but we do not draw it here.}\label{sphere_tori:fig}
\end{figure}

\subsection*{The cusps}
Recall the notation introduced in Definition \ref{Sn:defn}. The \emph{type} of a cusp is the homeomorphism type of a Euclidean cone 3-manifold section (we only determine the homeomorphism type, not the isometry type.)

\begin{prop}
For every $t\in (0,1]$ the hyperbolic cone four-manifold $W_t$ has 12 cusps of three-torus type, plus some additional cusps only at the critical times:
\begin{itemize}
\item when $t=1$ there are 12 additional cusps of three-torus type,
\item when $t=t_1$ there are 8 additional cusps of type $S^2(\frac{2\pi}3) \times S^1$, 
\item when $t = t_2$ there are 8 additional cusps of type $S^3(2\arccos \frac 13)$.
\end{itemize}
\end{prop}
\begin{proof}
Every ideal vertex $v$ of $P_t$ has a Euclidean link $\Delta$, a Euclidean polyhedron whose faces are coloured by the walls in $P_t$ they are contained in.
Each ideal vertex of $P_t$ gives rise to some cusps in $W_t$ whose Euclidean sections are obtained by mirroring $\Delta$ according to the colours.
We refer to Figure \ref{links_eucl:fig}.
Here are the details:

\begin{itemize}
\item For every $t\in(0,1)$ the polytope $P_t$ has 12 ideal vertices $v$ whose link is a parallelepiped, with opposite faces coloured with P, N, and L. Each parallelepiped gives rise to a cusp of three-torus type.
\item When $t=1$ the 24-cell $P_1$ has 12 more ideal vertices, identical to the 12 analysed above.
\item When $t=t_1$, the polytope $P_t$ has 8 additional ideal vertices, whose link is a right prism with triangular base. The two base triangles are coloured in N and L, while the lateral faces have P. By mirroring we get the 8 additional cusps of type $S^2(\frac{2\pi}3) \times S^1$.
\item When $t=t_2$, the polytope $P_t$ has 2 additional ideal vertices, whose link is a regular tetrahedron $\Delta$, with all faces coloured with P. By mirroring we get 8 cusps of type $S^3(2\arccos \frac 13)$. (If we mirror along a colour that is not there, we just take two disjoint copies of the object, and this applies here twice to the missing colours L and N.)
\end{itemize}
The proof is complete.
\end{proof}

\subsection*{The surgeries}
At the critical times $1,t_1$, and $t_2$ the cone-manifold $W_t$ changes by some surgeries that we now analyse. Recall that $W_1$ is a cusped hyperbolic four-manifold with 24 cusps and no singularities. As usual, we start with $W_1$ and we run $t$ backwards.

\begin{prop}
As soon as $t<1$, the cone-manifold $W_t$ modifies from $W_1$ by Dehn filling twelve cusps with twelve red cone-tori.
\end{prop}

Topologically, each of these 12 cusp is diffeomorphic to $S^1\times S^1\times S^1 \times [0,+\infty)$ and is replaced by a ``solid torus'' $S^1 \times S^1 \times D^2$. Each new red cone-torus is a core $S^1\times S^1 \times \{0\}$ of one such solid torus: its area $4\pi - 8\varphi$ and its cone angle $2\theta$ are both arbitrarily small when $t$ is close to 1, and they increase as $t$ tends to $t_1$, like in the familiar three-dimensional hyperbolic Dehn filling picture. When $t\to t_1$ the cone angle $2\theta$ tends to $\frac {2\pi}3$.

Recall that the singular set $\Sigma$ contains also 8 green cone-spheres whose cone angles vary from $2 \pi$ to $0$ as $t$ goes from 1 to $t_1$.

\begin{prop}
At the critical time $t_1$ the 8 green cone-spheres are drilled and create 8 new cusps. As soon as $t<t_1$, the 8 cusps are filled with 8 yellow small closed geodesics.
\end{prop}

Every green cone-sphere has a tubular neighborhood homeomorphic to $S^2\times D^2$, and the drilling substitutes it with a cusp homeomorphic to $S^2 \times S^1 \times [0,+\infty)$. Recall that we are in a cone-manifold (or orbifold) context: the $S^2$ factor is the \emph{flat} cone sphere $S^2(\frac {2\pi}3)$, hence $S^2\times S^1$ is a flat cone three-manifold.

As soon as $t<t_1$, each such cusp is substituted with a $D^3 \times S^1$. The new core closed curve $\{0\} \times S^1$ is a small closed geodesic.

\begin{rem}
The substitution of a $S^2$ (with trivial normal bundle) with a $S^1$ is a common topological surgery in dimension four: it consists in replacing an embedded $S^2\times D^2$ with $D^3 \times S^1$, glued along the same boundary $S^2\times S^1$.  
We have just discovered an example where the surgery may be realized as a smooth path of hyperbolic cone four-manifolds. Both the cores $S^2$ and $S^1$ are geodesic all along the path. We call this path a \emph{hyperbolic Dehn surgery} in Theorem \ref{main2:teo}.
\end{rem}

A similar, but different, kind of hyperbolic surgery arises at the next critical time. We start by noticing the following.

\begin{prop}
When $t\in (t_2,t_1)$ the manifold $W_t$ contains four geodesic copies of the hyperbolic cone three-manifold $S^3(2\theta)$, that collapse when $t\to t_2$. At the critical time $t_2$ these are drilled and create 8 new cusps. As soon as $t<t_2$, the 8 cusps are filled with 8 four-balls.
\end{prop}
\begin{proof}
When $t\in (t_2,t_1)$ each letter wall $\l G$ and $\l H$ is a hyperbolic regular tetrahedron with dihedral angle $\theta$; when mirrored in $W_t$, these walls form four geodesic copies of $S^3(2\theta)$. When $t \to t_2$ these walls collapse to ideal vertices, which become finite as soon as $t<t_2$.
\end{proof}

Each geodesic $S^3(2\theta)$ has a tubular neighborhood homeomorphic to $S^3 \times [-1,1]$, and the drilling substitutes it with two cusps, each homeomorphic to $S^3 \times [0,+\infty)$. Here $S^3$ is the flat $S^3(2\theta)$, since $\cos (\theta) = \frac 13$ at the critical time $t_2$. 

As soon as $t<t_2$, each cusp is filled with a $D^4$.
We will determine the topology of $W_t$ when $t<t_2$ in the next section.

\begin{rem}
The substitution of a $S^3$ (with trivial normal bundle) with a $S^0$ is another common topological surgery in dimension four: we substitute $S^3 \times D^1$ with $D^4 \times S^0$, glued along the same boundary $S^3\times S^0$, and we have just discovered that it can also be realized as a smooth path of hyperbolic cone-manifolds. It is also called a hyperbolic Dehn surgery in Theorem \ref{main2:teo}.
\end{rem}

\begin{rem} \label{topology:rem}
The topology of $W_t$ in the last interval $(0,t_2)$ is surprisingly simple: we will show in Proposition \ref{product:prop} below that $W_t$ is diffeomorphic to a product $C\times S^1$, where $C$ is some cusped hyperbolic 3-manifold, when $t\in (0,t_2)$. 

Therefore the manifold $W_t$ for $t\in (t_2,t_1)$ is obtained from $C\times S^1$ by a simple surgery, the replacement of four copies of $S^0$ with four $S^3$, and hence $W_t$ is diffeomorphic to $(C\times S^1) \#_4 (S^1\times S^3)$ when $t\in (t_2,t_1)$. 

Finally, the manifold $W_t$ for $t\in (t_1,1)$ is obtained from the latter by one more surgery, that replaces eight copies of $S^1$ with eight $S^2$. We can build a five-dimensional film interpretation of this topological process: start with $C\times D^2$, then add four 1-handles, and eight 2-handles.
\end{rem}

\subsection*{Orbifolds}
We have already noted that $W_t$ is an orbifold at $t=t_1$, whose singular locus is a surface with cone angle $\frac{2\pi}3$. There is also one more orbifold in the family $W_t$, of a quite different nature: at the time $t=\bar t$ the singular set $\Sigma$ is a \emph{foam} (a two-dimensional complex with generic singularities) with all cone angles $\pi$; the singularities are locally like those of the double of a right-angled polytope. 

Summing up, the cone-manifold $W_t$ is an orbifold at the times
$$1, \quad t_1 = \sqrt{\frac 35}, \quad \bar t=\sqrt{\frac 13}.$$
These correspond to the times when $P_t$ is a Coxeter polytope. In fact the colouring technique furnishes regular orbifold coverings such that 
$$P_1=W_1/_{(\mathbb{Z}/_{2\mathbb{Z}})^3}, \qquad P_{t_1}=W_{t_1}/_{(\mathbb{Z}/_{2\mathbb{Z}})^3}, \qquad P_{\bar t}=W_{\bar t}/_{(\mathbb{Z}/_{2\mathbb{Z}})^3}.$$
The three orbifolds $W_t$ are arithmetic, since $P_t$ is (see Section \ref{quotient:section}). Moreover 
$$\chi(W_1) = 8, \qquad \chi(W_{t_1}) = 8, \qquad \chi (W_{\bar t}) = 5$$
as a consequence of Propositions \ref{euler_t1:prop} and \ref{euler_bart:prop}. 
We will prove in Proposition \ref{product:prop} that the underlying space of $W_{\bar t}$ is topologically a product $C\times S^1$.

\subsection*{The final degeneration}
We now study $W_t$ as $t\to 0$ and show that $W_t$ degenerates to a hyperbolic three-manifold.

\begin{figure}
\vspace{.5 cm} 
\includegraphics[angle=270]{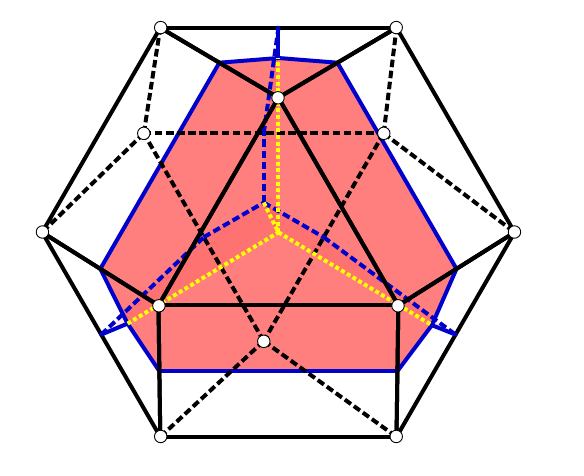}
\nota{
In the last time interval $(0,t_2)$, the 4 odd (or even) positive walls of the polytope $P_t$ form an ideal right-angled cuboctahedron $Q$ with centre $v$, pleated along six red pentagons.
In the picture, the edges of $Q$ are black.
The faces of $Q$ are divided as follows: (i)
four ideal triangles, each a common face of an odd  wall $\mathbf{i}^+$ with its negative counterpart $\mathbf{i}^-$;
(ii) four ideal triangles, each subdivided by the red pleats into three quadrilaterals, all faces of \emph{the same} even negative wall;
(iii) six ideal quadrilaterals, each subdivided by a red pleat into two quadrilaterals, both faces of \emph{the same} letter wall.  
The edges of each positive wall are coloured as follows:
the black edges are contained in edges of $Q$, the blue edges are contained in faces of $Q$ (they are the red edges of the even negative and letter walls in Figure \ref{walls_medi:fig}), the yellow edges are contained in the interior of $Q$ and intersect in the centre $v$ of $Q$.
Each pleating pentagon is a red face of an odd positive wall and in the picture has three blue edges and two yellow edges.
}\label{altogether:fig}
\end{figure}

We already know that the polytope $P_t$ tends to the three-dimensional ideal right-angled cuboctahedron $P_0$ shown in Figure \ref{altogether:fig}. The cuboctahedron $P_0$ can be naturally coloured with two colours, one assigned to the triangles and the other to the quadrilaterals. Let $C$ be the hyperbolic three-manifold constructed from $P_0$ by mirroring it according to this colouring: the three-manifold $C$ is tessellated into four copies of $P_0$ and we call it the \emph{cuboctahedral manifold}. 
The cuboctahedral manifold has 12 toric cusps, one for each ideal vertex of $P_0$.

We now completely determine the topology of $W_t$ in the last interval $(0,t_2)$, as anticipated in Remark \ref{topology:rem}.

\begin{prop}\label{product:prop}
When $t\in(0,t_2)$, the manifold $W_t$ is diffeomorphic to $C\times S^1$.
\end{prop}
\begin{proof}
When $t\in (0,t_2)$, we see from Figure \ref{walls_plus:fig} that the 4 positive walls of the same parity (say odd: hence $\p 1, \p 3, \p 5$, and $\p 7$) intersect in a vertex $v$ of $P_t$ whose link is a regular tetrahedron with dihedral angles $\theta$: there are two vertices like that, see Proposition \ref{combgrandi}-(4).

We now consider these four positive walls altogether as a single wall $Q$, pleated along some faces: Figure \ref{altogether:fig} shows that $Q$ is a cuboctahedron, pleated along six red pentagons with pleating angle $\theta$. Since we are interested only in the topology of $W_t$, we may ignore the pleating (that is, we pretend that $\theta=\pi$).

Combinatorially, the polytope $P_t$ is isomorphic to the prism $Q \times I$ over $Q$. The horizontal walls $Q \times \{0,1\}$ are the two positive (even and odd) cuboctahedra. The lateral walls are: 
\begin{itemize}
\item the 6 letter walls, that are prisms over the ideal quadrilaterals of $Q$, and 
\item the 8 negative walls, that are prisms over the ideal triangles.
\end{itemize}
(Remember that we ignore the pleats and treat two faces of a wall adjacent along a red edge as the same face). The manifold $W_t$ is obtained from $Q\times I$ via the colouring technique and is hence diffeomorphic to $C\times S^1$.
\end{proof}

Now, recall the fixed cuboctahedron $P_0=P_t\cap \matH^3$ described et the end of Section \ref{polytope:sec}.
As $t\to0$, the non-right dihedral angles of $P_t$ tend to $\pi$ and the polytope collapses to the polyhedron $P_0$.
Correspondingly, when the cone-angles of the hyperbolic cone-manifold $W_t$ tend to $2\pi$, the hyperbolic structure degenerates to that of the cuboctahedral manifold $C$, in a way that we now state precisely.

Let a \emph{holonomy representation} of a hyperbolic cone-manifold be a holonomy representation of its regular locus (the representation is unique up to conjugation). Here our construction furnishes for every $t\in(0,1)$ a holonomy representation
$$\rho_t:\pi_1(W_t\backslash\Sigma_t)\longrightarrow\mathrm{Isom}(\mathbb{H}^4).$$
Let $\rho\colon \pi_1(C) \to \Iso(\matH^3) < \Iso(\matH^4)$ be the faithful and discrete representation of the cuboctahedral manifold $C$.
 
\begin{prop}
As $t\to 0$, the representation $\rho_t$ converges algebraically to a representation $\rho_0$ with ${\rm Im}(\rho_0) = {\rm Im}(\rho)$.
\end{prop}
\begin{proof}
For every $g\in \pi_1(W_t \setminus \Sigma_t)$ the isometry $\rho_t(g)$ is a composition of reflections along the hyperplanes defining the polytope $P_t$. 

As $t\to 0$, each half-space $\p 0, \m 0, \ldots \l E, \l F$ converges to some half-space whose boundary hyperplane is either $\matH^3$ or orthogonal to $\matH^3$. This shows that $\rho_t(g)$ converges to a $\rho_0(g)$ contained in the image of $\rho$.
By analyzing the generators of $\pi_1(W_t \setminus \Sigma_t)$ we get ${\rm Im}(\rho_0) = {\rm Im}(\rho)$.
\end{proof}

This degeneration is similar to the one famously described by Thurston \cite{bibbia} where a family of hyperbolic cone-structures on a Seifert fibered manifold degenerates to the hyperbolic structure of the base orbifold as the cone-angle approaches $2\pi$.

The proof of Theorem \ref{main2:teo} is complete.

\subsection*{The family is analytic}
We remark that the deformation $W_t$ is \emph{analytic} in the following sense: the holonomy $\rho_t(\gamma)$ of an element $\gamma\in\pi_1(W_t\setminus \Sigma_t)$ varies analytically in $t$, because it is a product of reflections along hyperplanes dual to space-like vectors that vary analytically in $t$.

Note that the topology of $W_t \setminus \Sigma_t$ changes only at the critical times $1$ and $t_2$. One can check that there is a natural embedding $\pi_1(W_{t_2 - \varepsilon}\setminus \Sigma_{t_2-\varepsilon}) \hookrightarrow \pi_1(W_{t_2 + \varepsilon}\setminus \Sigma_{t_2+\varepsilon})$, so the above definition actually makes sense also when $t$ crosses $t_2$.

\subsection{The deforming cone-manifolds $N_t$}
In the previous section we have constructed an interpolation between a manifold $W_1$ and an orbifold $W_{t_1}$ through hyperbolic cone-manifolds $W_t$ with $t \in [t_1,1]$, whose singular locus $\Sigma$ is an immersed surface with varying cone angles. This interpolation is similar to the one required by Theorem \ref{main:teo}, the main difference being that $W_{t_1}$ is ``only'' an orbifold and not a manifold. In order to prove Theorem \ref{main:teo}, in this section we now need to promote the orbifold $W_{t_1}$ to a manifold. To do so, we construct a new manifold $N_t$ by assembling some copies of $P_t$ via a more complicated pattern than the one realizing $W_t$.

The orbifold $W_{t_1}$ contains a singular red surface (that consists of some punctured tori) with cone angle $\frac {2\pi} 3 = 2\theta$. We get this cone angle because every red quadrilateral in $P_t$ has dihedral angle $\theta$, and meets 2 copies of $P_t$ in $W_t$. We now modify the construction of the previous section, so that each quadrilateral will meet 6 copies of $P_t$: this will make a total cone angle $6\theta = 2\pi$ at $t_1$ and hence the singularity will disappear.

To this purpose, we still use the P/N/L colouring of the walls of $P_t$, we still mirror $P_t$ along N and L, but we glue the positive walls altogether with a more complicate pattern, that ensures that each red quadrilateral in the resulting complex has valence 6 instead of 2. This more complicate pattern is constructed by transposing into this context the famous triangulation of the figure-eight complement with two tetrahedra: the nice feature of this triangulation is that all edges have valence 6, and this is exactly what we need here.

\subsection*{The figure-eight knot pattern}
We start by studying the symmetries of $P_t$.

\begin{lemma}\label{symm_positive:lemma}
For every bijection
$$\sigma:\lbrace\p1,\p3,\p5,\p7\rbrace\longrightarrow\lbrace\p0,\p2,\p4,\p6\rbrace$$
there exists a unique symmetry $s\in K$ of the polytope $P_t$ such that  $s(\p i)=\sigma(\p i)$ for every $\p i\in\lbrace\p1,\p3,\p5,\p7\rbrace$ and for all $t\in(0,1]$.
\end{lemma}
\begin{proof}
Recall from Section \ref{symmetries:sec} the group of symmetries $K$ of $P_t$ and its subgroup $H$. The group $H$ acts 
on both sets $\lbrace\p1,\p3,\p5,\p7\rbrace$ and $\lbrace\p0,\p2,\p4,\p6\rbrace$ as their permutation group, and the roll symmetry $R$ exchanges the two sets. \end{proof}

Note that, when $t\in(t_1,1)$, each positive wall is adjacent to all the other positive walls of the same parity and there are no triple intersections among positive walls, see Figure \ref{walls:fig}.
Therefore, the four odd (resp.~even) positive walls are arranged with the combinatorial pattern of a three-dimensional regular ideal tetrahedron:
each wall corresponds to a face of the tetrahedron, while each red quadrilateral (intersection of two walls) corresponds to an edge of the ideal tetrahedron.

\begin{figure}
\labellist
\small\hair 2pt
\pinlabel \textcolor{blue}{$\p 3$} at 30 80
\pinlabel {$\p 1$} at 30 60
\pinlabel \textcolor{blue}{$\p 7$} at 140 95
\pinlabel {$\p 5$} at 130 115

\pinlabel \textcolor{blue}{$\p 4$} at 220 80
\pinlabel {$\p 0$} at 220 60
\pinlabel \textcolor{blue}{$\p 2$} at 330 95
\pinlabel {$\p 6$} at 320 115
\endlabellist
\centering
\includegraphics[width=9 cm]{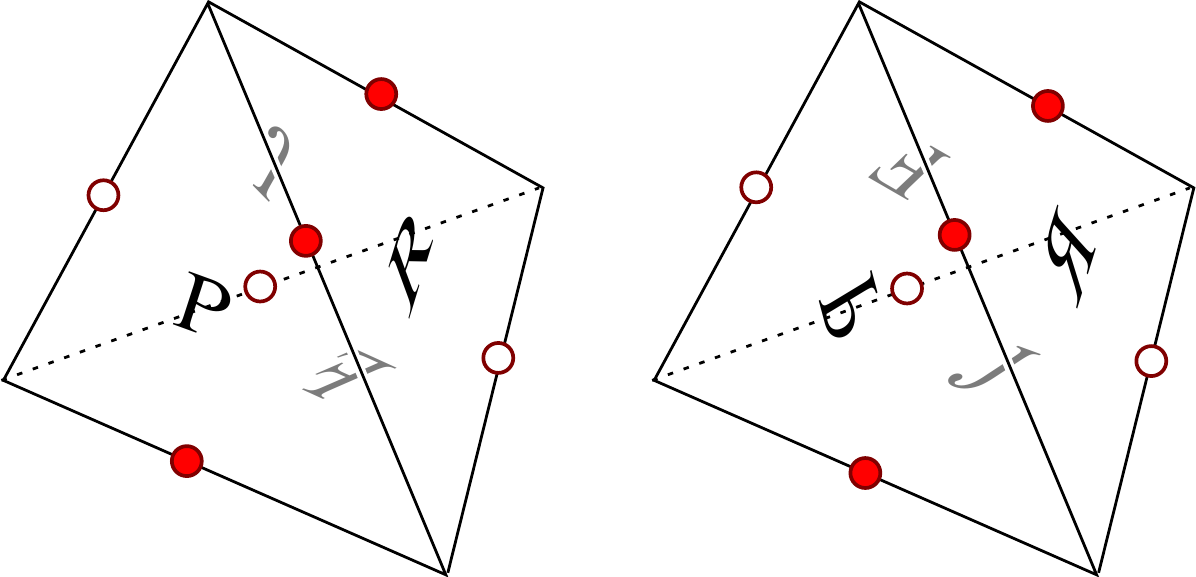}
\nota{This is the ideal triangulation of the figure eight knot complement. We identify the odd (even) positive walls with the four faces of the left (right) tetrahedron, as shown here. Front faces are labeled in black, and back faces in blue. The two resulting edges of the triangulation have valence six and are dotted in red and white.}\label{triangulation_figure8:fig}
\end{figure}

Consider the two ideal regular tetrahedra of Figure \ref{triangulation_figure8:fig}. We identify each four-uple of walls
$$\lbrace\p1,\p3,\p5,\p7\rbrace,\qquad \lbrace\p0,\p2,\p4,\p6\rbrace$$
with the faces of the left and right ideal tetrahedron, as shown in the figure.

The letters $F$, $J$, $P$, $R$ in the figure determine a well-known face-pairing of the two tetrahedra: this is the face-pairing giving rise to the ideal triangulation of the figure-eight knot complement. It has the following nice combinatorial features:
\begin{itemize}
\item each edge in the resulting combinatorial triangulation has valence 6,
\item the return maps around the two edges are trivial.
\end{itemize} 

The face-pairing of Figure \ref{triangulation_figure8:fig} induces a wall-pairing 
$$S = \big\{s_{\p 1}, s_{\p 3}, s_{\p 5}, s_{\p 7}\big\}$$
between the odd and even positive walls of $P_t$. Each $s_{\p i}$ is an isometry from $\p i$ to some even positive wall, determined as follows. Every $\p i \in \{\p1,\p3,\p5,\p7\}$ corresponds to a face of the left tetrahedron, which is glued to some face of the right one according to the pattern shown in Figure \ref{triangulation_figure8:fig}.
The gluing extends to a unique isometry between the two tetrahedra, that induces a bijection
$$\sigma \colon \lbrace\p1,\p3,\p5,\p7\rbrace \longrightarrow \lbrace\p0,\p2,\p4,\p6\rbrace.$$
The bijection in turns determines a symmetry $s_{\p i}$ of $P_t$ by Lemma \ref{symm_positive:lemma},
that restricts to an isometry between $\p i$ and $\sigma(\p i)$.
Note that the face-pairing $S$ glues the wall $\l G$ to $\l H$ exactly with the pattern of Figure \ref{triangulation_figure8:fig}.

For instance, the symmetry $s_{\p 1}$ sends $\p 1$ to $\p 0$ (the two faces in Figure \ref{triangulation_figure8:fig} have the same letter $P$), and by looking at the orientation of the letter $P$ we also see that $s_{\p 1}$ acts as follows:
$$\p 3 \to \p 4, \quad \p 7 \to \p 6, \quad \p 5 \to \p 2.$$
This determines the isometry $s_{\p 1}$ between the walls $\p 1$ and $\p 0$. Following this recipe, it is not difficult to check that the wall-pairings in $S$ are restrictions of the following isometries of $\matH^4$ and symmetries of $P_t$:
\begin{align*}
s_{\p1}\colon (x_0,x_1,x_2,x_3,x_4) & \longmapsto(x_0,x_3,x_1,-x_2,-x_4), \\
s_{\p3}\colon (x_0,x_1,x_2,x_3,x_4) & \longmapsto(x_0,x_2,x_3,-x_1,-x_4), \\ 
s_{\p5}\colon (x_0,x_1,x_2,x_3,x_4) & \longmapsto(x_0,-x_3,x_1,x_2,-x_4), \\
s_{\p7}\colon (x_0,x_1,x_2,x_3,x_4) & \longmapsto(x_0,x_2,-x_3,x_1,-x_4).
\end{align*}
Note that all such symmetries are orientation-preserving. This implies that the resulting cone-manifold $N_t$ (defined in the following paragraph) will \emph{not} be orientable.

\subsection*{The manifolds $N_t$}
Finally, we are ready to define the desired cone-manifold $N_t$.
\begin{defn}\label{4copie:def}
Let $N_t$ be the hyperbolic cone-manifold obtained by picking four copies $P_t^{ij}, i,j\in\{0,1\}$ of the polytope $P_t$ and by pairing their walls as follows:
\begin{enumerate}
\item identify every L wall in $P_t^{0j}$ with the corresponding wall in $P_t^{1j}$;
\item identify every N wall in $P_t^{i0}$ with the corresponding wall in $P_t^{i1}$;
\item identify the P walls in $P_t^{ij}$ in pairs via the wall-pairing $S$. 
\end{enumerate}
In (1) and (2) we identify the corresponding walls using the identity map. We do the identifications (1), (2), and (3) for all $i,j \in \{0,1\}$.
\end{defn}

The hyperbolic cone-manifold $N_t$ is defined for all $t\in (0,1]$, but we will be interested essentially in the interval $[t_1,1]$. 

The idea lying behind this construction is that everything should work locally like with $W_t$, except that now every red quadrilateral is incident to 6 copies of $P_t$ instead of 2 and hence $N_{t_1}$ will be a manifold and not an orbifold.
We now analyse $N_t$ carefully. 

\subsection*{The singular set $\Sigma$} 
As for $W_t$, we start by analyzing the singular set $\Sigma$. 

\begin{prop} \label{solo2:prop}
The singular set $\Sigma$ of $N_t$ is the union of the green and red faces of the four copies of $P_t$.
\end{prop}
\begin{proof}
Let $x\in \partial P_t$ be a point that does not lie in a green or red face. The point $x$ is contained in one, two, or three walls that are pairwise at right angles and have distinct colors. Since the identifications of the walls L and N are just mirrors, and that of the walls P preserves the colourings L and N, one sees easily that $x$ becomes a smooth point in $N_t$. 
\end{proof}

As for $W_t$, to understand the singular set $\Sigma$ of $N_t$ it suffices to analyse the green and red faces of $P_t$. 

\begin{prop} \label{Nt:prop}
When $t\in (t_1,1)$ the singular set $\Sigma$ is a geodesically  immersed surface $\Sigma = T_0 \cup T_0' \cup T_1$, union of two disjoint red cone-tori $T_0 \sqcup T_0'$ and a green cone-torus $T_1$, with cone angles $6\theta$ and $4\varphi $ respectively, intersecting in four points. The three tori have trivial normal bundles.
\end{prop}
\begin{proof}
To understand $\Sigma$, we analyse all the finite vertices $v$ in $P_t$ and determine the unit tangent space of their images in $N_t$, as in the proof of Proposition \ref{Sigma:W:prop}.
We refer to Figure \ref{links:fig}. 

By Proposition \ref{combpiccoli} there are two types of vertices $v$ to analyse, with spherical link $\Delta = I_{\frac \pi 2} * I_\theta$ or $I_\varphi*I_\theta$. The two types are considered similarly, so we only focus on $I_{\varphi} * I_\theta$. The 4 faces of $\Delta$ are coloured as P, P, N, L.
After mirroring along negative walls, the link becomes $I_{2\varphi}*I_\theta$ and then, mirroring along the letter walls, we get a link $C_{4\varphi}*I_\theta$.   

The join $C_{4\varphi}*I_\theta$ has two ``faces'' coloured as P, each isometric to a spherical disc with a cone point $4\varphi$ in its centre, tessellated into four triangles. As opposite to $W_t$, the P faces here are not doubled: they are paired according to the pattern of Figure \ref{triangulation_figure8:fig}. Since every edge has valence 6 in this pattern, 6 copies of $C_{4\varphi}*I_\theta$ are glued cyclically. Since the return map around every edge in Figure \ref{triangulation_figure8:fig} is the identity (and not an edge reversal), the 6 copies are glued cyclically also with a trivial return map, giving rise to $C_{4\varphi}*C_{6\theta}$. 

We have discovered that the link of $v$ is $S^1*C_{6\theta}$ or $C_{4\varphi}*C_{6\theta}$, according to the vertex type. We deduce that $\Sigma$ is an immersed geodesic surface, made up of embedded orthogonal red and green surfaces having cone angles $6\theta$ and $4\varphi$. 

A simple analysis on the topology of $\Sigma$ shows that it consists of:
\begin{itemize}
\item two red cone-tori as in Figure \ref{polygons1:fig}-(bottom-left), each with two cone points of angle $4\varphi$, as we had in $W_t$; 
\item one green cone-torus with four cone points of angle $6\theta$, which decomposes into eight green equilateral triangles like the single torus cusp section of the figure-eight knot complement triangulation in Figure \ref{triangulation_figure8:fig}. 
\end{itemize}
It is also quite easy to check that their normal bundles are trivial. 
\end{proof}

\begin{cor}
When $t\in [t_1,1]$ the family $N_t$ interpolates analytically between two cusped hyperbolic manifolds $N_1$ and $N_{t_1}$. 
\end{cor}
\begin{proof}
The cone-manifolds $N_1$ and $N_{t_1}$ have no singularities, since $4\varphi$ and $6\theta$ are either 0 or $2\pi$ for that values.
\end{proof}

In the interpolation, the red tori are drilled at $t=1$ and the green tori are drilled at $t=t_1$, producing new cusps. 

\subsection*{The cusps}We now study the cusps of $N_t$. 

\begin{prop}\label{cusps_Nt:prop}
The cone-manifold $N_t$ has
\begin{itemize}
\item three cusps at $t=t_1$,
\item two cusps when $t\in (t_1,1)$,
\item four cusps at $t=1$.
\end{itemize}
The section of each cusp is a flat three-torus.
\end{prop}
\begin{proof}
We refer to Figure \ref{links_eucl:fig} for the links of the ideal vertices of $P_t$.

For $t\in(t_1,1)$, consider the 12 ideal vertices $v$ of $P_t$. The link of $v$ is a parallelepiped with faces coloured in P, N, L. Recall that opposite faces share the same colour, and if their colour is P or N, then they have opposite parity. By mirroring the parallelepiped along N and L we get $S^1\times S^1\times I$. The pairing of the P faces then form some cycles. Each cycle gives a cusp and is a flat mapping torus with fibre $S^1\times S^1$.

We now determine these cycles and the resulting mapping tori.
We denote the 12 parallelepipeds as
$$C_{01},\ C_{21},\ C_{61},\ C_{03},\ C_{23},\ C_{43},\ C_{05},\ C_{45},\ C_{65},\ C_{27},\ C_{47},\ C_{67},$$ 
where $C_{ij}$ is the link of the ideal vertex of $P_t$ adjacent to the four numbered walls $\ensuremath{\boldsymbol i}^\pm$ and $\ensuremath{\boldsymbol j}^\pm$ (and two letter walls, see Proposition \ref{combpiccoli}). A computation shows that there are two cycles: 
$$C_{01}\cup_{s_{\p1}}C_{03}\cup_{s_{\p3}}C_{23}\cup_{s_{\p3}}C_{27}\cup_{s_{\p7}}C_{45}\cup_{s_{\p5}}
C_{61}\cup_{s_{\p1}}C_{01},$$
$$C_{05}\cup_{s_{\p5}}C_{65}\cup_{s_{\p5}}C_{67}\cup_{s_{\p7}}C_{47}\cup_{s_{\p7}}
C_{43}\cup_{s_{\p3}}C_{21}\cup_{s_{\p1}}C_{05}.$$ 
Therefore there are two cusps. The fact that each cycle has an even number of elements implies that both cusp sections are three-tori. Indeed, each $s_{\p i}$ glues the odd P rectangle of a parallelepiped to the even P rectangle of the subsequent one; 
the opposite edges of each such rectangle are both coloured in N or L, and $s_{\p i}$ preserves the colouring but exchanges the parity of N; it also inverts the natural orientation of the rectangle; however, since we compose an even number $6$ of them, we get a mapping torus with monodromy ${\matr {-1}001}^{6} = \matr 1001$.

The additional cusps are obtained by drilling tori having trivial normal bundles, therefore they are also of three-torus type.
\end{proof}
\begin{rem}
We have here a third independent argument to show that $\mathrm{Vol}(P_{t_1})=\mathrm{Vol}(P_1) =\frac{4\pi^2}{3}$, after Proposition \ref{euler_t1:prop} and Proposition \ref{vol_piccoli}. The manifold $N_{t_1}$ is topologically obtained from $N_1$ by Dehn surgeries (first filling and then drilling along different tori), and these operations do not modify the Euler characteristic of a four-manifold. Therefore $\chi (N_{t_1}) = \chi(N_1)$, which implies $\Vol(N_{t_1}) = \Vol(N_1)$ and hence $\Vol(P_{t_1}) = \Vol(P_1)$.

Actually, an analogous reasoning could have been done in the previous section for $W_1$ and $W_{t_1}$ in the orbifold context.\end{rem}

\begin{rem} \label{figure8:rem}
The hyperbolic manifold $N_{t_1}$ contains a geodesic hypersurface diffeomorphic to the figure-eight knot complement. It comes from gluing together the walls $\l G$ and $\l H$ in $P_t^{00}$, which are regular ideal tetrahedra when $t=t_1$. This confirms the recent discovery that the figure-eight knot complement embeds geodesically \cite{S}.
\end{rem}

The cone-manifold $W_t$ is tessellated into eight copies of $P_t$, while $N_t$ is tessellated into only four. Therefore we have $\chi(N_1) = \chi(N_{t_1}) = 4$. (Recall that $\Vol(N) = \frac{4\pi^2}{3}\chi(N)$ for every hyperbolic 4-dimensional orbifold $N$.) 

In the next section we will quotient $N_t$ to a new cone-manifold $M_t$ and further cut the Euler characteristic by two.

\subsection*{Another Dehn filling}
We only say few words on the cone-manifolds $N_t$ when $t< t_1$. We note that as soon as $t<t_1$ the cone-angle $6\theta$ is greater than $2\pi$ and $N_t$ is not supported on a manifold any more. Indeed, as soon as $t<t_1$, the topology of $N_t$ changes from that of $N_{t_1}$ by a Dehn filling that is different from the ones already considered and that was mentioned in the introduction: it consists of the collapsing of one $S^1\times S^1$ factor in the $S^1\times S^1\times S^1$ shape of the cusp, which produces a small simple closed geodesic (as was mentioned in the introduction). This type of Dehn filling was already considered in \cite{FM,FM2}. 

\subsection{The manifolds $M_t$}
The family $N_t$ with $t\in [t_1,0]$ is quite like the $M_t$ required for proving Theorem \ref{main:teo}, except that the singular set $\Sigma$ contains two red tori instead of one and a green torus instead of a green Klein bottle (see Proposition \ref{Nt:prop}). We now construct $M_t$ as a quotient $M_t= N_t/_\iota$ where $\iota$ is an appropriate fixed-point-free isometric involution that interchanges the two red tori (and the two cusps of $M_t$).

To construct $\iota$, we exploit the well-known fact that the figure-eight knot complement has a fixed-point-free isometric involution $\rho$ that permutes the two ideal tetrahedra in Figure \ref{triangulation_figure8:fig} and the two edges, producing the non-orientable Gieseking manifold as a quotient (with a single tetrahedron and a single edge). Looking at Figure \ref{triangulation_figure8:fig}, the involution $\rho$ sends the left tetrahedron to the right by acting on the faces as follows:
$$\p 1 \to \p 4, \quad \p 3 \to \p 6, \quad \p 5 \to \p 2, \quad \p 7 \to \p 0.$$
This corresponds to the following isometric involution of $P_t$:
$$r\colon (x_0, x_1, x_2, x_3, x_4) \longmapsto (x_0, -x_1, -x_2, -x_3, -x_4).$$
The fact that $\rho$ is an isometry of the figure-eight knot complement implies that $\rho$ preserves the identifications of the faces in Figure \ref{triangulation_figure8:fig}, and this translates into the following equalities for $r$ that one can verify directly, since $s_{\p 7} = s_{\p 1}^{-1}$, $s_{\p 5} = s_{\p 3}^{-1}$, and $r$ commutes with them:
$$r = s_{\p 7} r s_{\p 1} = s_{\p 5} r s_{\p 3} = s_{\p 3} r s_{\p 5} = s_{\p 1} r s_{\p 7}.$$
These equalities say that $r$ preserves the identification between the positive walls of $P_t$, and since $r$ also preserves the N and L colours it descends to an isometric involution $r\colon N_t \to N_t$ that acts as described on each copy $P_t^{ij}$ of $P_t$.

The involution $r\colon N_t \to N_t$ has four fixed points: the four centers of the $P_t^{ij}$ (there is no $x\in P_t$ that is identified with $r(x)$ through the wall-pairing $S$). To eliminate these fixed points, we define 
$$\iota = h\circ r$$
where $h$ is the isometric involution of $N_t$ that sends $P_t^{ij}$ to $P_t^{1-i,1-j}$ via the identity map for each $i,j \in \{0,1\}$. (The isometries $h$ and $r$ commute.) The isometry $\iota$ is fixed-point-free.
Therefore the quotient $M_t = N_t/_\iota$ is a hyperbolic cone-manifold.

The involution $\iota$ exchanges the two red tori in the singular set of $M_t$, hence the singular set $\Sigma$ of $M_t$ contains a single red torus; it acts on the green torus as a fixed-point-free orientation-\emph{reversing} involution, hence $\Sigma$ also contains a green Klein bottle, tessellated into four equilateral green triangles like in a cusp section of the Gieseking manifold. (Similar to Remark \ref{figure8:rem}, the hyperbolic manifold $M_{t_1}$ contains a geodesically embedded copy of the Gieseking manifold.)

\begin{prop}
Both $T$ and $K$ have trivial normal bundle in $M_t$.
\end{prop}
\begin{proof}
The tori $T_0, T_0', T_1$ in $N_t$ have trivial normal bundles. Therefore $T$ also has, and the normal bundle of $K$ is $(T_1 \times D^2)/_{\iota}$ where $\iota$ sends $(x,z)$ to $(i(x),-z)$. The resulting bundle is easily seen to be isomorphic to $K\times D^2$.
\end{proof}

The proof of Theorem \ref{main:teo} is complete -- it only remains to rescale and invert linearly the time parameter $t$ from $[1,t_1]$ to $[0,1]$.

\subsection{Commensurability}
We prove here the following.
\begin{prop}
The hyperbolic arithmetic four-manifolds $M_0$ and $M_1$ of Theorem \ref{main:teo} are not commensurable.
\end{prop}
\begin{proof}
We first prove that the manifolds $M_0$ and $M_1$ are commensurable to the orbifolds $P_1$ and $P_{t_1}$, respectively (recall the time reparametrisation for $M_t$ at the end of the last section). 

The manifold $M_0$ is clearly commensurable with $N_1$.
The manifold $N_1$ is constructed by gluing some identical copies of $P_1$ along some isometrical pairings of their facets. The isometrical pairings that we used are in fact all restrictions of some isometry of $P_1$, hence $M_0$ is a covering of the orbifold $P_1/_{\Iso(P_1)}$ and therefore $M_0$ and $P_1$ are commensurable. The argument for $M_1$ and $P_{t_1}$ is the same.

The thesis now follows from Proposition \ref{not:comm:prop} below.
\end{proof}

We now concentrate on the Coxeter polytopes $P_1$ and $P_{t_1}$, and actually on their quotients $Q_1$ and $Q_{t_1}$. We already know that they are both arithmetic, hence the manifolds $M_0$ and $M_1$ also are.

Recall that two subgroups $\Gamma_1, \Gamma_2 < \Iso(\matH^n)$ are \emph{commensurable} (in a wide sense) if there is a $g\in \Iso(\matH^n)$ such that the intersection of $g^{-1}\Gamma_1g$ and $\Gamma_2$ has finite index in both. This is an equivalence relation.

We briefly describe a procedure due to Maclachlan \cite{Mac} to detect the commensurability class of any arithmetic hyperbolic reflection group $\Gamma<\mathrm{Isom}(\matH^n)$ of finite co-volume. We assume for simplicity that $n=4$ and $\Gamma$ is not co-compact (thus the field of definition is $\matQ$). We also refer to \cite[Section 4]{GJK} and \cite[Section 5.1.2]{J}.

\subsection*{Notations and facts}
We use the following notations:
\begin{itemize}
\item for $a,b\in\matQ^*$, we denote by $(a,b)$ the associated quaternion algebra over $\matQ$;
\item the symbol $\otimes$ is the tensor product over $\matQ$;
\item $\mathrm{Br}(\matQ)$ is the Brauer group of the field $\matQ$;
\item for a central simple $\matQ$-algebra $B$, we let $\left[B\right]\in\mathrm{Br}(\matQ)$ be the Brauer equivalence class of $B$.
\end{itemize}

Recall that the Brauer group is an Abelian group. The group operation is given by $[B_1]\cdot[B_2]=[B_1\otimes B_2]$. In the Brauer group, the class of any quaternion algebra has order two. 
Viceversa, any order-two element of $\mathrm{Br}(\matQ)$ is represented by a quaternion algebra.

For any $\matQ$-quaternion algebra $B$ there are algorithms to compute its \emph{ramification set}, which is a finite set of even cardinality whose elements are prime numbers or $\infty$. The ramification set is a complete invariant of the isomorphism class of $B$ as a quaternion algebra. It is empty $\Leftrightarrow$ $B\simeq M_2(\matQ)$ $\Leftrightarrow$ $[B]=1\in\mathrm{Br}(\matQ)$.

Moreover, for any $\matQ$-quaternion algebras $B_1$ and $B_2$, up to equivalence there exists a unique quaternion algebra $B$ such that $[B_1]\cdot[B_2]=[B]\in\mathrm{Br}(\matQ)$. Hence, it makes sense to talk about the \emph{ramification set} of $[B_1\otimes B_2]$ as the ramification set of the quaternion algebra $B$. This set is the symmetric difference of the ramification sets of $B_1$ and $B_2$.

The commensurability classes of non-uniform arithmetic lattices of the Lie group $\mathrm{Isom}(\matH^4)$ are in bijection with the isomorphism classes of quaternion algebras over $\matQ$, which are classified by their ramification sets.
\subsection*{The algorithm}
Given $N$ \emph{unit} space-like vectors $e_i\in\matR^{1,4}$ ($i=1,\ldots,N$) defining the reflection group $\Gamma$, the following algorithm gives a finite set of prime numbers or $\infty$ which characterises the commensurability class of $\Gamma$.
\begin{enumerate}
\item Compute the Gram matrix $G=(g_{ij})_{ij}$ of $\Gamma$, that is, $g_{ij}=\langle e_i,e_j\rangle$. 
\item Determine all vectors of the form $v_{i_1,\ldots,i_k}=g_{1,i_1}g_{i_1,i_2}\ldots g_{i_{k-1},i_k}e_{i_k}$.
\item The $\matQ$-vector space $V=\mathrm{span}_\matQ\lbrace v_{i_1,\ldots,i_k}\rbrace$ has dimension 5.
Determine a $\matQ$-basis $\mathcal{B}=\lbrace v_1,\ldots,v_5\rbrace$ of $V$.
\item Consider the associated quadratic form $q_G$ over $V$: it is of signature $(4,1)$.
Compute the matrix $Q$ of the form $q_G$ with respect to the basis $\mathcal{B}$. Diagonalise the form, to get a diagonal matrix: $D=\mathrm{diag(a_1,\ldots,a_5)}$, with $a_i\in\matQ^*$.
\item Compute the Hasse invariant $s(q_G)=\left[\bigotimes_{i<j}(a_i,a_j)\right]\in\mathrm{Br}(\matQ)$ and the Witt invariant $c(q_G)=s(q_G)\cdot[(-1,-1)]\in\mathrm{Br}(\matQ)$.
\item Compute the ramification sets of $s(q_G)$ and $c(q_G)$. To this aim, we will often use \cite[Propositions 4.13, 4.15]{GJK}.
\end{enumerate}

We apply the algorithm to discover the following. 
\begin{prop} \label{not:comm:prop}
The 24-cell $P_1$ is commensurable with $P_{\bar t}$ and is \emph{not} commensurable with $P_{t_1}$.
\end{prop}
\begin{proof}
We apply the algorithm to the arithmetic Coxeter polytopes $Q_1$, $Q_{t_1}$ and $Q_{\bar t}$. 
Recall that the vectors $\l A$, $\l L$, $\l M$, $\l N$ are constant, in contrast with 
$\p 0, \m 0, \p 3$, $\m 3, \l G, \l H $ 
that depend on $t$.

We start with $Q_1$ and find:
\begin{align*}
e_{\p0} =\left(1,\tfrac{\sqrt2}{2},\tfrac{\sqrt2}{2},\tfrac{\sqrt2}{2},\tfrac{\sqrt2}{2}\right), &
 \quad e_{\m0}=\left(1,\tfrac{\sqrt2}{2},\tfrac{\sqrt2}{2},\tfrac{\sqrt2}{2},-\tfrac{\sqrt2}{2}\right),\\
e_{\p3} =\left(1,\tfrac{\sqrt2}{2},\tfrac{\sqrt2}{2},-\tfrac{\sqrt2}{2},-\tfrac{\sqrt2}{2}\right), &
\quad e_{\m3}=\left(1,\tfrac{\sqrt2}{2},\tfrac{\sqrt2}{2},-\tfrac{\sqrt2}{2},\tfrac{\sqrt2}{2}\right), \\
e_{\l G}=\left(1,0,0,0,-\sqrt2\right), & \quad e_{\l H}=\left(1,0,0,0,\sqrt2\right), \\
e_{\l A}=\left(1,\sqrt2,0,0,0\right), & \quad e_{\l L}=\left(0,-\tfrac{\sqrt2}2,\tfrac{\sqrt2}2,0,0\right), \\
e_{\l M}=\left(0,0,-\tfrac{\sqrt2}2,\tfrac{\sqrt2}2,0\right), &
\quad e_{\l N}=\left(0,0,-\tfrac{\sqrt2}2,-\tfrac{\sqrt2}2,0\right).
\end{align*}
The Gram matrix is
$$G=
\begin{bmatrix}
			1 & 0 & -1 & 0 & -2 & 0 & 0 & 0 & 0 & -1 \\
			0 & 1 & 0 & -1 & 0 & -2 & 0 & 0 & 0 & -1 \\
			-1 & 0 & 1 & 0 & 0 & -2 & 0 & 0 & -1 & 0 \\
			0 & -1 & 0 & 1 & -2 & 0 & 0 & 0 & -1 & 0 \\
			-2 & 0 & 0 & -2 & 1 & -3 & -1 & 0 & 0 & 0 \\
			0 & -2 & -2 & 0 & -3 & 1 & -1 & 0 & 0 & 0 \\
			0 & 0 & 0 & 0 & -1 & -1 & 1 & -1 & 0 & 0 \\
			0 & 0 & 0 & 0 & 0 & 0 & -1 & 1 & -\frac12 & -\frac12 \\
			0 & 0 & -1 & -1 & 0 & 0 & 0 & -\frac12 & 1 & 0 \\
			-1 & -1 & 0 & 0 & 0 & 0 & 0 & -\frac12 & 0 & 1				  
\end{bmatrix}.
$$
We can choose $\mathcal{B}=\left(e_{\l H},e_{\l A},e_{\l L},e_{\l M},e_{\l N}\right)$, so that in this case $Q$ is just a submatrix of $G$:
$$Q=
\begin{bmatrix} 1 & -1 & 0 & 0 & 0 \\
			 -1 & 1 & -1 & 0 & 0 \\
			0 & -1 & 1 & -\frac12 & -\frac12 \\
			0 & 0 & -\frac12 & 1 & 0 \\
			0 & 0 & -\frac12 & 0 & 1				  
\end{bmatrix}.
$$
A diagonal form is $D=\mathrm{diag}(1,1,-1,1,1)$, thus the Hasse invariant is trivial.

We now turn to $Q_{t_1}$ and find
\begin{align*}
e_{\p0}=\left(\tfrac{\sqrt3}2,\tfrac{\sqrt6}4,\tfrac{\sqrt6}4,\tfrac{\sqrt6}4,\tfrac{\sqrt{10}}4\right), & \quad
e_{\m0}=\left(\tfrac{\sqrt5}2,\tfrac{\sqrt{10}}4,\tfrac{\sqrt{10}}4,\tfrac{\sqrt{10}}4,-\tfrac{\sqrt6}4\right) \\
e_{\p3}=\left(\tfrac{\sqrt3}2,\tfrac{\sqrt6}4,\tfrac{\sqrt6}4,-\tfrac{\sqrt6}4,-\tfrac{\sqrt{10}}4\right), & \quad
e_{\m3}=\left(\tfrac{\sqrt5}2,\tfrac{\sqrt{10}}4,\tfrac{\sqrt{10}}4,-\tfrac{\sqrt{10}}4,\tfrac{\sqrt6}4\right) \\
e_{\l G}=\left(\sqrt5,0,0,0,-\sqrt6\right), & \quad
e_{\l H}=\left(\sqrt5,0,0,0,\sqrt6\right) \\
e_{\l A}=\left(1,\sqrt2,0,0,0\right), & \quad
e_{\l L}=\left(0,-\tfrac{\sqrt2}2,\tfrac{\sqrt2}2,0,0\right) \\
e_{\l M}=\left(0,0,-\tfrac{\sqrt2}2,\tfrac{\sqrt2}2,0\right), & \quad
e_{\l N}=\left(0,0,-\tfrac{\sqrt2}2,-\tfrac{\sqrt2}2,0\right).
\end{align*}
The Gram matrix is
$$G=
\begin{bmatrix}
			1 & 0 & -1 & 0 & -\sqrt{15} & 0 & 0 & 0 & 0 & -\tfrac{\sqrt3}2 \\
			0 & 1 & 0 & -1 & -1 & -4 & 0 & 0 & 0 & -\tfrac{\sqrt5}2 \\
			-1 & 0 & 1 & 0 & 0 & -\sqrt{15} & 0 & 0 & -\tfrac{\sqrt3}2 & 0 \\
			0 & -1 & 0 & 1 & -4 & -1 & 0 & 0 & -\tfrac{\sqrt5}2 & 0 \\
			-\sqrt{15} & -1 & 0 & -4 & 1 & -11 & -\sqrt5 & 0 & 0 & 0 \\
			0 & -4 & -\sqrt{15} & -1 & -11 & 1 & -\sqrt5 & 0 & 0 & 0 \\
			0 & 0 & 0 & 0 & -\sqrt5 & -\sqrt5 & 1 & -1 & 0 & 0 \\
			0 & 0 & 0 & 0 & 0 & 0 & -1 & 1 & -\tfrac12 & -\tfrac12 \\
			0 & 0 & -\tfrac{\sqrt3}2 & -\tfrac{\sqrt5}2 & 0 & 0 & 0 & -\tfrac12 & 1 & 0 \\
			-\tfrac{\sqrt3}2 & -\tfrac{\sqrt5}2 & 0 & 0 & 0 & 0 & 0 & -\tfrac12 & 0 & 1				  
\end{bmatrix}.
$$
We can choose $\mathcal{B}=\left(\sqrt{5}e_{\l H}, e_{\l A}, e_{\l L}, e_{\l M}, e_{\l N}\right)$, to get
$$Q=
\begin{bmatrix} 5 & -5 & 0 & 0 & 0 \\
			 -5 & 1 & -1 & 0 & 0 \\
			0 & -1 & 1 & -\tfrac12 & -\tfrac12 \\
			0 & 0 & -\tfrac12 & 1 & 0 \\
			0 & 0 & -\tfrac12 & 0 & 1				  
\end{bmatrix}.
$$
A diagonal form is $D=\mathrm{diag}(5,-1,3,1,1)$. Thus, the Hasse invariant is 
$$[(5,-1)] \cdot [(5,3)] \cdot [(-1,3)] = [(5,-3)] \cdot [(-1,3)].$$
The ramification points of $[(5,-3)]$ and $[(-1,3)]$ are respectively $\{3,5\}$ and $\{2,3\}$ hence the ramification points of the product are $\{2,5\}$ and hence the element is non-trivial in the Brauer group $\mathrm{Br}(\matQ)$.

We finally look at $Q_{\bar t}$ and find
\begin{align*}
e_{\p0}=\left(\tfrac{\sqrt2}2,\tfrac12,\tfrac12,\tfrac12,\tfrac{\sqrt3}2\right), \quad &
e_{\m0}=\left(\tfrac{\sqrt6}2,\tfrac{\sqrt3}2,\tfrac{\sqrt3}2,\tfrac{\sqrt3}2,-\tfrac12\right) \\
e_{\p3}=\left(\tfrac{\sqrt2}2,\tfrac12,\tfrac12,-\tfrac12,-\tfrac{\sqrt3}2\right), \quad &
e_{\m3}=\left(\tfrac{\sqrt6}2,\tfrac{\sqrt3}2,\tfrac{\sqrt3}2,-\tfrac{\sqrt3}2,\tfrac12\right) \\
e_{\l A}=\left(1,\sqrt2,0,0,0\right), \quad & e_{\l L}=\left(0,-\tfrac{\sqrt2}2,\tfrac{\sqrt2}2,0,0\right) \\
e_{\l M}=\left(0,0,-\tfrac{\sqrt2}2,\tfrac{\sqrt2}2,0\right), \quad &
e_{\l N}=\left(0,0,-\tfrac{\sqrt2}2,-\tfrac{\sqrt2}2,0\right). 
\end{align*}
The Gram matrix is
$$G=
\begin{bmatrix}
			1 & 0 & -1 & 0 & 0 & 0 & 0 & -\tfrac{\sqrt2}2 \\
			0 & 1 & 0 & -1 & 0 & 0 & 0 & -\tfrac{\sqrt6}2 \\
			-1 & 0 & 1 & 0 & 0 & 0 & -\tfrac{\sqrt2}2 & 0 \\
			0 & -1 & 0 & 1 & 0 & 0 & -\tfrac{\sqrt6}2 & 0 \\
			0 & 0 & 0 & 0 & 1 & -1 & 0 & 0 \\
			0 & 0 & 0 & 0 & -1 & 1 & -\tfrac12 & -\tfrac12 \\
			0 & 0 & -\tfrac{\sqrt2}2 & -\tfrac{\sqrt6}2 & 0 & -\tfrac12 & 1 & 0 \\
			-\tfrac{\sqrt2}2 & -\tfrac{\sqrt6}2 & 0 & 0 & 0 & -\tfrac12 & 0 & 1				  
\end{bmatrix}.
$$
We can choose $\mathcal{B}=\left(\sqrt3e_{\m3},\sqrt2e_{\l A},\sqrt2e_{\l L},\sqrt2e_{\l M},\sqrt2e_{\l N}\right)$, to get:
$$Q=
\begin{bmatrix} 3 & 0 & 0 & -3 & 0 \\
			 0 & 2 & -2 & 0 & 0 \\
			0 & -2 & 2 & -1 & -1 \\
			-3 & 0 & -1 & 2 & 0 \\
			0 & 0 & -1 & 0 & 2				  
\end{bmatrix}.
$$
A diagonal form is $D=\mathrm{diag}(3,2,2,-1,2)$. Thus, the Hasse invariant is $$[(3,2)]^3\cdot[(3,-1)]=[(3,2)]\cdot[(3,-1)]=[(3,-2)]=[(3,1-3)]=1\in\mathrm{Br}(\matQ).$$
This completes the proof.
\end{proof}

\end{document}